\documentclass[a4paper,11pt,eqno]{amsart}
\usepackage{amsmath}
\usepackage{amssymb}
\usepackage{amsfonts}
\usepackage{amscd}
\usepackage{amsthm}
\usepackage{mathtools}
\usepackage{mathrsfs}
\usepackage{multirow,bigdelim}
\usepackage{MnSymbol}
\usepackage[bb=boondox,bbscaled=.95,cal=boondoxo]{mathalfa}
\usepackage[numbers]{natbib}
\usepackage{xcolor}
\usepackage{url}
\usepackage{hyperref}
\usepackage{tikz}
\usetikzlibrary{cd}
\usetikzlibrary{decorations.pathmorphing}
\usetikzlibrary{patterns}
\newcommand{\bfA}{{\rm\bf A}}

\newcommand{\CC}{{\rm\bf C}}
\newcommand{\RR}{{\rm\bf R}}
\newcommand{\QQ}{{\rm\bf Q}}
\newcommand{\ZZ}{{\rm\bf Z}}

\newcommand{\GG}{{\rm\bf G}}

\newcommand{\EE}{{\rm\bf E}}

\newcommand{\Adeles}{{\rm\bf A}}

\newcommand{\OO}{\mathcal{O}}

\newcommand{\cF}{{\mathcal{F}}}

\DeclareMathOperator{\Ob}{\mathrm{Ob}}
\DeclareMathOperator{\Ar}{\mathrm{Ar}}
\DeclareMathOperator{\Spec}{\mathrm{Spec}}

\DeclareMathOperator{\Gm}{\mathrm{{\bf G}_m}}
\DeclareMathOperator{\GL}{\mathrm{GL}}

\DeclareMathOperator{\Oo}{\mathrm{O}}
\DeclareMathOperator{\U}{\mathrm{U}}
\DeclareMathOperator{\SO}{\mathrm{SO}}

\DeclareMathOperator{\Aut}{\mathrm{Aut}}
\DeclareMathOperator{\End}{\mathrm{End}}
\DeclareMathOperator{\Hom}{\mathrm{Hom}}

\DeclareMathOperator{\D}{{\mathcal D}}
\DeclareMathOperator{\Pp}{{\mathcal P}}

\DeclareMathOperator{\image}{\mathrm{im}}

\DeclareMathOperator{\Gal}{\mathrm {Gal}}
\DeclareMathOperator{\ind}{\mathrm{ind}}
\DeclareMathOperator{\res}{\mathrm{Res}}

\DeclareMathOperator{\ad}{\mathrm{ad}}

\DeclareMathOperator{\Lie}{\mathrm{Lie}}

\DeclareMathOperator{\rtype}{\mathrm{rtype}}
\DeclareMathOperator{\type}{\mathrm{type}}

\DeclareMathOperator{\cohType}{\mathrm{cohType}}
\DeclareMathOperator{\cohTypes}{\mathrm{cohTypes}}

\newcommand{\lieg}{{\mathfrak {g}}}
\newcommand{\lieh}{{\mathfrak {h}}}
\newcommand{\liek}{{\mathfrak {k}}}
\newcommand{\liel}{{\mathfrak {l}}}

\newcommand{\lies}{{\mathfrak {s}}}
\newcommand{\liet}{{\mathfrak {t}}}

\newcommand{\lieq}{{\mathfrak {q}}}

\newcommand{\lieu}{{\mathfrak {u}}}

\newcommand{\liez}{{\mathfrak {z}}}

\newcommand{\lieso}{{\mathfrak {so}}}

\DeclareMathOperator{\Supp}{Supp}
\DeclareMathOperator{\ann}{ann}

\DeclareMathOperator{\sgn}{sgn}

\DeclareMathOperator{\vol}{\mathrm {vol}}

\theoremstyle{plain}
\newtheorem{theorem}{Theorem}[subsection]
\newtheorem{lemma}[theorem]{Lemma}
\newtheorem{corollary}[theorem]{Corollary}
\newtheorem{proposition}[theorem]{Proposition}

\newtheorem{variant}[theorem]{Variant}

\newtheorem{definition}[theorem]{Definition}
\theoremstyle{remark}
\newtheorem{construction}[theorem]{Construction}
\newtheorem{example}[theorem]{Example}
\newtheorem{remark}[theorem]{Remark}

\bibpunct{[}{]}{,}{n}{}{,}

\numberwithin{equation}{subsection}

\begin{document}
	
	\title[Locally algebraic representations \& integral structures on cohomology]{Locally algebraic representations and\\integral structures on\\the cohomology of arithmetic groups}

        \author{Fabian Januszewski${}^\ast$}\let\thefootnote\relax\footnotetext{${}^\ast$Institut f\"ur Mathematik, Fakult\"at EIM, Paderborn University, Warburger Str.\ 100, 33098 Paderborn, Germany; \texttt{fabian.januszewski@math.uni-paderborn.de}}

	%\date{}
	
	\begin{abstract}
          This paper introduces the notion of locally algebraic representations and corresponding sheaves in the context of the cohomology of arithmetic groups. These representations are of relevance for the study of integral structures and special values of cohomological automorphic representations, as well as corresponding period relations. We introduce and investigate related concepts such as locally algebraic $(\lieg,K)$-modules and cohomological types of automorphic representations. Applying the recently developed theory of tdos and twisted $\mathcal D$-modules over schemes by Hayashi and the author, we establish the existence of canonical global $1/N$-integral structures on spaces of automorphic cusp forms. As an application, we define canonical periods attached to regular algebraic automorphic representations, potentially related to the action of Venkatesh's derived Hecke algebra on cuspidal cohomology.
	\end{abstract}
	
	\maketitle
	\vspace*{-1em}
	{\footnotesize
		\tableofcontents
	}
	
	\section*{Introduction}

        The cohomology of arithmetic groups plays a central role in modern number theory, with applications to the study of automorphic forms, $L$-functions, and Galois representations. Classically, one associates to every rational representation $(V_\CC,\rho_G)$ of a connected reductive group $G$ over $\QQ$ a sheaf $\widetilde{V}_\CC$ of $\CC$-vector spaces on the locally symmetric spaces
        \begin{equation}
        X_G(K_f)=G(\QQ)\backslash G(\bfA)/K_\infty S(\RR)^0 K_f
          \label{eq:introXG}
        \end{equation}
        of level $K_f\subseteq G(\bfA_f)$ associated to $G$. Here $S\subseteq G$ is the maximal $\QQ$-split torus in the center of $G$ and $K_\infty\subseteq G(\RR)$ is maximally compact.

        Arithmetically, the sheaf cohomology
        \begin{equation}
          H^\bullet(X_G(K_f);\widetilde{V}_\CC)
          \label{eq:introsheafcohomology}
        \end{equation}
        can be described in terms of the cohomological automorphic spectrum of $G(\bfA)$ (cf.\ \cite{franke1998,frankeschwermer1998,lischwermer}), and this allows for the study of rationality and integrality properties of automorphic representations and associated $L$-functions \cite{manin1972,manin1976,clozel1990,schmidt1993,hida1994,kazhdanmazurschmidt2000,mahnkopf2005,kastenschmidt2008,raghuramshahidi2008,raghuram2010,raghuramtanabe2011,januszewski2011,januszewski2014,grobnerraghuram2014,januszewski2015,raghuram2016,namikawa2016,balasubramanyamraghuram2017,januszewskirationality,sun2017,januszewskiperiods1,sun2019,jiangetal2019,djr2020,harderraghuram,prasannavenkatesh2021,januszewskiajm,januszewskiperiods2,haranamikawa}, as well as to relate automorphic representation to Galois representations \cite{clozel1990,buzzardgee2011,scholze2015,harrislantaylorthorne,venkatesh2017,venkatesh2018ICM,venkatesh2019,clozel2022,harderbook}.

        The rationality of the representation $V_\CC$ of $G$ imposes parity conditions on the archimedean component of any automorphic representation contributing to the sheaf cohomology \eqref{eq:introsheafcohomology}. From an automorphic perspective, this appears to be a non-canonical choice.

        To illustrate this, consider the symmetric square transfer of Gelbart and Jacquet of a cuspidal cohomological automorphic representation for $\GL(2)$ contributing to cohomology
        \[
        H^\bullet(X_{\GL(2)}(K_f);\widetilde{\bf1}_\CC)
        \]
        with trivial coefficients. This transfer will not contribute to
        \[
        H^\bullet(X_{\GL(3)}(K_f');\widetilde{V}_\CC)
        \]
        for any rational representation $V$ of $\GL(3)$.

        This phenomenon is related to the observation that for a general reductive group $G$, there is no canonical choice of motivic periods or Galois representation associated to an algebraic automorphic representation $\Pi$ of $G$: As Langlands observed, the $L$-functions $L(s,\Pi,r)$ associated to $\Pi$ are naturally parametrized by representations $r\colon{}^LG\to\GL_n(\CC)$ of the $L$-group. Therefore, Galois representations and periods in the motivic sense are only meaningful for a given automorphic representation $\Pi$ after fixing such a representation $r$.

        Conceptually, this raises the question of a notion of cohomological automorphic representation compatible with Langlands' Principle of Functoriality on the one hand, and the motivic notions of Galois representation and periods on the other hand. From this perspective, imposing parity conditions on the archimedean component $\Pi_\infty$ of $\Pi$ via $V_\CC$ appears not desirable.

        On a technical level, a known solution to this problem is to replace the space $X_G(K_f)$ by a finite Galois cover
        \begin{equation}
          X_{G}(K_f)_0:=G(\QQ)\backslash G(\Adeles)/S(\RR)^0K(\RR)^0K_f
          \label{eq:introXG0}
        \end{equation}
        whose covering group $\pi_0(K_\infty,K_f)$ is a factor group of the group $\pi_0(K_\infty)=\pi_0(G(\RR))$ of connected components of $G(\RR)$.

        Then for each character of $\pi_0(K_\infty,K_f)$ we may cut out a subspace of cohomology corresponding to a different parity condition at infinity. We know that $\pi_0(G(\RR))$ is an elementary abelian $2$-group. Therefore, cutting out $\pi_0(K_\infty,K_f)$-eigenspaces commutes with computing cohomology whenever $2$ is invertible on the coefficient system. However, when $2$ is not invertible, the functor of $\pi_0(K_\infty,K_f)$-invariants will no longer be exact, hence its higher left derived functors should be taken into consideration.

        This applies in particular to lattices $V_\OO\subseteq V_E$ over a ring of integers $\OO\subseteq E$ of a number field $E/\QQ$. Then in light of the Leray spectral sequence
        \begin{equation}
          H^p(\pi_0(K_\infty,K_f); H^q(X_G(K_f)_0;\widetilde{V}_\OO))\;\Rightarrow\;R^{p+q}(\Gamma(X_G(K_f)_0;-)^{\pi_0(K_\infty,K_f)})(\widetilde{V}_\OO),
          \label{eq:firstspectralsequence}
        \end{equation}
        the right derived functors of the composition of the global sections functor with the functor of $\pi_0(K_\infty,K_f))$-invariants appear to be conceptually preferable objects of study.

        The right hand side in \eqref{eq:firstspectralsequence} admits the following sheaf theoretic interpretation which is motivated by the monodromy theorem applied to the $\pi_0(K_\infty,K_f)$-cover of $X_G(K_f)$. If we furnish $V_\OO$ with a trivial action of $\pi_0(K_\infty,K_f)$, then we may define a new sheaf $\widetilde{V}_\OO^{\pi_0(K_\infty,K_f)}$ on $X_G(K_f)$ whose sections are in a suitable sense $\pi_0(K_\infty,K_f)$-invariant. The sheaf cohomology of $\widetilde{V}_\OO^{\pi_0(K_\infty,K_f)}$ agrees with the abutment of the spectral sequence \eqref{eq:firstspectralsequence} and it turns out that the construction of $\widetilde{V}_\OO^{\pi_0(K_\infty,K_f)}$ is meaningful for any action of $\pi_0(K_\infty,K_f)$ on $V_\OO$ which commutes with a given scheme theoretic action of $G$ on $V_\OO$.

        This observation naturally leads to the notion of \lq{}locally algebraic representation\rq{} of $G$ we introduce in Definition \ref{def:locallyalgebraicrepresentation}, taking into account an additional $\pi_0(K_\infty)$-action. To any locally algebraic representation $V$ of $G$ we associate a sheaf $\widetilde{V}$ (cf.\ Definition \ref{def:locallyalgebraicsheaf}). The corresponding sheaf cohomology associated to locally algebraic representations contains {\em all} cohomological automorphic representations $\Pi$ of $G$ --- the parity condition on the archimedean component $\Pi_\infty$ depending on the auxiliary action of $\pi_0(K_\infty)$ on $V_E$, which may be chosen appropriately.

        While the spectral sequence \eqref{eq:firstspectralsequence} degenerates away from $2$, the concept of locally algebraic representation and the associated sheaves seem conceptually preferable even in characteristic $0$. From the perspective of orbifolds in the language of proper \'etale Lie groupoids, locally algebraic representations parametrize a natural class of locally constant sheaves (cf.\ Definition \ref{def:locallyalgebraicorbifoldsheaf}).
        
        The notion of locally algebraic representations and their associated sheaves allows to treat all cohomological automorphic representations of $G$ at once without giving a preference to a particular parity condition which would be imposed by sticking to rational representations, as is illustrated by the symmetric square transfer example above. Furthermore, cohomology with coefficients in locally algebraic representations preserves multiplicity one properties in bottom and top degree of the cuspidal range.

        Moreover, the language of locally algebraic representations overcomes the following technical issue pertaining to parity conditions at infinity: The scheme theoretic model of $\pi_0(K_\infty)$ induced by a choice of $\QQ$-rational model $K$ of the maximal compact subgroup $K_\infty$ is finite \'etale but not constant in general. Therefore, there is no corresponding category of $(\lieg,K)$-modules over $\QQ$ taking into account all parity conditions at once (cf.\ Definition \ref{def:locallyalgebraicgKmodule}). A fortiori there are no corresponding categories of $(\lieg,K)$-modules over $\ZZ$ or $\ZZ[1/2]$. Therefore, in order to have a corresponding global theory of rational and integral structures on arbitrary algebraic automorphic representations taking into account all parity conditions, the notion of locally algebraic representations and a corresponding notion of \lq{}locally algebraic $(\lieg,K)$-module\rq{} appear indispensable.

        At first sight, {\em locally algebraic $(\lieg,K)$-modules}, defined as $(\lieg,K)$-modules with an auxiliary commuting action of the constant group scheme $\pi_\circ^G$ associated to $\pi_0(K_\infty)$, have some counter-intuitive properties: To each locally algebraic $(\lieg,K)$-module $X$ over $\CC$ we may associate a classical complex $(\lieg_\CC,K_\infty)$-module $X_\CC$ by letting $K_\infty$ act on $X_\CC$ via the \lq{}diagonal embedding\rq{}
        \[
        K_\infty\to K(\RR)\times\pi_\circ^G(\RR),
        \]
        cf.\ Definition \ref{def:locallyalgebraiccomplexmodule}. Then the functor $X\mapsto X_\CC$ is not fully faithful. What may seem like a redundancy in the definition of locally algebraic $(\lieg,K)$-module, finds its justification in the notion of a \lq{}cohomological type\rq{} for $G$, which we introduce in Definition \ref{def:cohomologicaltypesofG}.

        A {\em cohomological type} consists of an equivalence class of isomorphism classes of pairs $(A,V)$ where $A$ is a locally algebraic cohomologically induced standard module and $V$ is a corresponding locally algebraic representation with the property that the relative Lie algebra cohomology of $A\otimes V$ is well defined and non-zero.

        In Definition \ref{def:cohomologicaltypeofrepresentation} we associate to every cohomological automorphic representation $\Pi$ a cohomological type $\cohType(\Pi)$. The cohomological type encodes all irreducible locally algebraic representations $V$ for which $\Pi$ contributes to \eqref{eq:introsheafcohomology} and as such is the natural basepoint-free set parametrizing cohomologically defined periods attached to $\Pi$.

        Using Vogan--Zuckerman theory \cite{voganzuckerman1984}, cohomological types are amenable to being classified. The important case of tempered cohomological types is intimately related to global integral structures on spaces of automorphic cusp forms.

        As an application, we deduce from the half-integral models of certain cohomologically induced standard modules constructed by Hayashi and the author in \cite{hayashijanuszewski} locally algebraic half-integral models of locally algebraic cohomologically induced standard modules. Refining the author's arguments in \cite{januszewskirationality}, we prove the existence of canonical half-integral structures on automorphic representations and global spaces of cusp forms, which are naturally parametrized by representatives of tempered cohomological types.

        We construct in section \ref{sec:integralmodels} global $1/N$-integral structures on cuspidal automorphic representations, cf.\ Theorems \ref{thm:globalarithmeticmodels} and \ref{thm:gln} for $N=2d_F$, $d_F$ denoting the absolute discriminant of the totally real or CM field $F/\QQ$. In Theorems \ref{thm:globalhalfintegralstructures} and \ref{thm:globalhalfintegralstructuresKf} we prove the existence of a canonical global half-integral structure on certain spaces of automorphic cusp forms for $\GL(n)$ over a totally real or CM-field. These results generalize the rational strucures constructed by the author in \cite{januszewskirationality} and as such also extend to arbitrary reductive groups, with the same limitations as in loc.\ cit. (cf.\ Remark \ref{rmk:globalgeneralreductiveG}). As in the case of rational structures, the results are sharpest for $\GL(n)$ over a totally real or a CM field.

        We remark that localizing the $1/N$-integral structures constructed in Theorem \ref{thm:globalhalfintegralstructures} and Theorem \ref{thm:globalhalfintegralstructuresKf} to rational structures, the results are a natural extension of the results obtained in \cite{januszewskirationality} by taking into account all possible parity conditions at $\infty$ on the sheaf level over their minimal fields of definition.

        With the $1/N$-integral structures constructed in Theorem \ref{thm:globalhalfintegralstructuresKf} at hand, we associate in Theorem \ref{thm:periods} to every cohomological cuspidal automorphic representation $\Pi$ of $\GL(n)$ over a totally real or a CM field and to every element $\alpha\in\cohType(\Pi)$ in the cohomological type of $\Pi$ and every finite place $v$ of $\QQ(\Pi)$ not dividing $N$ a canonical set of periods, unique up to units in $\OO_v^\times$, or more generally double cosets for $\GL_r(\OO_v)$.

        The existence of such canonical periods raises the natural question how they are related to Venkatesh's Conjecture on cuspidal cohomology as a module over the derived Hecke algebra, cf.\ Remark \ref{rmk:venkatesh}.

        In Theorem \ref{thm:AEdecomposition} we relate cohomological types, admissible twists of locally algebraic representations and rationality properties of restrictions of locally algebraic $(\lieg,K)$-modules with a view toward period relations for automorphic $L$-functions. We explicate the situation for products of $\GL(n)$ in Theorem \ref{thm:glnrationalreducibility}, extending previous results of the author obtained in \cite{januszewskiperiods1} to the general locally algebraic case.

        The motivation for such considerations are Theorem \ref{thm:rationalperiodrelations} and Theorem \ref{thm:integralperiodrelations} which provide a general archimedean period calculation, which may serve as a blueprint for proving archimedean period relations for special values of $L$-functions associated to cohomological automorphic representations of a general connected reductive group $G$.

        Finally, in section \ref{sec:applicationstoautomorphicL} we indicate known cases, where such methods apply and give a detailed account in the case of Hilbert modular cusp forms.

        \subsection*{Outline of the paper}

        Section \ref{sec:locallyalgebraicrepresentations} is devoted to the study of locally algebraic representations and related concepts such as lattices in locally algebraic representations (Definition \ref{def:integralmodelsoflocallyalgebraicrepresentations}), associated complex representations of the Lie group $G(\RR)$ (Definition \ref{def:locallyalgebraiccomplex}), and sheaves on the spaces $X_G(K)$ and the orbifolds $\mathcal X_G(K_f)$ (Definitions \ref{def:locallyalgebraicorbifoldsheaf} and \ref{def:locallyalgebraicsheaf}).

        The various sheaves on the orbifolds and the underlying orbit spaces and their cohomologies are related by pushforward functors and corresponding Leray spectral sequences. The language of orbifolds appears to be the most appropriate setting to study the cohomology of arithmetic groups without unnecessary restrictions.

        We also briefly discuss a more general notion of locally algebraic representations taking non-archimedean aspects into account in Remark \ref{rmk:generallocallyalgebraicrepresentations}.

        Section \ref{sec:locallyalgebraicgKmodules} is devoted to locally algebraic $(\lieg,K)$-modules and their basic properties. A specialized notion is that of a {\em locally algebraic cohomologically induced standard module} which we introduce and investigate in subsections \ref{sec:locallyalgebraiccohomologicallyinducedstandardmodules}, \ref{sec:admissibleweights} and \ref{sec:relativeliealgebracohomology}. This notion is fundamental for the notion of {\em cohomological type} introduced in the following section.

        Section \ref{sec:cohomologicalrepresentations} puts the previously defined notions and established properties into the context of cohomological automorphic representations.

        We begin with an overview of the fundamental properties of the cohomology of arithmetic groups with coefficients in locally algebraic representations both from an orbifold- and a classical topological viewpoint in subsection \ref{sec:cohomologyofarithmeticgroups}.

        In subsection \ref{sec:cohomologicaltypes} we introduce the notion of {\em cohomological type}. In subsection \ref{sec:cycleandsupportofcohomologicaltype} we introduce associated notions of {\em cycle} and {\em support} of a cohomological type, and subsection \ref{sec:temperedcohomologicaltypes} is dedicated to the class of tempered cohomological types. Finally, in section \ref{sec:automorphiccohomologicaltype} we put the previous discussion into the context of cohomological automorphic representations.

        Section \ref{sec:integralmodels} is devoted to global $1/N$-integral models of automorphic representations with origin in the cohomology of arithmetic groups with coefficients in a lattice in a locally algebraic representation. We obtain results for a general reductive group $G$ (or rather classical groups) in subsection \ref{sec:automorphicintegralstructures}.  For $\GL(n)$ we obtain sharper results in subsection \ref{sec:integralstructuresforgln}. In particular, Theorem \ref{thm:cohomologicaltypesofGLnoverQ} contains a classification of tempered cohomological types for $\GL(n)$ over totally real and CM fields.

        While section \ref{sec:integralmodels} was focused on individual automorphic representations, we extend the existence results for $1/N$-integral structures to global spaces of cusp forms for $\GL(n)$ in subsection \ref{sec:canonicalintegralstructures} of section \ref{sec:canonicalintegralstructuresandperiods}.

        Subsection \ref{sec:canonicalperiods} contains our construction of canonical periods in the case of $\GL(n)$.

        We provide a blueprint of general archimedean period relations in the language of locally algebraic representations in subsection \ref{sec:specialvalues}, where we also briefly discuss some examples.

        \subsection*{Acknowledgements}

        The author is indepted to his coauthor from \cite{hayashijanuszewski} for many very valuable remarks on preliminary drafts of this paper. The author thanks Mladen Dimitrov for enlightening discussions. Very special thanks go to G\"unter Harder, who devoted a lot of his time studying integral structures on the cohomology of arithmetic groups. The author will miss the stimulating discussions he had with him. The author acknowledges support by the Deutsche Forschungsgemeinschaft (DFG, German Research Foundation) -- SFB-TRR 358/1 2023 -- 491392403 and by the Max-Planck-Institute for Mathematics in Bonn.

        \section*{Notation}
        
        We write $\bfA=\RR\times\bfA_f$ for the topological ring of ad\`eles of $\QQ$, $\bfA_f$ denoting the ring of finite ad\`eles. Adelization of linear algebraic groups is understood in the sense of A.\ Weil. In particular, if $G$ is a linear algebraic group over $\QQ$, then $G(\Adeles)$ and $G(\Adeles_f)$ denote the adelization and the finite adelization of $G$ over $\QQ$. We have an isomorphism $G(\Adeles)\cong G(\RR)\times G(\Adeles_f)$ of locally compact topological groups.

        We write $K^0$ for the connected component of the identity of a topological group $K$ and $\pi_0(K)=K/K^0$ for the group of connected components.

        If $G$ is an algebraic group, we write $G^\circ$ for the connected component of the identity. We usually assume that $G$ fits into an exact sequence
        \[
        1\to G^\circ\to G\to\pi_\circ(G)\to 1
        \]
        of group schemes (\'etale or $fppf$-sheaves) where $\pi_\circ(G)$ is assumed to be finite \'etale.

        If $X$ is a scheme over a commutative ring $A$, we write $X_A$ if we want to emphasize that we consider $X$ as a scheme over $A$. In particular, if $B$ is a (commutative) $A$-algebra, we write $X_B$ for the base change of $X_A$ to $B$.

        All groupoids we consider are small. For a groupoid $\mathcal G$, we denote by $\Ob\mathcal G$ the set of objects and by $\Ar\mathcal G$ the set of all arrows.

        \section{Locally algebraic representations}\label{sec:locallyalgebraicrepresentations}

        \subsection{Arithmetic manifolds}

        Write $S\subseteq G$ for the maximal $\QQ$-split torus in the center of $G$. For any compact open subgroup $K_f\subseteq G(\bfA_f)$ we consider as before in \eqref{eq:introXG} and \eqref{eq:introXG0} the spaces
	\begin{align}
	X_G(K_f)\;&=\;G(\QQ)\backslash G(\bfA)/S(\RR)^0 K_\infty K_f\;=\;G(\QQ)\backslash \left(G(\RR)/S(\RR)^0 K(\RR) \times G(\bfA_f)/K_f\right),\\
	X_G(K_f)_0\;&=\;G(\QQ)\backslash G(\bfA)/S(\RR)^0 K(\RR)^0 K_f.
	\end{align}
	These spaces are of finite volume and if $K_f$ is sufficiently small, i.\,e.\ if for all $g\in G(\bfA_f)$ the arithmetic groups
	\begin{equation}
		\Gamma_{K_f,g}:=G(\QQ)\cap gK_fg^{-1}
		\label{eq:arithmeticgroupforKf}
	\end{equation}
	are torsion free, then $X_G(K_f)$ and $X_G(K_f)_0$ are manifolds: $X_G(K_f)$ and $X_G(K_f)_0$ decompose into the disjoint union of finitely many locally symmetric spaces of the form
	\[
	\Gamma_{K_f,g}\backslash G(\RR)/S(\RR)^0 K(\RR),\quad\text{and}\quad
        \Gamma_{K_f,g}\backslash G(\RR)/S(\RR)^0 K(\RR)^0,
	\]
        for $g\in G(\bfA_f)$ running through a system of representatives for the finitely many connected components of the space $X_G(K_f)$ (resp.\ $X_G(K_f)_0$).
        %and
        %\begin{equation}
        %  \Gamma_{K_f,g}^+:=\Gamma_{K_f,g}\cap G(\RR)^0.
        %  \label{eq:GammaKfplus}
        %\end{equation}

        In particular, we have for any sheaf $F$ on $X_G(K_f)$ and any sheaf $F_0$ on $X_G(K_f)_0$ isomorphisms
        \begin{align}
          H^\bullet(X_G(K_f);F)&=\bigoplus_{g}H^\bullet(\Gamma_{K_f,g}\backslash G(\RR)/S(\RR)^0 K(\RR);F|_{\Gamma_{K_f,g}\backslash G(\RR)/S(\RR)^0 K(\RR)}),\label{eq:cohomologydirectsum}\\
          H^\bullet(X_G(K_f)_0;F_0)&=\bigoplus_{g}H^\bullet(\Gamma_{K_f,g}\backslash G(\RR)^0/S(\RR)^0 K(\RR)_0;F_0|_{\Gamma_{K_f,g}\backslash G(\RR)/S(\RR)^0 K(\RR)^0}),\label{eq:cohomologydirectsum0}
        \end{align}
        which are compatible with the canonical projection $p_0\colon X_G(K_f)_0\to X_G(K_f)$ if $F=p_{0,\ast}F_0$. In fact, $p_0$ is a Galois cover and we refer to Remark \ref{rmk:galoiscover} below for a more detailed discussion.
        
        \subsection{Arithmetic orbifolds}
        
        Regardless of whether $X_G(K_f)$ may be considered as a manifold, we consider an orbifold $\mathscr X_G(K_f)$ in the following sense. We understand the notion of orbifold as a mild generalization of the notion defined in \cite{satake1956}, where orbifolds are refered to as \lq{}$V$-manifolds\rq{}, and where isotropy groups of points act effectively on the respective orbifold charts. We do not require that isotropy groups act effectively on orbifold charts, i.\,e. we do not assume that our orbifolds are reduced. We do not require orbifolds to be connected. For an introduction to orbifolds from the point of view of Lie groupoids we refer to \cite{moerdijkoverview} and Chapter 5 in \cite{moerdijkmrcun}. This perspective is relevant in our discussion of sheaves below.
        
        We define $\mathscr X_G(K_f)$ as the orbifold quotient of the space
        \[
        G(\QQ)\backslash G(\bfA)/S(\RR)^0 K_f
        \]
        by the right action of the compact but not necessarily connected Lie group $K_\infty$. Then $X_G(K_f)$ is the orbit space $|\mathscr X_G(K_f)|$ of $\mathscr X_G(K_f)$ and we have a canonical covering map
        \begin{equation}
          p_{K_f}\colon \mathscr X_G(K_f)\to X_G(K_f).
          \label{eq:orbifoldtomanifold}
        \end{equation}
        of orbifolds, where we consider $X_G(K_f)$ as an orbifold with trivial isotropy groups.

        We may think of $\mathscr X_G(K_f)$ as the topological space $X_G(K_f)$ together with an orbifold atlas consisting of orbifold charts $(\widetilde{U},G,\phi)$ with $\widetilde{U}\subseteq\RR^{\dim X_G(K_f)}$ open, $G$ a finite group acting not necessarily effectively via diffeomorphisms on $\widetilde{U}$ from the right, and $\phi\colon\widetilde{U}\to X_G(K_f)$ is continuous and $G$-invariant (i.\,e.\ $\varphi\circ g=\varphi$ for all $g\in G$) with the additional property that its image $U:=\varphi(\widetilde{U})\subseteq X_G(K_f)$ is open and via $\varphi$ homeomorphic to $\widetilde{U}/G$.

        Given an orbifold atlas, the notion of sheaf on $\mathscr X_G(K_f)$ is straightforward and gives rise to a topos. A sheaf $\cF$ on $\mathscr X_G(K_f)$ is locally, on a chart $(\widetilde{U},G,\phi)$, given by a sheaf $\cF|_{\widetilde{U}}$ of abelian groups on the topological space $\widetilde{U}$ and for every $g\in G$ an isomorphism
        \[
        \cF(\lambda_g)\colon \cF|_{\widetilde{U}}\to\lambda_g^\ast\left(\cF|_{\widetilde{U}}\right)
        \]
        where $\lambda_g\colon\widetilde{U}\to\widetilde{U}$ denotes the action of $g$ with the property that for all $g,h\in G$ the diagram
        \[
        \begin{CD}
          \cF|_{\widetilde{U}}@>{\cF(\lambda_g)}>>\lambda_g^\ast\left(\cF|_{\widetilde{U}}\right)\\
          @V{\cF(\lambda_{gh})}VV @VV{\lambda_g^\ast\cF(\lambda_h)}V\\
          (\lambda_{gh})^\ast\left(\cF|_{\widetilde{U}}\right)@=\lambda_g^\ast\lambda_h^\ast\left(\cF|_{\widetilde{U}}\right)\\
        \end{CD}
        \]
        commutes. In other words $\cF|_{\widetilde{U}}$ is a $G$-equivariant sheaf on $\widetilde{U}$.

        The category of sheaves of abelian group on $\mathscr X_G(K_f)$ is abelian and has enough injectives. Sheaf cohomology is the right derived functor of the functor of sections, which are defined as collections of sections on orbifold charts compatible with all chart embeddings and the equivariant structure.

        The Grothendieck spectral sequence for the composition $\Gamma(\mathscr X_G(K_f);-)=\Gamma(X_G(K_f);p_\ast(-))$ is the Leray spectral sequence
        \begin{equation}
          E_2^{p,q}=H^p(X_G(K_f);R^qp_\ast(\cF))\;\Longrightarrow\;H^{p+q}(\mathscr X_G(K_f);\cF).
          \label{eq:lerayspectralsequence}
        \end{equation}
        For any $x\in\mathscr X_G(K_f)$, the isotropy group $G_x$ at $x$ naturally acts on the stalk $\cF_x$ and we have a natural isomorphism
        \begin{equation}
          \left(R^qp_\ast(\cF)\right)_{p(x)}=H^q(G_x;\cF_x).
          \label{eq:orbifoldstalks}
        \end{equation}
        
        \begin{remark}\label{rmk:leraydegeneration}
          The identity \eqref{eq:orbifoldstalks} shows that whenever the orders of the isotropy groups of $\mathscr X_G(K_f)$ are invertible on the corresponding stalks, the Leray spectral sequence \eqref{eq:lerayspectralsequence} degenerates to an isomorphism
          \begin{equation}
            H^\bullet(X_G(K_f);p_\ast(\cF))\;\cong\;H^\bullet(\mathscr X_G(K_f);\cF).
          \end{equation}
          This is always the case if $\cF$ is a sheaf of vector spaces over a field of characteristic $0$.
        \end{remark}

	\subsection{Locally algebraic representations}\label{subsec:locallyalgebraicrepresentations}
	
	\begin{definition}[Locally algebraic representation]\label{def:locallyalgebraicrepresentation}
	  A {\em locally algebraic representation $V$ of $G$} over a field $E$ of characteristic $0$ is a triple $V=(V_{E}, \rho_G,\rho_{\pi_\circ^G})$ consisting of a finite-dimensional $E$-vector space $V_E$, a rational action $\rho_G:G_E\times V_E\to V_E$ of $G$ on $V_E$, and a scheme theoretic action $\rho_{\pi_\circ^G}$ of the constant group scheme $\pi_\circ^G$ on $V_E$, subject to the condition that the actions of $G$ and $\pi_\circ^G$ on $V_E$ commute. We call $V$ {\em algebraic} or {\em rational} if $\rho_{\pi_\circ^G}$ is the trivial action.

          We define morphisms of locally algebraic representations as $E$-linear maps commuting with all given actions.
	\end{definition}

        \begin{example}\label{ex:GLnoverQ}
          For $G=\GL(n)$ over $\QQ$, $\pi_0(\GL_n(\RR))$ is a finite group of order $2$. Therefore there are two characters (one-dimensional $\QQ$-rational representations) of the constant group scheme $\pi_\circ^G$: The trivial character ${\bf1}$ and the sign character ${\rm sgn}$. For representatives $g_\infty\in \GL_n(\RR)$, we have the explicit formula
          \[
            {\rm sgn}(g_\infty)=\frac{\det g_\infty}{|\det g_\infty|}.
          \]
          We conclude that we have for each absolutely irreducible rational representation $(V_\QQ,\rho_{\GL(n)})$ of $\GL(n)$ over $\QQ$ two non-isomorphic locally algebraic representations:
          \[
            V=(V_\QQ,\rho_{\GL(n)},{\bf1})\quad\text{and}\quad V'=(V_\QQ,\rho_{\GL(n)},{\rm sgn}).
          \]
          We will see in Proposition \ref{prop:absolutelyirreduciblelocallyalgebraic} below that all absolutely irreducible rational representations of $\GL(n)/\QQ$ are of this form.
        \end{example}

        The notion of locally algebraic representation of $G$ naturally extends to $K_f$ as follows.
        \begin{definition}[General locally algebraic representations]
          Fix a compact open subgroup $K_f\subseteq G(\bfA_f)$. Then a {\em general locally algebraic representation of $G$ at level $K_f$} consists of a quadruple $(V_E,\rho_G,\rho_{\pi_\circ^G},\rho_{K_f})$ where $(V_E,\rho_G)$ is a rational representation of $G$, $\rho_{\pi_\circ^G}$ is a representation of the group scheme $\pi_\circ^G$ on $V_E$, and $\rho_{K_f}\colon K_f\to\Aut_E(V_E)$ is a representation of $K_f$ with the following property:

          There exists a finite sets $S$ of primes such that $K_f=K_S\times K_f^S$ where $K_S\subseteq\prod_{p\in S}\QQ_p$ is an open subgroup and $K_f^S$ is a compact open subgroup of $\QQ\otimes_\ZZ\prod_{p\not\in S}\ZZ_p$. Moreover, $\rho_{K_f}$ is trivial on $K_f^S$ and there is an open subgroup $K_S'\subseteq K_S$ with the property that $\rho_{K_f}|_{K_S'}$ acts on $V_E$ via the restriction of a rational representation $\rho_S\colon G\to\Aut_E(V_E)$ to $K_S'$. Finally, we require that all three actions $\rho_G,\rho_{\pi_\circ^G},$ and $\rho_{K_f}$ commute pairwise.
        \end{definition}

        \begin{remark}[Non-archimedean locally algebraic representations]\label{rmk:generallocallyalgebraicrepresentations}
          This general notion of locally algebraic representation is well suited for $p$-adic coefficient systems, i.\,e.\ the case $S=\{p\}$ and $E$ a $p$-adic field, seems to be the most relevant one from a practical perspective. There is a well-known recipe to translate a locally algebraic representation of the form $(V_E,\rho_G,\rho_{\pi_\circ^G},{\bf1})$ into another of the form $(V_E,{\bf1},\rho_{\pi_\circ^G},\rho_{K_f})$, preserving the respective associated sheaves on $\mathscr X_G(K_f)$ and $X_G(K_f)$ up to natural isomorphism.

          However, there are locally algebraic representations $(V_E,{\bf1},\rho_{\pi_\circ^G},\rho_{K_f})$ in the general sense whose sheaves do not arise from locally algebraic representations $(V_E,\rho_G,\rho_{\pi_\circ^G},{\bf1})$ in the strict sense: Take $G=\GL(1)$ and let $K_f$ act via an even continuous finite order character, while keeping all other actions trivial. For an odd character one may choose $\rho_{\pi_\circ^G}$ as the sign character to obtain a non-trivial sheaf on $\mathscr X_G(K_f)$ (resp.\ $X_G(K_f)$) in this case.

          That being said, we remark that a significant part (albeit not all) of our discussion below translates with no or only minor modifications to this more general notion of locally algebraic representation. Since the necessary modifications to treat the general case are straightforward, the locally algebraic representations considered below will all be the ones defined in Definition \ref{def:locallyalgebraicrepresentation}.
        \end{remark}

        \begin{definition}[Integral models of locally algebraic representations]\label{def:integralmodelsoflocallyalgebraicrepresentations}
          If $G$ admits a (smooth) model over a subring $\OO\subseteq E$ with fraction field $E$ and if $V_\OO\subseteq V_E$ is an $\OO$-lattice, then a representation $(V_\OO,\rho_{G\times\pi_\circ^G})$ of the group scheme $G\times\pi_\circ^G$ over $\OO$ on $V_\OO$ will be called an {\em $\OO$-model} of or by abuse of language a {\em $K_f$-stable $\OO$-lattice} in a locally algebraic representation $V=(V_E,\rho_G,\rho_{\pi_\circ^G})$, if the inclusion $V_\OO\to V_E$ induces an isomorphism $E\otimes_\OO V_\OO\cong V_E$ of locally algebraic representations.
        \end{definition}

	\begin{remark}\label{rmk:locallyalgebraicviaconstantgroupscheme}
          Given a finite-dimensional vector space $V_E$, the following data are equivalent:
          \begin{itemize}
            \item[(i)] A locally algebraic rational representation $V=(V_E,\rho_G,\rho_{\pi_\circ^G})$ over $E$,
            \item[(ii)] A representation $V=(V_E,\rho_{G\times\pi_\circ^G})$ of the group scheme $G\times\pi_\circ^G$ over $E$,
            \item[(iii)] A rational representation $(V_E,\rho_G)$ of $G$ over $E$ together with an $E$-rational representation $\pi_0(K_\infty)\to\Aut_E(V_E)$ of the abstract group $\pi_0(K_\infty)$ such that the resulting actions of $G$ and $\pi_0(K_\infty)$ on $V_E$ commute.
          \end{itemize}
          We will switch freely between these three incarnations of locally algebraic representations.
	\end{remark}

        \begin{proposition}
          The category $\underline{\text{LocAlgRep}}_E(G)$ of locally algebraic representations $V$ of $G$ over $E$ is a rigid $E$-linear tensor category admitting a fiber functor $\mathcal F\colon \underline{\text{LocAlgRep}}_E(G)\to\underline{Vec}_E$ into $E$-vector spaces. In particular, $\underline{\text{LocAlgRep}}_E(G)$ is Tannakian and its fundamental group is $G\times\pi_\circ^G$.
        \end{proposition}

        \begin{proof}
          This is clear by Remark \ref{rmk:locallyalgebraicviaconstantgroupscheme}.
        \end{proof}

        \begin{proposition}[{Matsumoto \cite[Theor\`eme 1]{matsumoto}}]\label{prop:matsumoto}
          Let $G$ denote a connected reductive group over $\QQ$. Then the component group $\pi_0(G(\RR))$ is an elementary abelian $2$-group.
        \end{proposition}
	
	\begin{proposition}\label{prop:absolutelyirreduciblelocallyalgebraic}
	  Let $V=(V_E,\rho_G,\rho_{\pi_\circ^G})$ denote an irreducible locally algebraic representation of $G$ over a field $E$ of characteristic $0$. Then
          \begin{itemize}
          \item[(a)] $V$ is (absolutely) irreducible if and only if the rational representation $(V_E,\rho_G)$ is (absolutely) irreducible.
          \item[(b)] If $V$ is irreducible, $\pi_0(K_\infty)$ acts on $V_E$ via $\rho_{\pi_\circ^G}$ as scalars in $\{\pm1\}$.
          \end{itemize}
	\end{proposition}

        \begin{proof}
          We proof (b) first. Let $V$ be irreducible and let $k\in\pi_0(K_\infty)$ denote an arbitrary element. By Proposition \ref{prop:matsumoto}, $\pi_0(K_\infty)$ is abelian and therefore $k$ commutes not only with the action of $G$ von $V_E$, but also with the action of $\pi_0(K_\infty)$. Again by Proposition \ref{prop:matsumoto}, the minimal polynomial of $\rho_{\pi_\circ^G}(k)$ on $V$ is a divisor of $X^2-1$ and hence admits a root in $E$. The standard argument of Schur's Lemma then shows that $\rho_{\pi_\circ^G}(k)=\pm{\bf1}_V$. This proves (b). Statement (a) is an immediate consequence of (b).
        \end{proof}
        
        \begin{definition}[Field of rationality of a locally algebraic representation]\label{def:fieldofrationality}
          Let $V=(V_E,\rho_G,\rho_{\pi_\circ^G})$ denote a locally algebraic representation of $G$ over $E\subseteq\CC$. Put for any subfield $E'\subseteq E$
          \[
          G_{V}\;:=\;\{\sigma\in\Aut(\CC/E')\mid \CC\otimes_{\sigma,E}V_E\cong \CC\otimes_E V_E\;\text{as loc.\ alg.\ representations}\}.
          \]
          We call $E'(V):=\CC^{G_V}$ the {\em field of rationality} of $V$ over $E$.
        \end{definition}

        \begin{proposition}\label{prop:locallyalgebraicfieldofrationality}
          The field of rationality $E'(V)$ of an irreducible locally algebraic representation in the sense of Definition \ref{def:fieldofrationality} agrees with the field of rationality $E'(V_E,\rho_G)$ of the underlying rational representation $(V_E,\rho_G)$ of $G$.
        \end{proposition}

        \begin{proof}
          We have $E'(V_E,\rho_G)\subseteq E'(V)$ essentially by definition. By Proposition \ref{prop:absolutelyirreduciblelocallyalgebraic} (b) the actions of $\pi_0(K_\infty)$ on $\CC\otimes_{\sigma,E}V_E$ and $\CC\otimes_E V_E$ agree for any $\sigma\in\Aut(\CC/\QQ)$. Hence $E'(V)\subseteq E'(V_E,\rho_G)$.
        \end{proof}
        
        \begin{proposition}[Field of definition of a locally algebraic representation]\label{prop:locallyalgebraicfieldofdefinition}
          Let $V=(V_E,\rho_G,\rho_{\pi_\circ^G})$ for $E\subseteq \CC$ denote an irreducible locally algebraic representation of $G$. Then for any subfield $\QQ(V)\subseteq E_0\subseteq E$ the following are equivalent:
          \begin{itemize}
          \item[(a)] $V$ is defined over $E_0$, i.\,e.\ there exists a locally algebraic representation $V_0=(V_{E_0},\rho_G^0,\rho_{\pi_\circ^G}^0)$ and an isomorphism $E\otimes_{E_0}V_{E_0}\cong V_E$ of locally algebraic representations.
          \item[(b)] The rational representation $(V_E,\rho_G)$ is defined over $E_0$.
          \end{itemize}
          If $V$ is absolutely irreducible, then in case (a) the model $V_0$ of $V$ over $E_0$ is unique up to isomorphism if it exists.
        \end{proposition}

        \begin{proof}
          If $(V_E,\rho_G)$ admits a model over $E_0$, then by Proposition \ref{prop:absolutelyirreduciblelocallyalgebraic} (b), the action $\rho_{\pi_\circ^G}$ descends to $E_0$ as well. This shows the equivalence of (a) and (b).

          The uniqueness of models for absolutely irreducible representations is a consequence of Schur's Lemma combined with
          \[
          H^1(\Gal(\overline{E}/E_0);\Aut(\overline{E}\otimes_{E_0}V_{E_0}))=
          H^1(\Gal(\overline{E}/E_0);\overline{E}^\times)=0
          \]
          by Hilbert 90 for any algebraic closure $\overline{E}$ of $E$.
        \end{proof}

        \begin{corollary}\label{cor:quasisplitrationality}
          If $G$ is quasi-split, then any locally algebraic representation $V$ of $G$ is defined over its field of rationality $\QQ(V)$.
        \end{corollary}

        \begin{proof}
          By Propositions \ref{prop:locallyalgebraicfieldofrationality} and \ref{prop:locallyalgebraicfieldofdefinition} it suffices to remark that for quasi-split $G$ every rational representation $(V_E,\rho_G)$ of $G$ is defined over its field of rationality (cf.~\cite{boreltits1965}).
        \end{proof}

        \begin{theorem}[Classification of irreducible locally algebraic representations]\label{thm:irreduciblelocallyalgebraicrepresentations}
          Fix a subfield $E\subseteq\CC$.
          \begin{itemize}
          \item[(a)] Any irreducible locally algebraic representation $\overline{V}_0$ of $G$ over $\CC$ admits a model over a finite extension $E_0/E(V_0)$ of its field of rationality over $E$. In particular $\overline{V}_0$ admits a model $V_0$ over a finite extension $E_0/E$.
          \item[(b)] If $V_0$ denotes an irreducible locally algebraic representation of $G$ over a finite extension $E_0/E$, then $\res_{E_0/E}V_0=V^{\oplus m}$ for any irreducible locally algebraic representation $V$ over $E$ occuring in $\res_{E_0/E}V_0$.
          \item[(c)] If $E_0=E(V_0)$ agrees with the field of rationality of $V_0$ over $E$, then $\res_{E_0/E}V_0$ is irreducible.
          \item[(d)] Every irreducible locally algebraic representation $V$ of $G$ over $E$ occurs in a restriction of scalars of an absolutely irreducible locally algebraic representation in a finite extension of $E$ as in (b).
          \end{itemize}
        \end{theorem}

          \begin{proof}
            Since $G$ splits over a finite extension of $E$, statement follows by Proposition \ref{prop:locallyalgebraicfieldofdefinition}.

            \smallskip
            Assume we are in the situation of (b), i.\,e.\ $V_0$ is an irreducible locally algebraic representation over a finite extension $E_0/E$ and let $V\subseteq\res_{E_0/E}V_0$ denote an irreducible subrepresentation.
            By the universal property of restriction of scalars we find for any irreducible constituent $W$ in $\res_{E_0/E}V_0$
            \[
            \Hom_{G\times\pi_0^{\rm const}(K)}(V_0,E_0\otimes_E W)=
            \Hom_{G\times\pi_0^{\rm const}(K)}(\res_{E_0/E}V_0,W)
            \neq 0,
            \]
            i.\,e.\ $V_0$ occurs in $E_0\otimes_E W$. Therefore we have
            \[
            E_0\otimes_E\Hom_{G\times\pi_\circ^G}(V,W)=
            \Hom_{G\times\pi_\circ^G}(E_0\otimes_E V,E_0\otimes W)\neq 0.
            \]
            Hence $V$ and $W$ are isomorphic and (b) follows.

            \smallskip
            Assume that $V_0$ is an irreducible locally representation defined over its field of rationality $E_0=E(V_0)$. For any embedding $\sigma\colon E_0\to\CC$ extending the fixed embedding $E\to\CC$ we consider the Galois-conjugate locally algebraic representation 
          \[
          \overline{V}_0^{\sigma}:=\CC\otimes_{\sigma,E_0} V_0.
          \]
          By the definition of the field of rationality $E(V_0)$ two embeddings $\sigma_1,\sigma_2\colon E_0\to\CC$ extending $E\to\CC$ give rise to isomorphic locally algebraic representations $\overline{V}_0^{\sigma_1}$ and $\overline{V}_0^{\sigma_2}$ if and only if $\sigma_1=\sigma_2$ (extend $\sigma_i$ to automorphisms $\overline{\sigma}_i\colon\CC\to\CC$ and proceed from there). In particular we find that there are
          \[
          [E_0:E]=\#\Hom_{E}(E_0,\CC)
          \]
          pairwise distinct isomorphism classes of locally algebraic representations among the locally algebraic representations $\overline{V}_0^{\sigma}$, for arbitrary $\sigma\colon E_0\to\CC$.

          The decomposition
          \[
          \CC\otimes_{E_0}\res_{E_0/E}V_0=\bigoplus_{\sigma\colon E_0\to\CC}\CC\otimes_{\sigma,E_0}V_0
          =\bigoplus_{\sigma\colon E_0\to\CC}\overline{V}_0^\sigma,
          \]
          where $\sigma$ ranges over the embeddings $E_0\to\CC$ extending the fixed embedding $E\to\CC$, shows that $\res_{E_0/E}V_0$ must be irreducible: For any irreducible locally algebraic subrepresentation $W\subseteq\res_{E_0/E} V_0$ contains, after base change to $\CC$, one of the absolutely irreducibles $\overline{V}_0^\sigma$, and hence it contains all its Galois-conjugates, which implies that it contains $\CC\otimes_{E_0}\res_{E_0/E}V_0$. Hence $W$ contains $\res_{E_0/E}V_0$ again by Galois descent. This proves (c).

          Statement (d) follows by considering any irreducible constituent $\overline{V}_0\subseteq\CC\otimes_E V$ and any model of $\overline{V}_0$ over a sufficiently large finite extension $E_0/E$.
          \end{proof}

          \begin{corollary}\label{cor:quasisplitlocallyalgebraicrepresentations}
            If $G$ is quasi-split over $E$, any irreducible locally algebraic representation of $G$ over $E$ is an irreducible restriction of scalars $\res_{E_0/E}V_0$ of an absolutely irreducible locally algebraic representation $V_0$ defined over its field of rationality $E_0/E$.
          \end{corollary}

          \begin{proof}
            Invoke Corollary \ref{cor:quasisplitrationality} and statements (d) and (b).
          \end{proof}

	\begin{definition}[Complex representations attached to a locally algebraic representation]\label{def:locallyalgebraiccomplex}
		For any locally algebraic representation $V=(V_{E}, \rho_G,\rho_{\pi_\circ^G})$ over a field $E$ of characteristic $0$ and any embedding $\iota\colon E\to\CC$, we consider the action $\rho_{\pi_\circ^G}$ via the canonical morphism $G(\RR)\to\pi_0(G(\RR))=\pi_0(K(\RR))=\pi_\circ^G(\RR)$ as a $\CC$-linear action $\rho_{\pi_\circ^G}^\CC$ of $G(\RR)$ on $V_\CC=\CC\otimes_{\iota,E}V_E$. Then we let $G(\RR)$ act on $V_\CC$ via $\rho_G\circ\rho_{\pi_\circ^G}^\CC$, i.\,e.\ 
		\[
		\forall g\in G(\RR),v\in V_\CC\colon\quad g\cdot v=\rho_G(g)\left(\rho_{\pi_\circ^G}^\CC(g)(v)\right).
		\]
		We call the resulting representation $(V_\CC,\rho_\CC)$ the {\em complex representation associated to the locally algebraic representation $V$.}
	\end{definition}

	%We remark that since the two given actions commute, the action $\rho_\RR$ of $G(\RR)$ on $V_\CC$ is well defined.

        \begin{remark}
          Proposition \ref{prop:absolutelyirreduciblelocallyalgebraic} shows that the complex realization $(V_\CC,\rho_\CC)$ of an absolutely irreducible locally algebraic $V$ representation of $G$ is always realized as the complexification $\rho_G$ of the underlying absolutely irreducible rational representation $(V_E,\rho_G)$ of $G$ twisted by a character of $\pi_0(G(\RR))$ or equivalently of the constant group scheme $\pi_\circ^G$.
        \end{remark}

        \begin{remark}\label{rmk:locallyalgebraicequivalence}
          On the connected component $\GG(\RR)^0$ of the identity of $G(\RR)$ we have $\rho_\CC|_{G(\RR)^0}=\rho_G$. With the observation of the previous remark one may show that the functor $V\mapsto (V_\CC,\rho_\CC)$ induces an equivalence between the category $\underline{\text{LocAlgRep}}_\CC(G)$ of complex locally algebraic representations of $G$ and the category of finite-dimensional continuous complex $G(\RR)$-representations which as $G(\RR)^0$-representations are (restrictions of) rational representations. This justifies the terminology \lq{}locally algebraic representations\rq{}.

          We remark that the latter category is equivalent to the category of finite-dimensional complex $(\lieg_\CC,K_\infty)$-modules which as $\lieg_\CC$-modules are sums of modules with integral infinitesimal characters.
        \end{remark}

        \begin{remark}
          For any real embedding $E\to\RR$ it is straightfoward to extend Definition \ref{def:locallyalgebraiccomplex} to define a real representation $G(\RR)\to\Aut_\RR(\RR\otimes_E V_E)$ to a given locally algebraic representation $V$ over $E$, whose complexification is $V_\CC$.

          In general the descent problem in the extension $\CC/\RR$ for locally algebraic representations $V$ over $\CC$ is easily seen by Proposition \ref{prop:locallyalgebraicfieldofdefinition} or Theorem \ref{thm:irreduciblelocallyalgebraicrepresentations} to be equivalent to the descent problem for the underlying complex rational representation $V_\CC$. In particular any Brauer obstruction for irreducible locally algebraic representations of $G$ over $\CC$ originates in a corresponding Brauer obstruction for the underlying irreducible rational representation. This observation extends to descent of locally algebraic representations in general finite extensions.
        \end{remark}

	\subsubsection{Lattices in locally algebraic representations}

        \begin{definition}[$K_f$-stable lattices]
	Fix a number field $E$ and a rational representation $(V_{E},\rho_G)$ of $G\otimes_{\QQ} E$. For a positive integer $N$ and a $\OO_E[1/N]$-submodule $V_{\OO_E[1/N]}\subset V_E$. Let $K_f$ be a compact open subgroup $K_f\subseteq G(\bfA_f)$. For each $g\in G(\bfA_f)$ we write
	\[
	gV_{\OO_E[1/N]}\coloneqq 
	g\left(\widehat{\ZZ}\otimes_{\ZZ} V_{\OO_E[1/N]}\right)\cap V_{E},
	\]
	where the intersection takes place in the $G(\bfA_f)$-module $\bfA_f\otimes_\QQ V_{E}$ and $g$ acts on $\bfA_f\otimes_\QQ V_E$ by $\rho_G$. We say that $V_{\OO_E[1/N]}$ is {\em $K_f$-stable} if for all $k\in K_f$:
	\[
	kV_{\OO_E[1/N]}\subseteq V_{\OO_E[1/N]}.
	\]
        \end{definition}

        \begin{remark}
          If $V_{\OO_E[1/N]}$ is $K_f$-stable, then for every multiple $N'$ of $N$ the lattice $\ZZ[1/N']\otimes_{\ZZ[1/N]} V_{\OO_E[1/N]}$ is $K_f$-stable as well. In particular, existence of $K_f$-stable lattices reduces to the case $N=1$.
        \end{remark}
        
        We establish the general existence of $K_f$-stable lattices in the following
	
	\begin{example}[Existence of $K_f$-stable lattices]
		Fix any $\OO_E$-lattice $W\subseteq V_{E}$ and consider the continuous map
		\[
		K_f\times (\bfA\otimes_\QQ V_{E})\to \bfA\otimes_\QQ V_{E},\quad (k,v)\mapsto kv.
		\]
		Since $K_f$ and $\bfA\otimes_\ZZ W$ are both compact open, its image $K_fW_\bfA$ is compact open and
		\begin{equation}
			V_\OO:=K_fW_{\bfA}\cap V_{E}
			\label{eq:latticeconstruction}
		\end{equation}
		is a $K_f$ stable lattice by construction. 
	\end{example}
	
	\begin{example}[Scheme theoretic actions and lattices]\label{ex:schemeactiononlattice}
		If $G$ admits a smooth model $G_{\ZZ[1/N]}$ over a localization $\ZZ[1/N]$ for some non-zero integer $N$, then any $K_f$-stable $\OO[1/N]$-lattice $L\subseteq V_E$ generates an $\OO[1/N]$-model $V_{\OO[1/N]}$ of $V_E$ as a representation of the group scheme $G_{\ZZ[1/N]}$, cf.\ \cite{hayashi2018} Proposition 2.1.6. Moreover, whenever
		\begin{equation}
			K_f\subseteq G\left(\widehat{\ZZ}[1/N]\right),
			\label{eq:KfinGZhat}
		\end{equation}
		then $V_{\OO[1/N]}$ is $K_f$-stable.
	\end{example}
	
	\begin{example}[Smooth models of reductive groups]\label{ex:smoothgroups}
		Let $E$ be a number field, and $G$ be a split reductive group scheme over $\OO_E$. Then $\res_{\OO_E/\ZZ} G$ is a smooth affine group scheme over $\ZZ$ by \cite{boschetal} Section 7.6, Proposition 5. In particular, $\res_{\OO_E/\ZZ} G$ is flat over $\ZZ$. We also have a natural identification $(\res_{\OO_E/\ZZ} G)\otimes_{\ZZ}\QQ\cong \res_{E/\QQ} (G\otimes_{\OO_E} E)$.
	\end{example}

	\begin{remark}[Extension to lattices]
	  Now assume $V=(V_E,\rho_G,\rho_{\pi_\circ^G})$ is a locally algebraic representation of $G$ over a number field $E$. Let $W\subseteq V_E$ denote a $K_f$-stable lattice over $\OO_E[1/N]$. We consider $\pi_\circ^G$ as a constant group scheme over $\Spec\OO_E[1/N]$. By Proposition \ref{prop:absolutelyirreduciblelocallyalgebraic} we know that \'etale locally on $\Spec E$, $\pi_\circ^G$ acts on $V_E$ via characters and therefore \'etale locally on $\Spec\OO_E[1/N]$, $W$ is stable under the action $\rho_{\pi_\circ^G}$ and we conclude that $\pi_\circ^G$ acts on $W$ scheme-theoretically. In particular, in the context of Example \ref{ex:schemeactiononlattice}, $G\times\pi_\circ^G$ acts on $W$ scheme-theoretically.
        \end{remark}
	
	\subsection{Sheaves associated to locally algebraic representations}

        \subsubsection{Orbifold sheaves associated to locally algebraic representations}

        \begin{construction}
        In order to define for a locally algebraic representation $V=(V_E,\rho_G,\rho_{\pi_\circ^G})$ a sheaf $\widetilde{V}_E$ on the orbifold $\mathscr X_G(K_f)$, we introduce some notation. The orbifold $\mathscr X_G(K_f)$ corresponds to a proper \'etale Lie groupoid $\mathcal G(K_f)$. Up to Morita equivalence, we are free to choose $\mathcal G(K_f)$ as
        \begin{align*}
          \Ob\mathcal G(K_f)&:=G(\QQ)\backslash G(\bfA)/S(\RR)^0K_f,\\
          \Ar\mathcal G(K_f)&:=G(\QQ)\backslash G(\bfA)/S(\RR)^0K_f\times K_\infty,
        \end{align*}
        together with the structural morphisms
        \begin{align}
          s\colon\; &\Ar\,\mathcal G(K_f)\to{\rm Ob}\,\mathcal G(K_f), &&(x,k)\mapsto x,\label{eq:groupoidsource}\\
          t\colon\; &\Ar\,\mathcal G(K_f)\to{\rm Ob}\,\mathcal G(K_f), &&(x,k)\mapsto xk,\label{eq:groupoidtarget}
        \end{align}
        defining source and target of arrows\footnote{Remark that $\Ar\mathcal G(K_f)\cong K_\infty\times G(\QQ)\backslash G(\bfA)/S(\RR)^0K_f$ via $(x,k)\mapsto (k^{-1},x)$ if the reader prefers a left action groupoid construction.}.
        
	For a locally algebraic representation $(V_E,\rho_G,\rho_{\pi_\circ^G})$ we define a sheaf $\widetilde{V}_E$ of $E$-vector spaces on $\mathscr X_G(K_f)$ as follows. Consider $V_E$ as a discrete topological space and the associated vector bundle
        \[
        \mathcal E(V_E):=G(\QQ)\backslash (V_E\times G(\bfA)/S(\RR)^0K_f)
        \]
        with projection
        \begin{align}
          \pi_V^{K_f}\colon\;\mathcal E(V_E)
          &\to
          \Ob\mathcal G(K_f),\label{eq:bundleprojection}\\
          G(\QQ) (v,x)&\mapsto G(\QQ)x.\nonumber
        \end{align}
        Put
        \begin{align*}
          \mu_V^{K_f}\colon\;\mathcal E(V_E)\times_{\Ob\mathcal G(K_f)}\Ar\mathcal G(K_f)&\to \mathcal E(V_E),\\
          ((v,xk),(x,k))&\mapsto (\rho_{\pi_\circ^G}(\pi_0(k^{-1}))(v),x),
        \end{align*}
        where the fiber product is understood with respect to the projection $\pi_V^{K_f}$ and the target map \eqref{eq:groupoidtarget}. The pair $(\pi_V^{K_f},\mu_V^{K_f})$ defines a $\mathcal G(K_f)$-space with fiberwise $E$-linear action, i.\,e.\ we obtain a vector bundle $\mathcal E(V_E)$ on $\mathscr X_G(K_f)$. The projection $\pi$ is a local diffeomorphism, i.\,e.\ $\mathcal E(V_E)$ is \'etale over $X_G(K_f)$ and the following definition is meaningful.
        \end{construction}
        
        \begin{definition}[Sheaf associated to a locally algebraic representation on $\mathscr X_G(K_f)$]\label{def:locallyalgebraicorbifoldsheaf}
	  Given a locally algebraic representation $(V_E,\rho_G,\rho_{\pi_\circ^G})$ we define $\widetilde{\mathcal V}_E$ as the $K_\infty$-equivariant sheaf of $E$-vector spaces on $\Ob\mathcal G(K_f)$ whose sections are the continuous sections of the projection \eqref{eq:bundleprojection}.%$\pi_V^{K_f}\colon\mathcal E(V_E)\to\Ob\mathcal G(K_f)$.
        \end{definition}

        \begin{remark}[Explicit $K_\infty$-equivariant sections]
          Let $U\subseteq\Ob\mathcal G(K_f)$ be an open set. For any section $s\colon U\to\mathcal E(V_E)$ and any $u=G(\QQ)gS(\RR)^0K_f\in\overline{U}$ with $g\in G(\bfA)$ we have
          \[
          s(u)=G(\QQ)(v_u,g_uS(\RR)^0K_f),\quad v_u\in V_E,\,g_u\in G(\bfA).
          \]
          The relation $\pi_V^{K_f}\circ s={\bf1}_{U}$ implies that $g_u=\gamma_u^{-1} g$ for $\gamma_u\in G(\QQ)$. Therefore we may rewrite
          \[
          s(u)=G(\QQ)(\gamma_uv_u,gS(\RR)^0K_f).
          \]
          Observe that $G(\QQ)$ acts faithfully on $G(\bfA)/S(\RR)^0K_f$. Therefore continuity of $s$ shows that $gS(\RR)^0K_f\mapsto \gamma_uv_u$ defines a locally constant function on the set $\overline{U}\subseteq G(\bfA)/S(\RR)^0K_f$ of the cosets $gS(\RR)^0K_f$ above $U$. I.\,e.\ we obtain a natural isomorphism between the $E$-vector space of sections $s$ on $\overline{U}$ and the space of maps
          \[
          f\colon \overline{U}\to V_E,
          \]
          which are locally constant and $G(\QQ)$-equivariant: $f(\gamma x)=\rho_G(\gamma)f(x)$ for all $\gamma\in G(\QQ)$, $x\in\overline{U}$.

          Note that $K_\infty$ naturally acts on $f$ via
          \[
          (kf)(x)=\rho_{\pi_\circ^G}(k)f(xk),\quad k\in K_\infty,\,x\in\overline{U},
          \]
          which corresponds to the action of $\Ar\mathcal G(K_f)$ on $\mathcal E(V_E)$. In particular, we associated to $\mathcal E(V_E)$ a $K_\infty$-equivariant sheaf as desired.
        \end{remark}
        
        \begin{remark}[Sections on orbifold charts]
          The $K_\infty$-equivariant sheaf $\widetilde{\mathcal V}_E$ descends to an equivariant sheaf on orbifold charts as follows. By construction, $\mathscr X_G(K_f)$ is the orbifold quotient of the space $\Ob\mathcal G(K_f)$ by the action of $K_\infty$, i.\,e.\ we have the canonical projection map
	  \[
	  p\colon \Ob\mathcal G(K_f)\to X_G(K_f),
	  \]
          sending objects of the groupoid to $K_\infty$-orbits.

          For a connected simply connected orbifold chart $(\widetilde{U},G,\varphi)$ mapping onto the open set $U=\varphi(\widetilde{U})\subseteq X_G(K_f)=\widetilde{U}/G$, we have a canonical identification
          \[
          p^{-1}(U)/K_\infty=\widetilde{U}/G,
          \]
          of topological spaces. We may assume that there is an $\widetilde{x}\in\widetilde{U}$ with stabilizer $G_{\widetilde{x}}=G$. Put $x:=\varphi(x)$ and let $\overline{x}\in p^{-1}(U)$ denote a preimage.

          Moreover we may assume that we have a slice representation
          \[
          K_\infty\times_{K_{\infty,\overline{x}}}\overline{V}\to p^{-1}(U)
          \]
          with $\overline{V}$ a linear $K_{\infty,\overline{x}}$-representation over $\RR$. Now $K_{\infty,\overline{x}}\cong G_{\widetilde{x}}$ and we obtain a tautological slice representation
          \[
          \overline{V}=G\times_{G_{\widetilde{x}}}\overline{V}\to \widetilde{U}.
          \]
          Considering the action groupoid $\mathcal H$ with $\Ob\mathcal H=\widetilde{U}$ and $\Ar\mathcal H=\widetilde{U}\times G$ the above datum gives rise to a morphism $\phi\colon\mathcal H\to\mathcal G$ of groupoids with
          \[
          \Ob\phi\colon\Ob\mathcal H=\widetilde{U}\cong\overline{V}\to p^{-1}(U)\subseteq\Ob\mathcal G,
          \]
          and
          \[
          \Ar\phi\colon
          \Ar\mathcal H=\widetilde{U}\times G\cong\overline{V}\times K_{\infty,\widetilde{x}}\to p^{-1}(U)\times K_\infty\subseteq\Ar\mathcal G.
          \]
          Pulling back the $\mathcal G$-space $\mathcal E(V_E)$ along $\phi$ to an $\mathcal H$-space $\phi^*\mathcal E(V_E)$, the associated $G$-equivariant sheaf on $\widetilde{U}$ is the sheaf $\widetilde{\mathcal V}_E|_{\widetilde{U}}$ on the orbifold chart $\widetilde{U}$.
        \end{remark}

        \subsubsection{Topological sheaves associated to locally algebraic representations}

        \begin{definition}[Sheaf associated to a locally algebraic representation on $X_G(K_f)$]\label{def:locallyalgebraicsheaf}
	  Given a locally algebraic representation $(V_E,\rho_G,\rho_{\pi_\circ^G})$ we define the sheaf $\widetilde{V}_E$ of $E$-vector spaces on $X_G(K_f)$ as the pushforward $p_{K_f,\ast}\widetilde{\mathcal V}_E$ under the projection \eqref{eq:orbifoldtomanifold}.
        \end{definition}

        \begin{remark}[Explicit description of sections]
          Denote by          
	  \[
	  \pi\colon G(\bfA)/S(\RR)^0 K(\RR)^0 K_f\to X_G(K_f),
	  \]
	  the canonical projection. For $U\subseteq X_G(K_f)$ open we have
	  \[
	  \Gamma(U;\widetilde{V}_E)=
	  \{s\colon\pi^{-1}(U)\to V_E\mid s\;\text{locally constant and}\;
	  \]
	  \[\forall \gamma\in G(\QQ),k\in K(\RR)\,u\in\pi^{-1}(U):\;f(\gamma u k)=\rho_G(\gamma)\rho_{\pi_\circ^G}(\pi_0(k))f(u)\}.
	  \]
        \end{remark}

        \begin{remark}
          The sheaf cohomologies of $\widetilde{\mathcal V}_E$ and $\widetilde{V}_E$ are canonically isomorphic by the degenerating Leray spectral sequence \eqref{eq:lerayspectralsequence} (cf.\ Remark \ref{rmk:leraydegeneration}). This does not apply to the sheaf cohomology of associated to lattices in the next section.
        \end{remark}

        \subsubsection{Sheaves associated to locally algebraic lattices}
        
        \begin{definition}[Sheaf associated to a lattice in a locally algebraic representation]\label{def:locallyalgebraicintegralsheaf}
	  For any $K_f$-stable lattice $V_{\OO[1/N]}\subseteq V_E$ we define a subsheaf $\widetilde{\mathcal V}_{\OO[1/N]}\subseteq\widetilde{\mathcal V}_E$ (resp.\ $\widetilde{V}_{\OO[1/N]}\subseteq\widetilde{V}_E$) of $\OO[1/N]$-modules whose sections $s$ are characterized by the condition that the germs $s_x$ at all points $x\in|\mathscr X_G(K_f)|=X_G(K_f)$ which in turn is represented by an element $g\in G(\bfA)$ be contained in the submodule $g_fV_{\OO[1/N]}\subseteq V_E$.
        \end{definition}

        \begin{remark}
        We remark that the isomorphism $\widetilde{V}_{\OO[1/N],x}\to g_fV_{\OO[1/N]}$ identifying the stalk at $x$ with the submodule $g_fV_{\OO[1/N]}\subseteq V_E$ depends on the chosen representative $g$. Nonetheless, if $h$ is another representative of the double coset $x$, we may write $h=\gamma g s_\infty k_\infty k_f$ with $\gamma\in G(\QQ)$, $s_\infty\in S(\RR)^0$, $k_\infty\in K(\RR)$ and $k_f\in K_f$. The relation $h_f=\gamma g_fk_f$ shows that we have a commutative square
	\[
	\begin{CD}
		h_fV_\OO @>>> V_E\\
		@V{\cong}VV @VV{\rho_G(\gamma)\rho_{\pi_\circ^G}(\pi_0(k_\infty))}V\\
		g_fV_\OO @>>> V_E
	\end{CD}
	\]
        Therefore our characterization via stalks is independent of the chosen representative $g\in G(\bfA)$.
        \end{remark}

        \begin{remark}
          We have $p_{K_f,\ast}\widetilde{\mathcal V}_{\OO[1/N]}\cong\widetilde{V}_{\OO[1/N]}$ canonically. The sheaf cohomologies of $\widetilde{\mathcal V}_{\OO[1/N]}$ and $\widetilde{V}_{\OO[1/N]}$ are related via the Leray spectral sequence \eqref{eq:lerayspectralsequence}.
        \end{remark}

        \begin{remark}[A canonical Galois cover]\label{rmk:galoiscover}
          Put
          \begin{equation}
            \Gamma_0:=\image\left(G(\QQ)\cap K_f\to\pi_0(G(\RR))\right)
            \label{eq:Kfcomponents}
          \end{equation}
          and consider $\Gamma_0$ as a subgroup of $\pi_0(K_\infty)$. Passing to the factor group we obtain the corresponding relative component group
          \begin{equation}
            \pi_0(K_\infty,K_f):=\pi_0(K_\infty)/\image\left(G(\QQ)\cap K_f\to\pi_0(G(\RR))\right).
            \label{eq:relativepi0}
          \end{equation}
          We have the Galois cover
	  \[
	  p_0\colon X_G(K_f)_0\to X_G(K_f),
	  \]
	  whose covering group is the relative component group $\pi_0(K_\infty,K_f)$.

          On $X_G(K_f)_0$ we may define likewise a sheaf $\widetilde{V}_E$ associated to the rational representation $(V_E,\rho_G)$ omitting $\rho_{\pi_\circ^G}$. By Proposition \ref{prop:matsumoto}, $E$ contains the values of all characters $\chi$ of $\pi_0(K_\infty,K_f)$, hence
	  \begin{equation}
	    p_{0,\ast}(\widetilde{V}_E)\;=\;\bigoplus_{\chi\in\pi_0(K_\infty,K_f)^\vee}\widetilde{V}_{E,\chi},
	    \label{eq:VEdecomposition}
	  \end{equation}
	  where $\chi$ runs through the Pontryagin dual $\pi_0(K_\infty,K_f)^\vee$ of $\pi_0(K_\infty,K_f)$ and $\widetilde{V}_{E,\chi}$ is the sheaf on $X_G(K_f)$ associated to the locally algebraic representation
          \[
          V_\chi\;:=\;(V_E,\rho_G,\chi\circ(\pi_0(K_\infty)\to\pi_0(K_\infty,K_f))),
          \]
          cf.\ Definition \ref{def:locallyalgebraicsheaf}. Moreover, whenever $|\pi_0(K_\infty,K_f)|$ is invertible in $\OO[1/N]$ we have likewise
	  \begin{equation}
	    p_{0,\ast}(\widetilde{V}_{\OO[1/N]})\;=\;\bigoplus_{\chi\in\pi_0(K_\infty,K_f)^\vee}\widetilde{V}_{\OO[1/N],\chi}.
	    \label{eq:OEdecomposition}
	  \end{equation}
          Here the $K_f$-stable lattice $V_{\OO[1/N],\chi}\subseteq V_{E,\chi}$ is independent of $\chi$ (omitting the $\pi_0(K_\infty)$-action). In general, the right hand side is a subsheaf of the left hand side and we have a corresponding spectral sequence \eqref{eq:chispectralsequence} in sheaf cohomology (cf.\ section \ref{sec:cohomologicalrepresentations} below).
        \end{remark}

        \begin{remark}
          Similarly we may define an orbifold cover $\mathscr X_G(K_f)_0\to\mathscr X_G(K_f)$ and obtain orbifold analogues of the decompositions \eqref{eq:VEdecomposition} and \eqref{eq:OEdecomposition}. Such decompositions then are compatible with the Leray spectral sequence \eqref{eq:lerayspectralsequence}. Likewise, the spectral sequence \eqref{eq:chispectralsequence} admits an orbifold analogue, again comaptibe the Leray spectral sequence \eqref{eq:lerayspectralsequence}.
        \end{remark}
                
	\begin{remark}[Non-triviality considerations]\label{rmk:nontriviality}
          The sheaf $\widetilde{V}_E$ on the cover $X_G(K_f)_0$ is non-trivial if $V_E\neq 0$ and the center $Z(\Gamma)$ of $\Gamma=G(\QQ)\cap K_f$ acts trivially on $V_E$. However, the latter condition is too strong. In fact, a sufficient and necessary condition is
          \begin{equation}
            Z^0(\gamma):=Z(\gamma)\cap\ker\left(\Gamma\to\pi_0(G(\RR))\right)\;\subseteq\;\ker\rho_G.
            \label{eq:centralcondition}
          \end{equation}
          Combining this observation with \eqref{eq:VEdecomposition} and the fact that every local system $\widetilde{V}_E$ on $X_G(K_f)$ associated to a locally algebraic representation $V$ occurs as a subsheaf of $p_{0,\ast}(\widetilde{V}_E)$ by the adjointness of $p_{0,\ast}$ and $p_0^\ast$, we see that the sheaf $\widetilde{V}_E$ associated to the locally algebraic representation $V$ on $X_G(K_f)$ is non-trivial if and only if $V_E\neq 0$, condition \eqref{eq:centralcondition} is satisfied and the action $\rho_{\pi_\circ^G}$ on $V_E$ factors through $\pi_0(K_\infty,K_f)$.

          Regardless of the invertibility of the order of $\pi_0(K_\infty,K_f)$ in $\OO[1/N]$ the sheaf $\widetilde{V}_{\OO[1/N]}$ is non-trivial if and only if $\widetilde{V}_E$ is so. Hence the same criterion applies.
        \end{remark}

        \begin{remark}
          The condition on $\rho_{\pi_\circ^G}$ is empty if
	  \begin{equation}
	    K_f\cap G(\QQ)\subseteq G(\QQ)\cap G(\RR)^0.
	    \label{eq:Kfsign}
	  \end{equation}            
	\end{remark}

        \section{Locally algebraic $(\lieg,K)$-modules}\label{sec:locallyalgebraicgKmodules}

        In order to describe the automorphic representations contributing to the cohomology of the sheaves $\widetilde{V}_E$ associated to locally algebraic representations $V$ of $G$, we are naturally led to an extension of the notion of locally algebraic representations first to $(\lieg,K)$-modules and subsequently to cohomologically induced standard modules. These refinements of the classical theory are necessary to identify the appropriate corresponding rational structures on automorphic representations which will then enable us to refine our constructions to integral structures.

	\subsection{Locally algebraic $(\lieg,K)$-modules}

        In this section we consider the reductive pair $(\lieg,K)$ over $\QQ$ associated to $G$.
        
	\begin{definition}[Locally algebraic $(\lieg,K)$-module]\label{def:locallyalgebraicgKmodule}
	  A {\em locally algebraic $(\lieg,K)$-module} over a field $E$ of characteristic $0$ is a pair $X=(X_E,\rho_{\pi_\circ^G})$ consisting of a $(\lieg,K)$-module $X$ over $E$ in the sense of \cite{januszewskirationality} and an $E$-linear action $\rho_{\pi_\circ^G}$ of $\pi_0(K(\RR))$ on $X_E$ which commutes with the two actions of $K$ and $\lieg$ on $X$ given by the $(\lieg,K)$-module structure on $X_E$.

          We say that $X$ is {\em algebraic} or {\em rational} if $\rho_{\pi_\circ^G}={\bf1}$ is trivial.

          Morphisms of locally algebraic $(\lieg,K)$-modules are $E$-linear homorphisms of $(\lieg,K)$-modules which preserve the additional $\pi_0(K(\RR))$-actions.
	\end{definition}

        \begin{remark}
          Once again we may consider the action of $\pi_0(K(\RR))$ as a scheme theoretic action of the constant group scheme $\pi_\circ^G$ on $X$ in the sense of a cohomodule action of the Hopf algebra underlying $\pi_\circ^G$. From this scheme-theoretic standpoint, letting $\pi_\circ^G$ act trivially on $\lieg$, we may therefore consider a locally algebraic $(\lieg,K)$-module as a $(\lieg,K\times\pi_\circ^G)$-module and vice versa. The latter point of view will be a natural choice when considering integral models of locally algebraic $(\lieg,K)$-modules.
        \end{remark}

        \begin{example}
          Every locally algebraic rational representation $V=(V_E,\rho_G,\rho_{\pi_\circ^G})$ of $G$ gives rise to a finite-dimensional locally algebraic $(\lieg,K)$-module $X=(V_E,\rho_{\pi_\circ^G})$ and $X$ is algebraic if $V$ is algebraic but the converse is not true in general: Not every finite-dimensional $(\lieg,K)$-module comes from a rational representation. Nonetheless the functor $V\mapsto(V_E,\rho_{\pi_\circ^G})$ from locally algebraic rational representations to locally algebraic $(\lieg,K)$-modules is fully faithful for modules over fields of characteristic $0$.
        \end{example}
	
	\begin{remark}\label{rmk:absolutelyirreduciblelocallyalgebraicgKmodule}
          We remark that Propositions \ref{prop:locallyalgebraicfieldofrationality} and \ref{prop:locallyalgebraicfieldofdefinition} have natural analogs for locally algebraic $(\lieg,K)$-modules with identical proofs. In particular the absolute irreducibility of $A_\lieq(\lambda)_E$, Schur's Lemma implies that $\pi_0(K(\RR))$ acts via scalars in $\{\pm1\}$ and that descent problems for locally algebraic $(\lieg,K)$-modules are equivalent to the corresponding descent problems for algebraic $(\lieg,K)$-modules over fields of characteristic $0$ or more generally as long as $2$ is invertible on the base in an appropriate ring setting.
	\end{remark}

        The analogue of Theorem \ref{thm:irreduciblelocallyalgebraicrepresentations} is

        \begin{theorem}[Irreducibility of restrictions of scalars of locally algebraic $(\lieg,K)$-modules]\label{thm:irreduciblelocallyalgebraicgKmodules}
          Let $E$ be an arbitrary field of characteristic $0$. Let $X=(X_E,\rho_{\pi_\circ^G})$ denote a locally algebraic $(\lieg,K)$-module over $E$ whose base change to an algebraic closure $\overline{E}/E$ is of finite length and contains an absolutely irreducible locally algebraic submodule $\overline{X}_0\subseteq \overline{E}\otimes_E X_E$ which is assumed to be defined over its field of rationality $E_0:=E(\overline{X}_0)$. Let $X_0$ denote a model of $\overline{X}_0$ over $E_0$. Then $E_0/E$ is finite and $X$ is isomorphic to the restriction of scalars $\res_{E_0/E}X_0$.
        \end{theorem}

        \begin{proof}
          The proof proceeds along the same lines as the proof of Theorem \ref{thm:irreduciblelocallyalgebraicrepresentations} with the following additional observations: The field of rationality of $\overline{X}_0$ is finite over $E$ because all Galois-conjugates of $\overline{X}_0$ are submodules of $\overline{E}\otimes_E X_E$, which is a module of finite length. The rest of the proof proceeds as in the case of locally algebraic representations.
        \end{proof}

	\begin{definition}[Complex $(\lieg,K)$-module attached to a locally algebraic $(\lieg,K)$-module]\label{def:locallyalgebraiccomplexmodule}
	  For any locally algebraic $(\lieg,K)$-module $X=(X_E,\rho_{\pi_\circ^G})$ over a field $E$ of characteristic $0$ and any embedding $\iota\colon E\to\CC$, we twist the underlying $(\lieg_\CC,K_\infty)$-module structure over $\CC$ on $X_\CC:=\CC\otimes{\iota,E} X_E$ by the action of $K$ induced by $\rho_{\pi_\circ^G}$, i.\,e.\ the $K(\RR)$-action is defined via pullback along the canonical \lq{}diagonal\rq{} morphism
          \[
          K(\RR)\to K(\RR)\times \pi_\circ^G(\RR).
          \]
          This defines a $\CC$-linear action $\rho_{K,\CC}$ of $K(\RR)$ on $X_\CC$ and we call the resulting $(\lieg_\CC,K_\infty)$-module $X_\CC$ the {\em complex $(\lieg,K)$-module associated to the locally algebraic $(\lieg,K)$-module $X$.}
	\end{definition}

        \begin{remark}
           Explicitly, the new action of $K(\RR)$ on $X$ via $\rho_{K,\CC}$ is given by
	   \[
	   \forall k\in K(\RR),v\in X_\CC\colon\quad k\ast v=\rho_{\pi_\circ^G}(k)(k\cdot v),
	   \]
           where we wrote $k\ast v$ for the new action and $k\cdot v$ for the action of $K(\RR)$ on $X$ as a $(\lieg,K)$-module.

           We swiftly verify that the actions of $\lieg_\CC$ and $K(\RR)$ on $X$ are compatible: For any $g\in\lieg_\CC$, any $k\in K(\RR)$ and any $v\in X_\CC$ we have
           \begin{align*}
           k\ast(g\cdot v)&=
           \rho_{\pi_\circ^G}(k)(k\cdot (g\cdot v))\\&=
           \rho_{\pi_\circ^G}(k)(\ad(k)(g)\cdot(k\cdot v))\\&=
           \ad(k)(g)\cdot\rho_{\pi_\circ^G}(k)(k\cdot v)\\&=
           \ad(k)(g)\cdot(k\ast v).
           \end{align*}
           For any $k\in K(\RR)^0$ we have $k\ast v=k\cdot v$, which implies that the derivative of the new action agrees with the given action of $\Lie_\CC(K(\RR))\subseteq\lieg_\CC$ on $X_\CC$.
        \end{remark}

        \begin{remark}
	  The notion of complexification of locally algebraic representations $V$ from Definition \ref{def:locallyalgebraiccomplex} is compatible with the notion of complexification of locally algebraic $(\lieg,K)$-modules in Definition \ref{def:locallyalgebraiccomplexmodule}: The action of $\lieg_\CC$ on $V_\CC$ is the derivative of the action $\rho_\CC:G(\RR)\to\Aut_\CC(V_\CC)$ and
          \[
          \rho_{\pi_\circ^G,\CC}=\rho_\CC|_{K(\RR)}.
          \]
          Note that contrary to the case of locally algebraic representations the functor sending complex locally algebraic $(\lieg_\CC,K_\CC)$-modules to $(\lieg_\CC,K_\infty)$-modules is faithful but not full in general since non-isomorphic complex locally algebraic $(\lieg,K)$-modules may have isomorphic complex realizations: An example is the compact pair $(0,\pm1)$.
        \end{remark}
        
	\subsection{Locally algebraic cohomologically induced standard modules}\label{sec:locallyalgebraiccohomologicallyinducedstandardmodules}

        Let $V=(V_E,\rho_G,\rho_{\pi_\circ^G})$ denote an absolutely irreducible locally algebraic representation of $G$ over $E=\CC$ and put
        \begin{equation}
          \lambda=H^0(\bar{\lieu}',V_\CC)^\vee,
          \label{eq:lambdaforV}
        \end{equation}
        where $V_\CC$ denotes the complex realization of $V$ in the sense of Definition \ref{def:locallyalgebraiccomplex}. Then the corresponding cohomologically induced (standard) module is
        \begin{equation}
        A_{\lieq'}(\lambda)_\CC=
        R^uI^{\lieg_\CC,K_\CC}_{\lieq',L'\cap K_\CC}
        \left(\lambda\otimes_\CC \wedge^{\mathrm{top}} \lieg_\CC/\bar{\lieq}'\right),
        \label{eq:cohomologicalarchimedeancomponent}
        \end{equation}
        We refer to \cite{voganzuckerman1984} and Theorem 5.6 therein for a characterization of such modules in terms of non-vanishing relative Lie algebra cohomology. A reformulation in the language we use below is Theorem 6.4.2 of \cite{hayashijanuszewski}. In section 6.2 of the same work Hayashi and the author discuss the construction of models of cohomologically induced standard modules $A_{\lieq'}(\lambda)_\CC$ from the perspective of a scheme theoretic theory of $\mathcal D$-modules. This perspective will be of relevance for the construction of $1/N$-integral structures.
	
	\begin{definition}[Locally algebraic cohomologically induced standard module]\label{def:locallyalgebraicstandardmodule}
	  An {\em (absolutely irreducible) locally algebraic cohomologically induced standard module} for $(\lieg,K)$ over a field $E\subseteq\CC$ is an absolutely irreducible locally algebraic $(\lieg,K)$-module $X$ over $E$ which is an $E$-model of a complex cohomologically induced standard module $A=(A_{\lieq'}(\lambda)_\CC,\rho_{\pi_\circ^G})$ whose underlying $(\lieg_\CC,K_\infty)$-module $A_{\lieq'}(\lambda)_\CC$ is given by \eqref{eq:cohomologicalarchimedeancomponent} where the character $\lambda$ of the associated Levi pair $(\liel_\CC',L'\cap K(\RR))$ to the $\theta$-stable parabolic $\lieq'$ is associated to an extremal weight space of a an absolutely irreducible $E$-rational representation $(V_E,\rho_E)$ of $G$ via \eqref{eq:lambdaforV}.

          We call the locally algebraic representation $V=(V_E,\rho_E,\rho_{\pi_\circ^G})$ with the same $\pi_0(K(\RR))$-action $\rho_{\pi_\circ^G}$ as on $X$ the {\em weight} of $X$.

	  A {\em locally algebraic (cohomologically induced) standard module} for $(\lieg,K)$ is an irreducible locally algebraic $(\lieg,K)$-module $X$ which occurs as a subquotient in the restriction of scalars $\res_{E_0/E}A$ of an absolutely irreducible locally algebraic standard module $X_0$ over a finite extension $E_0/E$.

          We call an irreducible $E$-rational representation $V$ occuring in the restriction of scalars $\res_{E_0/E}V_0$ of the weight $V_0$ of $X_0$ over $E_0$ the {\em weight} of $X$ if the relative Lie algebra cohomology of $X$ with coefficients in $V$ is non-zero.
	\end{definition}

        \begin{remark}
          By absolute irreducibility, $\pi_0(K(\RR))$ acts on $X$ via scalars $\pm1$. This action naturally extens to an action of $\pi_0(K(\RR))$ on $V_E$.
        \end{remark}

        \begin{remark}\label{rmk:locallyalgebraicstandardmodulesarecohomologicallyinduced}
          For an absolutely irreducible locally algebraic standard module $A=(A_{\lieq'}(\lambda)_E,\rho_{\pi_\circ^G})$ of weight $V=(V_E,\rho_E,\rho_{\pi_\circ^G})$ the weight $\lambda$ in \eqref{eq:lambdaforV} is by definition defined via the complexification $V_\CC$ of the rational representation $(V_E,\rho_E)$. If we use instead the complexification of $V$ to define a weight $\lambda'$, then the module defined by the right hand side in \eqref{eq:cohomologicalarchimedeancomponent} for $\lambda'$ realizes the complexification of $A$ as a cohomologically induced standard module.
        \end{remark}
	
	\begin{remark}[Cohomological interpretation of the weight]\label{rmk:cohomologyoflocallyalgebraicstandardmodules}
          By the previous Remark \ref{rmk:locallyalgebraicstandardmodulesarecohomologicallyinduced}, the associated complex $(\lieg,K)$-module $A_{\lieq'}(\lambda)_\CC$ for the locally algebraic standard module $(A_{\lieq'}(\lambda),\rho_{\pi_\circ^G})$ in the sense of Definition \ref{def:locallyalgebraiccomplexmodule} has the property that
	  \[
	  H^\bullet(\lieg_\CC,S(\RR)^0 K(\RR); A_{\lieq'}(\lambda)_\CC\otimes V_\CC)\;\neq\;0.
	  \]
          where $V_\CC$ is the complexification of the weight $V$ of $G$ in the sense of Definitions \ref{def:locallyalgebraiccomplex}, cf.\ Theorem 5.6 in \cite{voganzuckerman1984} (also summarized in Theorem 6.4.2 in \cite{hayashijanuszewski}).
	\end{remark}

        \begin{remark}
          If $V_E$ is of regular highest weight, then Corollary 6.4.14 in \cite{hayashijanuszewski} shows that $A_{\lieq'}(\lambda)$ is uniquely determined by $V$ if the flag variety of $G_\CC$ contains a unique closed $K_\CC$-orbit.
        \end{remark}
	
	\begin{remark}[Restriction of scalars of locally algebraic standard modules defined over their field of rationality]
	  We emphasize that if $E/F$ is a finite extension of fields of characteristic $0$ and if $(A_{\lieq'}(\lambda)_E,\rho_{\pi_\circ^G})$ is an locally algebraic standard module over $E$, with field of rationality $E$ with respect to the extension $E/F$, then $\res_{E/F}A_{\lieq'}(\lambda)_E$ is an irreducible $(\lieg,K\times\pi_\circ^G)$-module over $F$, i.\,e.\ a locally algebraic $(\lieg,K)$-module in the sense of Definition \ref{def:locallyalgebraicstandardmodule}.
	\end{remark}

        \subsection{Admissibile weights}\label{sec:admissibleweights}

        \begin{definition}[Admissible weights and admissible twists]\label{def:admissibleweights}
          Let $A=(A_{\lieq'}(\lambda)_E,\rho_{\pi_\circ^G})$ denote a locally algebraic standard module of weight $V=(V_E,\rho_G,\rho_{\pi_\circ^G})$. We call a locally algebraic representation $V'=(V_E',\rho_G',\rho_{\pi_\circ^G}')$ {\em admissible} for $A$, if there is an isomorphism $\varphi\colon V_E'\to V_E$ of rational representations of $G$ and a $\QQ$-rational character $\chi\colon\pi_\circ(K)\to\Gm$ of the finite \'etale group scheme $\pi_\circ(K)$ with the property that $\rho_{\pi_\circ^G}'=(\varphi\circ\rho_{\pi_\circ^G})\otimes\chi^{\rm const}$ if we consider $\chi$ as a character $\chi^{\rm const}$ of $\pi_\circ^G$ (or equivalently as a character of $\pi_0(K(\RR))$). We call $A':=(A_\lieq'(\lambda)_E,\rho_{\pi_\circ^G}')$ an {\em admissible twist} of $A$.
        \end{definition}

        \begin{example}\label{ex:admissiblecharacters}
          Take $G=\res_{F/\QQ}\GL_n$ for a totally real number field $F/\QQ$. Then $\pi_\circ(K)=\res_{F/\QQ}\mu_2$ where $\mu_2$ denotes the multipliciative group scheme of square roots of unity. The only two $\QQ$-rational characters of $\pi_\circ(K)$ are the trivial character ${\bf1}$ and the sign character ${\rm sgn}_F\colon\pi_\circ(K)\to\GL_1$ which is given by the composition
          \[
            {\rm sgn}_F={\rm sgn}_\QQ\circ N_{F/\QQ}
          \]
          and ${\rm sgn}_\QQ\colon\mu_2\to\GL_1$ denotes the nontrivial character of $\mu_2$ over $\QQ$. We remark that both ${\bf1}$ and ${\rm sgn}_F$ descend to characters of $\ZZ[1/2]$.

          Therefore, the admissible weights of a locally algebraic standard module $A=(A_{\lieq'}(\lambda)_E,\rho_{\pi_\circ^G})$ of weight $V$ are given by $V_+:=(V_E,\rho_G,\rho_{\pi_\circ^G})$ and $V_-:=(V_E,\rho_G,\rho_{\pi_\circ^G}\otimes{\rm sgn}_F)$.
        \end{example}

        \subsection{Relative Lie algebra cohomology with locally algebraic coefficients}\label{sec:relativeliealgebracohomology}

        Recall that $S\subseteq G$ denotes the maximal $\QQ$-split torus in the center of $G$. Its Lie algebra is denoted by $\lies=\Lie(S)$. The Lie algebra ${}^0\lieg=Lie({}^0G)$ is complementary to $\lies$ in $\lieg$, i.\,e.\ we have $\lieg=\lies\oplus{}^0\lieg$ as Lie algebras. Then on the category of $(\lieg,K)$-modules we have a notion of relative $(\lieg,\lies,K)$-cohomology theory $H^\bullet(\lieg,\lies,K;-)$ defines as the right derived functor of the functor of $({}^0\lieg,K)$-invariants.

        Over a field $E$ of characteristic $0$ the cohomology $H^\bullet(\lieg,\lies,K;X)$ is computed by the complex
        \begin{equation}
          C^\bullet(\lieg,\lies,K;X)=\Hom_K\left(\bigwedge^\bullet(\lieg/\lies+\liek),X\right).
          \label{eq:relativeliealgebracohomologycomplex}
        \end{equation}
        
        \begin{definition}[Relative Lie algebra cohomology for locally algebraic standard modules]\label{def:locallyalgebraicrelativeLiealgebracohomology}
          Let $A=(A_{\lieq'}(\lambda)_E,\rho_{\pi_\circ^G})$ denote a locally algebraic standard module of weight $V=(V_E,\rho_G,\rho_{\pi_\circ^G})$. Then for every admissible weight $V'=(V_E,\rho_G,\rho_{\pi_\circ^G}\otimes\chi^{\rm const})$ we define relative Lie algebra cohomology of $A$ with coefficients in $V'$ as follows:
          \begin{equation}
            H^\bullet(\lieg,\lies,K;A\otimes V')_E:=
            H^\bullet(\lieg,\lies,K;A_{\lieq'}(\lambda)_E\otimes V_E\otimes\chi)
          \end{equation}
          where on the right hand side we consider $\chi$ as a character of $K$ over $E$ via the pullback along $K\to\pi_\circ(K)$ as usual, i.\,e.\ $A_{\lieq'}(\lambda)_E\otimes V_E\otimes\chi$ carries a natural $(\lieg,K)$-module structure over $E$ depending on $\chi$.

          Moreover, if for an arbitrary locally algebraic standard module $V'$ over $E$ we have
          \[
            H^q(\lieg_\CC,S(\RR)^0K(\RR);A_\CC\otimes V_\CC')=0
          \]
          in a cohomological degree $q$, define
          \[
            H^\bullet(\lieg,\lies,K;A\otimes V')_E:=0.
          \]

          We adopt the same definition over rings $\OO[1/N]$ using the integral analogue of the complex \eqref{eq:relativeliealgebracohomologycomplex} whenever the involved group schemes are \'etale over $\OO[1/N]$ and $\OO[1/N]$-integral models of the characters, locally algebraic $(\lieg,K)$-modules and locally algebraic representations are given.
        \end{definition}

	\begin{proposition}[Complex realization of relative Lie algebra cohomology]\label{prop:complexliealgebracohomology}
          Let $A=(A_{\lieq'}(\lambda)_E,\rho_{\pi_\circ^G})$ denote a locally algebraic standard module of weight $V=(V_E,\rho_G,\rho_{\pi_\circ^G})$ for $E\subseteq\CC$. Then for every admissible weight $V'=(V_E,\rho_G,\rho_{\pi_\circ^G}\otimes\chi^{\rm const})$ we have a canonical monomorphism
          \begin{equation}
            H^\bullet(\lieg,\lies,K;A\otimes V')_E\to
            H^\bullet(\lieg_\CC,S(\RR)^0K(\RR);A_{\lieq'}(\lambda)_\CC\otimes V_\CC')
            \label{eq:Erationalstructureonrelativeliealgebracohomology}
          \end{equation}
          into relative Lie algebra cohomology of the complex realizations of $A$ and $V'$.

          The monomorphism \eqref{eq:Erationalstructureonrelativeliealgebracohomology} defines a canonical $E$-rational stucture on the $(\lieg_\CC,S(\RR)^0K(\RR))$-cohomology of $A_{\lieq'}(\lambda)_\CC\otimes V_\CC'$.
        \end{proposition}

        \begin{proof}
          In light of the complex \eqref{eq:relativeliealgebracohomologycomplex} it suffices to observe that we have a canonical isomorphism
          \[
          \CC\otimes_E\left(A_{\lieq'}(\lambda)_E\otimes V_E\otimes\chi\right)
          \cong
          A_{\lieq'}(\lambda)_\CC\otimes V_\CC',
          \]
          since the contributions of $\rho_{\pi_\circ^G}$ to the two complexifications on the right hand side cancel.
        \end{proof}

        \begin{remark}[Local algebraicity and rationality properties]
          The finite \'etale group scheme $\pi_\circ(K)$ is not constant in general. Therefore, the naive definition of locally algebraic representations as finite-dimensional $(\lieg,K)$-modules which as $(\lieg,K^\circ)$-modules are rational representations yields a notion of rationality which in general is not compatible with the desired notion of rationality on automorphic representations. Consequently the underlying rational structure on $(\lieg,K)$-modules only allows to study those automorphic representations which occur in \eqref{eq:cohomologyofVE} with coefficients in a rational representation $V$ and potentially of its admissible twists in the sense of Definition \ref{def:admissibleweights}.

          From an automorphic perspective, the correct notion of rationality is that induced by the sheaves $\widetilde{V}_E$ associated to locally algebraic representations $V=(V_E,\rho_G,\rho_{\pi_\circ^G})$ of $G$ and therefore locally algebraic standard modules exhibit the correct corresponding rational structure on the level of $(\lieg,K)$-modules, which a fortiori allows us to consider the appropriate integral structures.
          
          We will discuss the link between cohomological automorphic representations and locally algebraic representations in the following section.
	\end{remark}

        \section{Cohomological automorphic representations}\label{sec:cohomologicalrepresentations}

        In this section we discuss the relation between automorphic representations and the cohomology of arithmetic groups with coefficients in a locally algebraic representation. This naturally leads to the notion of a \lq{}cohomological type\rq{} which we introduce in section \ref{sec:cohomologicaltypes} below. In section \ref{sec:cycleandsupportofcohomologicaltype} we associate to every cohomological type an algebraic cycle and a corresponding support. We classify the important class of tempered cohomological types in section \ref{sec:temperedcohomologicaltypes}. In section \ref{sec:automorphiccohomologicaltype} we associate to every cohomological automorphic representation a cohomological type.

	\subsection{Cohomology of arithmetic groups}\label{sec:cohomologyofarithmeticgroups}

	We keep the notation from the previous sections.
	
        \subsubsection{Cohomology with locally algebraic coefficients}

        Fix a field $E$ of characteristic $0$ and a locally algebraic representation $V$ of $G$. We consider the sheaf cohomology
	\begin{equation}
	  H^\bullet(X_G(K_f);\widetilde{V}_E).
          \label{eq:cohomologyofVE}
	\end{equation}

        \begin{remark}[Relation to group cohomology]
        If $K_f$ is sufficiently small in the sense that all arithmetic groups $\Gamma_{K_f,g}$ occuring in the direct sum decomposition \eqref{eq:cohomologydirectsum} are torsion free, we have a canonical isomorphism
        \[
        H^\bullet(\Gamma_{K_f,g}\backslash G(\RR)/S(\RR)^0 K(\RR);\widetilde{V}_E|_{\Gamma_{K_f,g}\backslash G(\RR)/S(\RR)^0 K(\RR)})\cong H^\bullet(\Gamma_{K_f,g};V_E)
        \]
        Hence, \eqref{eq:cohomologyofVE} has an explicit description in terms of the cohomology of the groups $\Gamma_{K_f,g}$ in this case.

        Without assumptions on $K_f$, the Leray spectral sequence \eqref{eq:lerayspectralsequence} degenerates (cf.\ Remark \ref{rmk:leraydegeneration}), which shows that \eqref{eq:cohomologyofVE} always admits an explicit description via the cohomology of the groups $\Gamma_{K_f,g}$.
        \end{remark}

        Each direct summand in \eqref{eq:cohomologydirectsum} and \eqref{eq:cohomologydirectsum0} is finite-dimensional due to the existence of the Borel--Serre compactification \cite{borelserre1973}. The finite failure of strong approximation shows that $X_G(K_f)$ and $X_G(K_f)_0$ have only finitely many connected components (cf.\ \cite{borel1963}), hence the above sheaf cohomology \eqref{eq:cohomologyofVE} is finite-dimensional.
	
	Moreover, the decomposition \eqref{eq:VEdecomposition} descends to a decomposition
	\begin{equation}
	  H^\bullet(X_G(K_f)_0;\widetilde{V}_E)\;=\;
	  \bigoplus_{\chi\in\pi_0(K_\infty,K_f)^\vee}H^\bullet(X_G(K_f);\widetilde{V}_{E,\chi}).
	  \label{eq:cohomologydecomposition}
	\end{equation}
	Here $\pi_0(K_\infty)$ acts canonically on the left hand side via $\pi_0(K_\infty,K_f)$ and each $\chi$-eigenspace corresponds to the corresponding summand on the right hand side. By \eqref{eq:OEdecomposition} this decomposition extends to coefficients in any $\OO[1/N]$ lattice whenever $|\pi_0(K_\infty,K_f)|\in\OO[1/N]^\times$.

        \begin{remark}
          In general, if the order $\#\pi_0(K_\infty)$, or more specifically $\#\pi_0(K_\infty,K_f)$, is not invertible in $\OO[1/N]^\times$, we have for each $\chi$ in \eqref{eq:cohomologydecomposition} a Grothendieck spectral sequence
          \begin{equation}
          H^{p}\left(\pi_0(K_\infty); \chi^{-1}\otimes H^{q}(X_G(K_f)_0;\widetilde{V}_{\OO[1/N]})\right)
          \;\Longrightarrow\;
          H^{p+q}(X_G(K_f);\widetilde{V}_{\OO[1/N],\chi}).
          \label{eq:chispectralsequence}
          \end{equation}
        \end{remark}

        \subsubsection{Hecke action on the cohomology of arithmetic groups}
	
	The set of double cosets $K_fgK_f$ for $g\in G(\bfA_f)$ is an $E$-basis of the Hecke algebra $\mathcal H_E(K_f)$ of level $K_f$ over $E$. Identifying elements of $\mathcal H_E(K_f)$ with bi-$K_f$-invariant locally constant functions
	\[
	f\colon G(\bfA_f)\to E
	\]
	with compact supports, the algebra structure on $\mathcal H_E(K_f)$ is that of a convolution algebra for the right invariant Haar measure on $G(\bfA_f)$ normalized in such way that $\vol(K_f)=1$.
	
	On any $G(\bfA_f)$-representation $M$ over $E$ on which $K_f$ acts trivially we have a canonical action of $\mathcal H_E(K_f)$ given explicitly by
	\[
	K_fgK_f\cdot m=\sum_{i=1}^r g_i\cdot m,
	\]
	whenever $g\in G(\bfA_f)$, $m\in M$ and
	\begin{equation}
		K_fgK_f=\bigsqcup_{i=1}^r g_iK_f.
		\label{eq:Kfdoublecosetdecomposition}
	\end{equation}
	Since group cohomology is the right derived functor of the functor $\Gamma_{K_f,g}$-invariants, the Hecke algebra for $\Gamma_{K_f,g}$ acts on group cohomology with values in $V_E$ and $V_{E,\chi}$ canonically, compatibly with the decomposition \eqref{eq:VEdecomposition}. This action extends to an action of $\mathcal H_E(K_f)$ on the sheaf cohomologies $H^\bullet(\mathscr X_G(K_f);\widetilde{V}_{E,\chi})$ and  $H^\bullet(X_G(K_f);\widetilde{V}_{E,\chi})$ for the orbifold $\mathscr X_G(K_f)$ and its underlying orbit space $X_G(K_f)$ respectively. Likewise, we hhave Hecke actions on the sheaf cohomology of the Galois covers $H^\bullet(\mathscr X_G(K_f)_0;\widetilde{V}_{E})$ and  $H^\bullet(X_G(K_f)_0;\widetilde{V}_{E})$. All these actions are compatible with the Leray spectral sequence \eqref{eq:lerayspectralsequence} as well as with the decomposition \eqref{eq:cohomologydecomposition}.

        They also naturally extend to cohomology of compact support as well as to the boundary $\partial X_G(K_f)$ of the Borel--Serre compactification, observing that the sheaf $\widetilde{V}_E$ naturally extends to $\partial X_G(K_f)$. In particular we have a Hecke equivariant fundamental long exact sequence
        \begin{equation}
          \cdots
          \to H_{\rm c}^\bullet(X_G(K_f);\widetilde{V}_{E})
          \to H^\bullet(X_G(K_f);\widetilde{V}_{E})
          \to H^\bullet(\partial X_G(K_f);\widetilde{V}_{E})
          \to H_{\rm c}^{\bullet+1}(X_G(K_f);\widetilde{V}_{E})
          \to\cdots
        \end{equation}
        and likewise in the orbifold case, which, by the degeneration of the Leray spectral sequence, over fields of characteristic $0$ agrees with the case of classical topological spaces (cf.\ Remark \ref{rmk:leraydegeneration}).

        These actions on sheaf cohomology have the following intrinsic description. To keep the notation simple, we fix from now on a locally algebraic representation $V=(V_E,\rho_G,\rho_{\pi_\circ^G})$ of $G$ and leave the elementary translation of the following to sheaf cohomology of the covering spaces $X_G(K_f)_0$ and the respective orbifolds to the reader.

        To emphasize that the sheaf $\widetilde{V}_{E}$ associated to $V$ (cf.\ Definition \ref{def:locallyalgebraicsheaf}) lives on $X_G(K_f)$, we write $\widetilde{V}_{E,K_f}$. Then for any $g\in G(\bfA_f)$ and any open $U\subseteq X_G(K_f)$, right translation induces a canonical isomorphism
	\[
	t_g\colon\Gamma(U;\widetilde{V}_{E,K_f})\to
	\Gamma(Ug;\widetilde{V}_{E,g^{-1}K_fg}),\quad s\mapsto s((-)\cdot g).
	\]
	To be concise, first compute the $g$-right translation of the pull back of $s$ on $G(\bfA)/S(\RR)^0 K(\RR)^0$ and push the result down to $X_G(g^{-1}K_fg)$.
	
	For the coset representatives $g_i$ on the right and side of \eqref{eq:Kfdoublecosetdecomposition} we consider the compact open $\bigcap_{j=1}^r g_j^{-1}K_fg_j$ and for each $1\leq i\leq r$ the corresponding projection map
	\[
	p_i\colon X_G\left(\bigcap_{j=1}^r g_j^{-1}K_fg_j\right)\to X_G(g_i^{-1}K_fg_i).
	\]
	Then we obtain for any section $s\in\Gamma(U;\widetilde{V}_E)$ a section
	\begin{equation}
		\sum_{i=1}^r t_{g_i}(s)\circ p_i\in
		\Gamma\left(\bigcap_{j=1}^rp_j^{-1}(Ug_j);\widetilde{V}_E\right).
		\label{eq:heckesection}
	\end{equation}
	This section is easily seen to be $K_f$-invariant: For any $h\in p^{-1}\left(\bigcap_{j=1}^rp_j^{-1}(Ug_j)\right)$ and any $k\in K_f$ we have
	\begin{align*}
		\left(\sum_{i=1}^r t_{g_i}(s)\circ p_i\right)(hk)
		&=\sum_{i=1}^r s(hkg_i)\\
		&=\sum_{i=1}^r s(hg_i)\\
		&=\left(\sum_{i=1}^r t_{g_i}(s)\circ p_i\right)(h).
	\end{align*}
	Moreover, since $U\subseteq X_G(K_f)$, its preimage $q^{-1}(U)$ under the projection
	\[
	q\colon G(\bfA)/S(\RR)^0 K(\RR)^0 \to X_G(K_f)
	\]
	is right $K_f$-invariant and by \eqref{eq:Kfdoublecosetdecomposition} any element $h\in q^{-1}(U)g K_f=q^{-1}(U)K_fgK_f$ may be written as
	\[
	h=ug_jk
	\]
	for some index $1\leq j\leq r$, $u\in q^{-1}(U)$ and $k\in K_f$.
	
	Now replacing $u$ by some $uk'$ for $k'\in K_f$ yields another representation
	\[
	h=u'g_{\ell}k''
	\]
	for some new index $1\leq\ell\leq r$, $u'\in q^{-1}(U)$ and $k''\in K_f$. Since all possible indices may be realized this way, we see that
	\[
	q^{-1}(U)g\subseteq \bigcup_{k\in K_f}\left(\bigcap_{j=1}^rq^{-1}(U)g_j\right)k.
	\]
	Hence the section \eqref{eq:heckesection} descends to a unique section
	\[
	K_fgK_f\cdot s\in\Gamma(Ug;\widetilde{V}_E),
	\]
	which describes the Hecke action on sections of the local system $\widetilde{V}_E$. Extending this action to a $\Gamma(X_G(K_f);-)$-acyclic resolution of $\widetilde{V}_E$ (an injective resolution for example), we obtain a canonical action of the Hecke algebra $\mathcal H_E(K_f)$ on $H^\bullet(X_G(K_f);\widetilde{V}_E)$.

        \subsection{Cohomological types}\label{sec:cohomologicaltypes}

        Put
        \[
          {}^0G:=\ker_{\chi\in X_\QQ(G)}\chi^2,
        \]
        which is the kernel of the squares of $\QQ$-rational characters of $G$. Then
        \[
        G(\RR)=S(\RR)^0\cdot{}^0G(\RR)\quad\text{and}\quad S(\RR)^0\cap{}^0G(\RR)=1.
        \]
        The closed subgroup ${}^0G(\RR)$ contains every arithmetic subgroup of $G(\QQ)$ as well as every compact subgroup of $G(\RR)$ (cf.\ \cite{borelserre1973}, Proposition 1.2). Moreover $[G(\RR),G(\RR)]=[G,G](\RR)\subseteq {}^0 G(\RR)$ and for any arithmetic subgroup $\Gamma\subseteq G(\QQ)$ the quotient space $\Gamma\backslash{}^0 G(\RR)$ is of finite invariant volume and in particular the abelian topological group ${}^0G(\RR)/\Gamma\cdot [G,G](\RR)$ is a compact torus (cf.\ loc.\ cit.). As before we write ${}^0\lieg_\CC$ for the complexified Lie algebra of ${}^0G$.

        Recall that a representation of $G(\RR)$ is called Casselman--Wallach if it is Fr\'echet, smooth, of moderate growth and its underlying Harish-Chandra module is admissible and finitely generated. The category of Casselman--Wallach representations of $G(\RR)$ is known to be equivalent to the category of finite length $(\lieg_\CC,K_\infty)$-modules (cf.\ \cite{casselman1989,bernsteinkroetz2014}).

        \begin{definition}[Essentially unitarizable representations]
          We call an irreducible Casselman--Wallach representation $V$ of $G(\RR)$ {\em essentially unitarizable} if its restriction to ${}^0G(\RR)$ is unitarizable. An essentially unitarizable $V$ is tempered if its restriction to $({}^0\lieg_\CC,K_\infty)$ is tempered. Likewise, an irreducible $(\lieg_\CC,K_\infty)$-module $X$ is called {\em essentially (infinitesimally) unitarizable}, if it is infinitesimally unitarizable as $({}^0\lieg_\CC,K_\infty)$-module.
        \end{definition}

        \begin{definition}[Equivalence of pairs]\label{def:equivalenceofpairs}
          Consider pairs $(A,V)$ where $V$ is an irreducible locally algebraic representation of $G$ and $A$ is a locally algebraic standard module of weight $V$, both defined over $E$ with the following property: For every embedding $\sigma\colon E\to\CC$ the Casselman--Wallach globalization of the complexification $A_\CC^\sigma$ of $A$ is essentially unitarizable.

          Two pairs $(A,V)$ and $(A',V')$ are {\em isomorphic}, if $A$ and $A'$ are isomorphic as locally algebraic $(\lieg,K)$-modules over $E$ and $V$ and $V'$ are isomorphic as locally algebraic representations of $G$ over $E$. Two isomorphism classes of pairs $(A,V)$ and $(A',V')$ over $E$ are {\em equivalent}, if $A_\CC\cong A_\CC'$ as $(\lieg_\CC,K_\infty)$-modules.
        \end{definition}

        \begin{definition}[Cohomological types]\label{def:cohomologicaltypesofG}
          A {\em cohomological type} for $G$ over a number field $E/\QQ$ is an equivalence class of isomorphism classes of pairs $(A,V)$ in the sense of Definition \ref{def:equivalenceofpairs}.

          The set of all cohomological types for $G$ over $E$ is denoted by $\cohTypes_E(G)$. For $\sigma\in\Aut(\CC/\QQ)$ we denote by $(A,V)^\sigma:=(A^\sigma,V^\sigma)$ where
          \[
          A^\sigma:=\sigma^{-1}(E)\otimes_{\sigma^{-1},E}A,\quad\text{and}\quad
          V^\sigma:=\sigma^{-1}(E)\otimes_{\sigma^{-1},E}V,
          \]
          the $\sigma$-conjugate pair. This defines an action of $\Aut(\CC/\QQ)$ on cohomological types.
        \end{definition}

        \begin{remark}\label{rmk:unitarizability}
          The unitarizability condition on $A_\CC^\sigma$ may be reformulated as an essential (conjugate) self-duality condition on $V$: If $A_\CC^\sigma=A_\lieq(\lambda)$ and if $\vartheta\colon\lieg_\CC\to\lieg_\CC$ denotes Cartan involution associated to $K(\RR)$, then the unitarizability condition is equivalent to
          \begin{equation}
            \vartheta(\lambda)=\lambda,
            \label{eq:lambdacartan}
          \end{equation}
          under the assumption that $\lambda$ is an extremal weight of $V_\CC$.
        \end{remark}

        For any cohomological type ${\mathrm t}\in\cohTypes_E(G)$ we have for all $[(A,V)]_{\cong}\in\rm{t}$:
        \[
        H^\bullet(\lieg_\CC,S(\RR)^0K(\RR);A_\CC\otimes V_\CC)\neq0.
        \]
        The converse is also true.
        
        \begin{proposition}\label{prop:cohomologicaltypes}
          The elements of of a cohomological type $\rm{t}\in\cohTypes_E(G)$ are in bijection with the isomorphism classes of locally algebraic representations  $V'$ of $G$ over $E$ with the property that
          \[
          H^\bullet(\lieg_\CC,S(\RR)^0K(\RR);A_\CC\otimes V_\CC')\neq0
          \]
          for a (resp.\ any) representative $[(A,V)]_{\cong}\in\rm{t}$.
        \end{proposition}

          \begin{proof}
            Following the discussion in Remark \ref{rmk:cohomologyoflocallyalgebraicstandardmodules}, this is a consequence of Proposition \ref{prop:complexliealgebracohomology} and Theorem 5.6 in \cite{voganzuckerman1984}, see also Theorem 6.4.2 in \cite{hayashijanuszewski}.
          \end{proof}

          \subsection{Cycle and support of a cohomological type}\label{sec:cycleandsupportofcohomologicaltype}

          Recall the notion of the finite \'etale scheme of {\em relative parabolic types} $\rtype(G,K)$ from \cite{hayashijanuszewski}, Definition 5.2.6 and Corollary 5.2.19: $\rtype(G,K)$ is defined as the (\'etale sheaf) quotient of the moduli scheme of $K$-stable parabolic subgroups of $G$ relative to $K$ (cf.\ Definition 5.2.4 of loc.\ cit.) by the action of $K$. Observe that $\Pp_{G}^{K\mathrm{-st}}$ is a closed smooth subscheme of the moduli scheme $\Pp_{G}$ of all parabolic subgroups of $G$ (cf.\ Corollary 5.2.19). Moreover, there is a forgetful map $gt\colon\rtype(G,K)\to\type(G)$ into the \'etale scheme of parabolic types of $G$, induced by the composition $\Pp_{G}^{K\mathrm{-st}}\to\Pp_{G}\to\type(G)$.

          We adopt the following notation of section 5 in \cite{hayashijanuszewski}. Over each relative parabolic type $x\in\rtype(G_E,K_E)$ write $\Pp_{G,x}^{K\mathrm{-st}}$ for the corresponding fiber of the canonical projection $rt\colon\Pp_{G}^{K\mathrm{-st}}\to\rtype(G,K)$. Likewise, we have a fiber $\Pp_{G,gt(x)}$ above $gt(g)\in\type(G)$ and a canonical monomorphism $i_x\colon \Pp_{G,x}^{K\mathrm{-st}}\to\Pp_{G,gt(x)}$.

          Write
          \[
          Z_E(G):=Z_0\left(\rtype(G_E,K_E)\times\Spec Z(\lieg_E)\times \Spec E[\pi_\circ^G(E)]\right)%\Pp_{G}^{K\mathrm{-st}}\right)
          \]
          for the group of algebraic $0$-cycles on the $E$-scheme $\rtype(G_E,K_E)\times\Spec Z(\lieg_E)\times\Spec E[\pi_\circ^G(E)]$. Let ${\rm t}\in\cohTypes_E(G)$ be a cohomological type. We define the {\em cycle} and the {\em (geometric) support} of $\rm t$ respectively as an effective $0$-cycle and its support in the following definition.

        \begin{definition}[Cycle and geometric support of a cohomologial type]\label{def:geometricsupport}
          Let ${\rm t}\in\cohTypes_E(G)$ be a cohomological type. Choose an arbitrary representative $[(A,V)]_\cong\in{\rm t}$. For any $x\in\rtype(G_\CC,K_\CC)$ choose a $Q_x\in\Pp_{G,x}^{K\mathrm{-st}}(\CC)$ with Lie algebra $\lieq_{x}=\liel_x\oplus\lieu_x$ and $\lambda$ denoting an irreducible constituent of the $(\lieq_x,L_x\cap K_\CC)$-action on $H^0(\overline{\lieu}_x;V_\CC)^\vee$, occuring with multiplicity $m_{x,\lambda}^V>0$. Denote by $m_{x,\lambda}^A\geq 0$ the multiplicity of $A_{\lieq_x}(\lambda)$ in $A_\CC$.
          
          Denote by $\rho_{\pi_\circ^G}$ the action of $\pi_\circ^G$ on $A$ and by $\CC[\rho_{\pi_\circ^G}]$ its extension to the group ring. Let $\ann_{Z(\lieg_\CC)}A_{\lieq_x}(\lambda)$ dnote the anniliator of $A_{\lieq_x}(\lambda)$ as a $Z(\lieg_\CC)$-module.

          The {\em cycle} associated to the cohomological type ${\rm t}$ is defined as
          \begin{equation}
            c({\rm t}):=
            \sum_{[(A,V)]_\cong\in{\rm t}}
            \sum_{x\in\rtype(G_\CC,K_\CC)}
            \sum_{\lambda}
            m_{x,\lambda}^V\cdot m_{x,\lambda}^A [x\times V(\ann_{Z(\lieg_\CC)}A_{\lieq_x}(\lambda))\times V(\ker\CC[\rho_{\pi_\circ^G}])],
          \end{equation}
          and the {\em (geometric) support} of ${\rm t}$ is defined as
          \begin{equation}
            \Supp{\rm t}:=\Supp c({\rm t})=
            \bigcup_{[(A,V)]_\cong\in{\rm t}}
            \bigcup_{\begin{subarray}cx,\lambda\colon\\ m_{x,\lambda}^A>0\end{subarray}}
            x\times V(\ann_{Z(\lieg_\CC)}A_{\lieq_x}(\lambda))\times V(\ker\CC[\rho_{\pi_\circ^G}]).
          \end{equation}
        \end{definition}

        \begin{remark}
          Note that the multiplicities $m_{x,\lambda}$ and $n_{x,\lambda}$ are independent of the choice of $Q_x$, but depend on the relative parabolic type $x$.
        \end{remark}

        \begin{remark}
          Theorem \ref{thm:irreduciblelocallyalgebraicrepresentations} and its Corollary \ref{cor:quasisplitlocallyalgebraicrepresentations} show that of $G$ is quasi-split, then $m_{x,\lambda}^V=1$ for all irreducible locally algebraic representations $V$ of $G$ over $E$.
        \end{remark}
        
        \begin{remark}
          Along the same lines, if the hypotheses of Theorem \ref{thm:irreduciblelocallyalgebraicgKmodules} are satisfied, then $m_{x,\lambda}^A=1$. The notion of admissibile models of $K(\RR)$ introduced in section 6 of \cite{januszewskirationality} provides a practical framework to show that the hypotheses of Theorem \ref{thm:irreduciblelocallyalgebraicgKmodules} are verified for certain models in the category of locally algebraic $(\lieg,K)$-modules. A more general approach is provided by combining the theory developed in \cite{hayashijanuszewski} with that of \cite{hayashilinebdl}.
        \end{remark}
        
        \begin{proposition}
          The cycle $c({\rm t})$ for ${\rm t}\in\cohTypes_E(G)$, which in Definition \ref{def:geometricsupport} is a priori defined over $\CC$, is well defined as a cycle over $E$ and $\Supp{\rm t}$ is a $0$-dimensional $E$-subscheme of $\rtype(G_E,K_E)\times\Spec Z(\lieg_E)\times\Spec E[\pi_\circ^G(E)]$.
        \end{proposition}

        \begin{proof}
          By Galois descent, it suffices to observe that $m_{x,\lambda}^V$ and $m_{x,\lambda}^A$ are Galois invariant. Let $\sigma\in\Aut(\CC/E)$ be arbitrary. In order to see this, we use the realization of cohomologically induced standard modules via twisted $\mathcal D$-modules in \cite{hayashijanuszewski}.

          In the notation of Theorem 6.2.1 of \cite{hayashijanuszewski} we have
          \[
          A_{\lieq_{x^\sigma}}(\lambda^\sigma)
          \cong \Gamma(\Pp_{G_\CC,tg(x^\sigma)},(i_{x^\sigma})_+\mathcal L^\sigma),
          \]
          where $\mathcal L$ denotes the $K_\CC$-equivariant line bundle induced by $G\times^{Q_x}\lambda$ on $\Pp_{G_\CC,tg(x)}$. Now by Proposition 1.3.25 combined with the Base Change Theorem for higher direct images (cf.\ Theorem 3.9.2) in \cite{hayashijanuszewski}, we have
          \[
          \Gamma(\Pp_{G_\CC,tg(x^\sigma)},(i_{x^\sigma})_+\mathcal L^\sigma)=
          \Gamma(\Pp_{G_\CC,tg(x)},(i_{x})_+\mathcal L)^\sigma.
          \]
          Another application of Theorem 6.2.1 of loc.\ cit.\ identifies the right hand side with $A_{\lieq_{x}}(\lambda)^\sigma$. The identity
          \begin{align*}
            m_{x^\sigma,\lambda^\sigma}^A
            &=\dim\Hom(A_{\lieq_{x^\sigma}}(\lambda^\sigma),A_\CC)\\
            &=\dim\Hom(A_{\lieq_{x}}(\lambda)^\sigma,A_\CC)\\
            &=\dim\Hom(A_{\lieq_{x}}(\lambda),A_\CC^{\sigma^{-1}})\\
            &=\dim\Hom(A_{\lieq_{x}}(\lambda),A_\CC)\\
            &=m_{x,\lambda}^A
          \end{align*}
          proves the claim for $m_{x,\lambda}^A$. Invoking the Borel--Bott--Weil Theorem, the Galois invariance of $m_{x,\lambda}^V$ follows along the same lines.
        \end{proof}

        \begin{corollary}\label{cor:canonicalcyclefactorization}
          As an element of
          \[
          Z_E(G):=Z_0\left(\rtype(G_E,K_E)\right)\times Z_0\left(\Spec Z(\lieg_E)\right)\times Z_0\left(\Spec E[\pi_\circ^G(E)]\right),
          \]
          the cycle of a cohomological type ${\rm t}\in\cohTypes_E(G)$ is given as
          \begin{equation}
          c({\rm t})=
          \left(
          \sum_{x\in\rtype(G_\CC,K_\CC)}
          \sum_{\lambda}
          m_{x,\lambda}^V\cdot m_{x,\lambda}^A[x]
          ,\;
          [V(\ann_{Z(\lieg)}A)],
          \;
          \sum_{\begin{subarray}c\chi\colon\pi_\circ(K)\to\GL_1\\(A\otimes\chi)_\CC\cong A_\CC\end{subarray}} [V(\ker E[\rho_{\pi_\circ^G}\otimes\chi^{\rm const}])]\right),
          \end{equation}
          where $(A,V)$ is any fixed chosen representative of ${\rm t}$ and on the right hand side $\chi\colon\pi_\circ(K)\to\GL_1$ runs over the $E$-rational characters of $\pi_\circ(K)$ with the property that the complex realization the $\chi$-twisted locally algebraic cohomologically induced standard module obtained from $A$ is isomorphic to $A_\CC$.
        \end{corollary}

        \begin{proof}
          This follows from the definition of the equivalence relation in isomorphism classes of pairs $(A,V)$ and the observation that the data for all three entries vary independently.
        \end{proof}

        \begin{remark}
          If $(A,V)$ is a representative of ${\rm t}$, then $\ann_{Z(\lieg)} A=\ann_{Z(\lieg)} V^\vee$ and the latter only depends on the underlying rational representation of $G$.
        \end{remark}

        \begin{proposition}\label{prop:supportclassification}
          For two cohomological types ${\rm t},{\rm t'}\in\cohTypes_E(G)$ we have
          \[
          {\rm t}={\rm t'}\quad\Leftrightarrow\quad \Supp{\rm t}\cap\Supp{\rm t}'\neq\emptyset.
          \]
          In particular the map
          \begin{equation}
            c\colon\cohTypes_E(G)\to Z_E(G),\quad {\rm t}\mapsto c({\rm t})
          \end{equation}
          is injective.
        \end{proposition}

        \begin{proof}
          By Proposition \ref{prop:locallyalgebraicfieldofdefinition} combined with Theorem \ref{thm:irreduciblelocallyalgebraicrepresentations} we conclude with Remark \ref{rmk:locallyalgebraicequivalence} that the maximal ideal $\ann_{Z(\lieg)}A$ together with $\rho_{\pi_\circ^G}$ supporting ${\rm t}$ in the sense of Corollary \ref{cor:canonicalcyclefactorization} characterizes the isomorphism class of the locally algebraic representation $V$ over $E$ for which $[(A,V)]_\cong\in{\rm t}$ for some $A$ uniquely (we do not assume $A$ to be given beforehand).
          
          Therefore, we may assume without loss of generality that $\Supp{\rm t}$ and $\Supp{\rm t}'$ have representatives $(A,V)$ and $(A',V')$ with $V=V'$, i.\,e.
          \[
          \ann_{Z(\lieg)}A=\ann_{Z(\lieg)}A'\quad\text{and}\quad \rho_{\pi_\circ^G}=\rho_{\pi_\circ^G}'.
          \]
          If $\Supp{\rm t}\cap\Supp{\rm t}'\neq\emptyset$, then this implies in light of Corollary \ref{cor:canonicalcyclefactorization} that there exists an $x\in\rtype(G_\CC,K_\CC)$ contributing both to $\Supp{\rm t}$ and $\Supp{\rm t}'$. This shows that $A_{\lieq_x}(\lambda)$ occurs both in $A_\CC$ and $A_\CC'$, which implies that the underlying algebraic standard modules over $E$ are isomorphic up to admissible twists. Hence $A_\CC=A_\CC'$ and ${\rm t}={\rm t}'$ follows as claimed.
        \end{proof}
        
        \begin{definition}\label{def:finecohomologicaltypes}
          Let $(V,\rho_G)$ denote an irreducible rational representation of $G$ over $E$. We say that $V$ {\em is the rational weight} of a cohomological type ${\mathrm t}\in\cohTypes_E(G)$, if
          \begin{equation}
            \Supp{\rm t}\subseteq
            \rtype(G_E,K_E)\times \{[V(\ann V)]\}\times Z_0(\Spec E[\pi_\circ^G(E)])\}.
            \label{eq:cohomologicaltypeVsupportcondition}
          \end{equation}
          We denote by
          \[
            \cohTypes_E(G, V)\subseteq
            \cohTypes_E(G)
          \]
          the subset of cohomological types of $G$ of rational weight $V$.

          For any closed subscheme $\xi\in\rtype(G_E,K_E)$ we write
          \[
            \cohTypes_E^{\bullet}(G, \xi)\subseteq
            \cohTypes_E^{\bullet}(G)
          \]
          for the subset of cohomological types ${\rm t}$, whose support satisfies
          \begin{equation}
            \Supp{\rm t}\subseteq\left(\rtype(G,K)\times \xi\times \Spec E[\pi_\circ^G(E)]\right).
            \label{eq:cohomologicaltypesupportcondition}
          \end{equation}
          By abuse of language we say that ${\rm t}$ {\em has (geometric) support in} $\xi$ if \eqref{eq:cohomologicaltypesupportcondition} is satisfied.

          Finally, given $V$ and $\xi$, put
          \[
            \cohTypes_E(G, \xi, V):=\cohTypes_E^{\bullet}(G, \xi)\cap\cohTypes_E(G, V).
          \]
        \end{definition}

        \begin{remark}
          Condition \eqref{eq:cohomologicaltypeVsupportcondition} is equivalent to the existence of an isomorphism class $[(A',V')]_{\cong}\in\rm{t}$ with $V'\cong(V,\rho_G,\rho_{\pi_\circ^G})$ as locally algebraic representations for some $\rho_{\pi_\circ^G}\colon\pi_\circ^G\to\End_E(V)$. In light of the definition of the geometric support in Definition \ref{def:geometricsupport}, condition \eqref{eq:cohomologicaltypesupportcondition} may be rephrased as a similar explicit condition on $A'$.
        \end{remark}
        
        \subsection{Tempered cohomological types}\label{sec:temperedcohomologicaltypes}

          Analogously to the moduli scheme of $K$-stable parabolic subgroups of $G$ we have a moduli scheme $\mathcal{B}_{G}^{K-\mathrm{st}}$ of $K$-stable Borel subgroups of $G$ and a corresponding \'etale scheme $\rtype_\emptyset(G,K)\subseteq\rtype(G,K)$ of relative Borel types, which is the (\'etale sheaf) quotient of $\mathcal{B}_{G}^{K-\mathrm{st}}$ by the action of $K$.

          \begin{definition}[Tempered cohomological types]
          A cohomological type ${\mathrm t}\in\cohTypes_E(G)$ is {\em tempered} if there exists an isomorphism class $[(A,V)]_{\cong}\in\rm{t}$ with the following property: There exists an embedding $E\to\CC$ with the property that the smooth Fr\'echet globalization of the $(\lieg_\CC,K_\infty)$-module $A$ is essentially unitarizable and tempered. We denote by
          \[
            \cohTypes_E^{\rm temp}(G)\subseteq
            \cohTypes_E(G)
          \]
          the subset of tempered cohomological types of $G$. Define the subsets
          \[
            \cohTypes_E^{\rm temp}(G,V),\,\cohTypes_E^{\rm temp}(G,\xi),\,\cohTypes_E^{\rm temp}(G,\xi,V),
          \]
          as the corresponding intersections with the subsets defined in Definition \ref{def:finecohomologicaltypes}.
        \end{definition}

        We have the following classification result.

        \begin{theorem}[Classification of tempered cohomological types]\label{thm:temperedcohomologicaltypes}
          For a cohomological type ${\mathrm t}\in\cohTypes_E(G)$ the following are equivalent:
          \begin{enumerate}
	    \renewcommand{\labelenumi}{(\roman{enumi})}
              \item $\rm t$ is tempered.
              \item $\rm t$ has support in $\rtype_\emptyset(G_E,K_E)$, i.\,e.\ 
                \[
                {\rm t}\in\cohTypes_E(G,\rtype_\emptyset(G_E,K_E)).
                %Z_0(\rtype_0(G_E,K_E)) \times Z_0\left(\Spec Z(\lieg_E)\right)\times Z_0\left(\Spec E[\pi_\circ^G(E)]\right).
                \]
          \end{enumerate}
        \end{theorem}

        \begin{proof}
          The classification of tempered representations in Theorem 14.2 in \cite{knappzuckerman1982} together with the Vogan--Zuckerman classification of cohomological $(\lieg,K)$-modules and its reformulation in terms of the Langlands classification in \cite{voganzuckerman1984} shows that a cohomologically induced standard module $A_\CC$ over $\CC$ is tempered if and only if it is cohomologically induced from a $\theta$-stable Borel subalgebra $\lieq_\CC\subseteq\lieg_\CC$. In light of Corollary \ref{cor:canonicalcyclefactorization} and Proposition \ref{prop:supportclassification} this shows the equivalence of (i) and (ii).
        \end{proof}

        \begin{corollary}\label{cor:temperedcohomologicaltypes}
          We have
          \[
          \cohTypes_E^{\rm temp}(G)=
          \cohTypes_E(G,\rtype_\emptyset(G_E,K_E))
          \]
          and we have a partition
          \begin{equation}
            \cohTypes_E^{\rm temp}(G)=
            \bigsqcup_{\begin{subarray}cx\in\rtype_\emptyset(G,K)(E)\\V\end{subarray}}\cohTypes_E(G,x,V),
            \label{eq:cohtypespartition}
          \end{equation}
          where $V$ ranges over representatives of all irreducible rational representations of $G$ and we consider $x\in\rtype_\emptyset(G,K)(E)$ as a closed points of the finite \'etale scheme $\rtype_\emptyset(G_E,K_E)$ over $E$.
        \end{corollary}

        \begin{proof}
          Theorem \ref{thm:temperedcohomologicaltypes} together with Theorem 5.6 in \cite{voganzuckerman1984} and the observation that for
          \[
          x,x'\in\rtype_\emptyset(G,K)(E)
          \]
          $x\neq x'$ the all corresponding cohomologically induced standard modules are non-isomorphic, shows the claim.
        \end{proof}

        \begin{proposition}[{Li-Schwermer, \cite{lischwermer}}]\label{prop:lischwermer}
          For any irreducible rational representation $(V,\rho_G)$ of $G$ over $E$ which over a splitting field of $G$ contains an absolutely irreducible representation which has a regular extremal weight, we have
          \begin{equation}
          \cohTypes_E(G,V)=
          \cohTypes_E^{\rm temp}(G,V)=
          \cohTypes_E(G,\rtype_\emptyset(G_E,K_E),V).
          \label{eq:lischwermer}
          \end{equation}
        \end{proposition}

        \begin{proof}
          Assume first that $V$ is an absolutely irreducible rational representation $V$ with a regular extremal weight. Then any extremal weight of $V$ is regular and the argument of the proof of Proposition 4.2 in \cite{lischwermer} shows that only for Borel subalgebras $\lieq_\CC\subseteq\lieg_\CC$ we may consider such an extremal weight as a character of $\lieq_\CC$. Therefore, by the definition of cohomological types (cf.\ Definition \ref{def:cohomologicaltypesofG}), $\cohTypes_E(G,V)=\cohTypes_E(G,\rtype_\emptyset(G_E,K_E),V)$ follows. The case of a general irreducible but not necessarily absolutely irreducible $V$ reduces to this case. The second identity in \eqref{eq:lischwermer} follows by Corollary \ref{cor:temperedcohomologicaltypes}.
        \end{proof}

	\subsection{Cohomological automorphic representations}\label{sec:automorphiccohomologicaltype}
	
	Let $\pi$ denote an automorphic representation of $G(\bfA)$. Then $\pi=\pi_\infty\otimes\pi_f$ for admissible representations $\pi_\infty$ of $G(\RR)$ and $\pi_f$ of $G(\bfA_f)$. We write $\pi_\infty^{(K_\infty)}$ for the underlying Harish-Chandra module of $K_\infty=K(\RR)$-finite vectors in $\pi_\infty$. For each compact open $K_f\subseteq G(\bfA_f)$ as before, the space $\pi_f^{K_f}$ of $K_f$-invariants is canonically an $\mathcal H_{K_f}(\CC)$-module. Proposition \ref{prop:cohomologicaltypes} motivates the following

        \begin{definition}[Cohomological automorphic representations]\label{def:cohomologicalrepresentation}
          Let $V=(V_E,\rho_G,\rho_{\pi_\circ^G})$ denote a locally algebraic representation of $G$ for $E\subseteq \CC$. We call a automorphic representation $\pi$ of $G(\Adeles)$ {\em cohomological of weight $V$} if there is a embedding
          \begin{equation}
            \pi_\infty^{(K_\infty)}\to A_\CC
            \label{eq:cohomologicalcondition}
          \end{equation}
          of $(\lieg_\CC,K_\infty)$-modules where $A$ is the complex realization of a locally algebraic cohomologically induced standard module $A=(A_E,\rho_{\pi_\circ^G})$ of weight $V$ over $E$ in the sense of Definition \ref{def:locallyalgebraicstandardmodule}.
        \end{definition}

        \begin{definition}[Cohomological type of a cohomological automorphic representation]\label{def:cohomologicaltypeofrepresentation}
          Let $\pi$ denote a cohomological automorphic representation of $G(\bfA)$ in the sense of definition \ref{def:cohomologicalrepresentation}. The cohomological type
          \[
            \cohType_E(\pi):=\left\{[(A,V)]_\cong\mid[(A,V)]_\cong\;\textrm{with}\;\pi_\infty^{(K_\infty)}\subseteq A_\CC\right\}\in\cohType_E(G)
            \]
            is called the {\em cohomological type} of $\pi$ over $E$.
        \end{definition}

        \begin{remark}
          If $\pi_\infty^{(K_\infty)}$ is not defined over $E$, then $\cohType_E(\pi)=\{[(A,V)]_\cong\}$ with $A_\CC$ a finite direct sum of absolutely irreducible cohomologically induced standard modules. Note that the cardinality of $\cohType_E(\pi)$ is independent of $E$ (cf.\ Proposition \ref{prop:locallyalgebraicfieldofdefinition} and Remark \ref{rmk:absolutelyirreduciblelocallyalgebraicgKmodule}).
        \end{remark}

        \begin{example}
          We refer to Theorem \ref{thm:cohomologicaltypesofGLn} below for the classification of cohomological types of regular algebraic cuspidal automorphic representations of $G=\res_{F/\QQ}\GL_n$ for $F$ totally real or CM over their respective fields of definition.
        \end{example}

        \begin{theorem}[{Hochschild--Mostov \cite[Theorem 6.1]{hochschildmostov}}]
          We have
          \begin{flalign*}
	  &H^\bullet(\lieg_\CC,S(\RR)^0 K(\RR);\pi_\infty\otimes V_\CC)\\
	  &=H^\bullet(\lieg_\CC,S(\RR)^0 K(\RR);\pi_\infty^{(K_\infty)}\otimes V_\CC)\\
	  &=H^\bullet(\lieg_\CC,\lies_\CC, K(\RR);\pi_\infty^{(K_\infty)}\otimes V_\CC),
          \end{flalign*}
          canonically.
        \end{theorem}

        Together with the observation in Remark \ref{rmk:cohomologyoflocallyalgebraicstandardmodules} we deduce
        \begin{corollary}\label{cor:cohomological}
          If $\pi$ is cohomological of weight $V$, then
          \begin{equation}
	    H^\bullet(\lieg_\CC,S(\RR)^0 K(\RR);\pi\otimes V_\CC)\neq 0.
            \label{eq:relativeliealgebracohomologynonvanishing}
          \end{equation}
          In particular, if $\pi$ is cuspidal, then the canonical $G(\Adeles_f)$-equivariant monomorphism
          \begin{flalign}
            &H^\bullet(\lieg_\CC,S(\RR)^0 K(\RR);\pi_\infty^{(K_\infty)}\otimes V_\CC)\otimes\pi_f\nonumber\\
            &\cong
            H^\bullet(\lieg_\CC,S(\RR)^0 K(\RR);\pi\otimes V_\CC)\label{eq:cohomologicalembedding}\\
            &\to
            \varinjlim_{K_f}H^\bullet(X_G(K_f); \widetilde{V}_\CC).\nonumber
          \end{flalign}
          realizes $\pi_f$ in the cohomology of arithmetic groups.
        \end{corollary}

        \begin{remark}
          If $\pi$ is cohomological of weight $V$ and if $\widetilde{V}_E$ denotes the sheaf associated to $V$, then for some (resp.\ all) sufficiently small compact open $K_f$ for which $\pi_f^{K_f}\neq 0$, the Hecke module $\pi_f^{K_f}$ occurs in $H^\bullet(X_G(K_f);\widetilde{V}_E)$.

        However, in general an automorphic representation $\pi$ is {\em not} determined by the isomorphism class of its finite part $\pi_f$. There may exist automorphic representations $\pi'$ of $G(\bfA)$ not isomorphic to $\pi$ but isomorphic to $\pi$ at all finite places, i.\,e.\ with the property $\pi_f\cong\pi_f'$.

        Therefore, for an automorphic representation $\pi$ being cohomological with coefficients in the locally algebraic representation $V$ implies that, but in general is not equivalent to, saying that $\pi_f$ occurs in the cohomology of $\widetilde{V}_E$ for $E$ sufficiently large.

        For groups with sufficiently strong multiplicity one properties such as $\GL_n$, this equivalence still holds because $\pi=\pi'$ is automatic: If $G=\res_{F/\QQ}\GL_n$, then for not necessarily cuspidal but isobaric automorphic representations $\pi_1,\pi_2$ of $G(\bfA)$ the condition $0\neq\pi_{1,f}^{K_f}\cong\pi_{2,f}^{K_f}$ implies $\pi_1\cong\pi_2$ by \cite{jacquetshalika1981,jacquetshalika1981.2}. In particular, in the case of $\GL_n$, $\pi_{\infty}$ is uniquely determined by $\pi_f^{K_f}$ whenever this $\mathcal H_{K_f}(\CC)$-module is non-zero.

        Therefore it is not surprising that our results in the case $\GL_n$ are more complete than for general $G$. Consequently we put special emphasis on the case of the general linear group and discuss it in more detail below.
        \end{remark}

        \begin{remark}[Refined rationality considerations]
	  Various subspaces of $H^\bullet(X_G(K_f);\widetilde{V})$ may be defined: inner cohomology via the distinguished triangle associated to restriction to the boundary of Borel--Serre compactification, cuspidal cohomology which corresponds to the sum of the canonical images of the inclusion \eqref{eq:cohomologicalembedding} for all $\pi'$ cuspidal cohomological of weight $V$. More general subspaces and subquotients may be defined via the support of constant terms. However, except for inner cohomology, it is not clear in general if these subspaces are rational over a field of definition of $V$. For discussions of related rationality questions we refer to \cite{clozel1990} in the case of $\GL_n$ and to section 8 in \cite{januszewskirationality} for general reductive $G$.
	
	Nonetheless, for every cohomological automorphic representation $\pi$ of $G$, the finite part $\pi_f$ admits a model over a number field. The archimedean component $\pi_\infty$ also admits a model over a number field: The underlying irreducible locally algebraic $(\mathfrak{g},K)$-module admits a model over a number field. Overall, we obtain a model of the irreducible $(\mathfrak{g},K\times\pi_\circ^G)\times G(\bfA_f)$-module inside $\pi$ over a number field, which for certain groups $G$ is strictly larger than the field of rationality of $\pi$ due to the existence of Brauer obstructions.
        \end{remark}

        \section{Integral models of automorphic representations}\label{sec:integralmodels}

        In this section we build on the results of the previous sections to construct global $1/N$-integral structures on automorphic representation and on global spaces of automorphic cusp forms. We keep the notation and the setting from the previous sections. In particular, as in section \ref{sec:cohomologicalrepresentations}, $G$ is a connected linear reductive group over $\QQ$ and $V=(V_E,\rho_G,\rho_{\pi_\circ^G})$ denotes a locally algebraic representation over $E$ and $V_{\OO[1/N]}\subseteq E$ is a $K_f$-stable $\OO[1/N]$-lattice in $V$.

	\subsection{Integral structures on the cohomology of arithmetic groups}\label{sec:integralcohomology}
	
	We already observed that
	\[
	\widetilde{V}_{\OO[1/N]}(U)=\{s\in\Gamma(U;\widetilde{V}_{E})\mid s(g_\infty S(\RR)^0 K(\RR) , g_fK_f)\in g_fV_{\OO[1/N]}\}
	\]
	is a subsheaf of $\OO[1/N]$-modules of $\widetilde{V}_{E}$. Each Hecke operator $K_fgK_f$ for $g,h\in G(\bfA_f)$ sends sections of $\widetilde{hV}_{\OO[1/N]}$ to sections of $\widetilde{ghV}_{\OO[1/N]}$. In particular, we obtain an $\OO[1/N]$-linear map
	\[
	K_fgK_f\cdot(-)\colon H^\bullet(X_G(hK_fh^{-1});\widetilde{hV}_{\OO[1/N]})\to
	H^\bullet(X_G(ghK_fh^{-1}g^{-1});\widetilde{ghV}_{\OO[1/N]}).
	\]
	Depending on the context, there are two ways to obtain an action of a suitable $1/N$-integral Hecke algebra on $H^\bullet(X_G(K_f);\widetilde{V}_{\OO[1/N]})$.

        \smallskip
	Firstly, one may consider the algebra $\mathcal H(K_f; V_{\OO[1/N]})$ of all operators of the form
	\[
	\sum_{i=1}^\ell a_i\cdot K_f g_i K_f\in\mathcal H_E(K_f)
	\]
	where $a_i\in E$, $g_i\in G(\bfA_f)$ with the property that
	\[
	\sum_{i=1}^\ell a_i\cdot g_iV_{\OO[1/N]}\subseteq V_{\OO[1/N]}.
	\]
	This algebra then acts on $H^\bullet(X_G(K_f);\widetilde{V}_{\OO[1/N]})$ and satisfies
	\[
	E\otimes_{\OO[1/N]} \mathcal H(K_f;\widetilde{V}_{\OO[1/N]})=\mathcal H_E(K_f).
	\]
	
        \smallskip
	The second approach is of relevance in the case of $p$-adic coefficients, i.\,e.\ consider for a place $v$ lying above a rational prime $p$ the valuation ring $\OO_{(v)}=\varinjlim_{v\nmid N}\OO[1/N]\subseteq E$ where $N$ runs through the non-zero elements of $\OO$ not divisble by $v$. Remark that our previous statements over $\OO[1/N]$ naturally extend to statements over $\OO_{(v)}$ and even to its $v$-adic completion. In the case of a discrete valuation ring $\OO_{(v)}$ there are specific choices of lattices $V_{\OO_{(v)}}$: Assume $G$ is quasi-split over $E$ and fix a Borel $B\subseteq G$ with Levi decomposition $B=TU$ (or more generally a parabolic subgroup, possibly dropping the assumption that $G$ is quasi-split). If $V_E$ is absolutely irreducible, then $V_{E}$ is generated by a $B$-highest weight vector $v_0\in V_{E}$ of weight $\lambda$ say. The vector $v_0$ generates a lattice $V_{\OO_{(v)}}$ with the property that for any $B$-dominant coroot $\alpha:\Gm\to T$ and any prime element $\varpi$ in $\OO_{(v)}$,
	\[
	\lambda^\vee(\varpi)\cdot\rho_V(\alpha(\varpi))(V_{\OO_{(v)}})\subseteq V_{\OO_{(v)}}.
	\]
	Here $\lambda^\vee$ denotes the highest weight of the contragredient of $V$. Then if $p$ is the residual characteristic of $\varpi$ and if $K_f=K_f^{(p)}\times I_p$ with $I_p\subseteq G(\ZZ_p)$ modulo $p$ equal to a subgroup in $B(\ZZ/p\ZZ)$ containing $U(\ZZ/p\ZZ)$, then the normalized Hecke operator
	$\lambda^\vee(p)K_f\alpha(p)K_f$ canonically acts on $H^\bullet(X_G(K_f);\widetilde{V}_{\OO_{(v)}})$. In fact, this operator is contained in the previously constructed integral Hecke algebra $\mathcal H(K_f;V_{\OO_{(v)}})$ when we extend the above construction to $\OO_{(v)}$.

        \smallskip
	Back in the general situation, assume $\pi$ is cohomological of weight $V=(V_E,\rho_G,\rho_{\pi_\circ^G})$ where $E/\QQ$ is a number field and $v$ is a place above $p$. By enlarging $E$ if necessary, we may assume that $\pi_f$ is defined over $E$ as well. Fix the canonical embedding
	\[
	i^{K_f}\colon \pi_f^{K_f}\to H^\bullet(X_G(K_f);\widetilde{V}_{\CC}),
	\]
	and intersect its image with the image of $H^\bullet(X_G(K_f);\widetilde{V}_{\OO[1/N]})$ (or over $\OO_{(v)}$ respectively) in
        \[
        \CC\otimes_E H^\bullet(X_G(K_f);\widetilde{V}_{E})=H^\bullet(X_G(K_f);\widetilde{V}_{\CC})
        \]
        to obtain a model $\pi^{K_f}_{f,\OO[1/N]}$ of $\pi_f^{K_f}$ over $\OO[1/N]$ as $\mathcal H(K_f;V_{\OO[1/N]})$-module, which comes with an isomorphism
	\[
	\varphi_{f}^{K_f}\colon\CC\otimes_{\OO[1/N]}\pi^{K_f}_{f,\OO[1/N]}\to\pi_f^{K_f}.
	\]
	Letting $K_f$ run through a neighborhood basis of the identity in $G(\Adeles_f)$ and choosing $i^{K_f}$ compatibly, we obtain with
	\[
	\pi_{f,\OO[1/N]}:=\varinjlim_{K_f}\pi^{K_f}_{f,\OO[1/N]}
	\]
	and
	\begin{equation}
		\varphi_{f}:=\varinjlim_{K_f}\varphi_{f}^{K_f}\colon\CC\otimes_{\OO[1/N]}\pi_{f,\OO[1/N]}\to\pi_f
		\label{eq:varphif}
	\end{equation}
	a model of the $G(\Adeles_f)$-module $\pi_f$ over $\OO[1/N]$ in the following sense: $\pi_{f,\OO[1/N]}$ admits a canonical action of the Hecke algebra
	\[
	\mathcal H(V_{\OO[1/N]}):=\varinjlim_{K_f}\mathcal H(K_f;V_{\OO[1/N]})
	\]
	which is an $\OO[1/N]$-model of the Hecke algebra $\mathcal H(G(\Adeles_f))$ of locally constant $\CC$-valued functions on $G(\Adeles_f)$ with compact supports. Therefore, $\varphi_{f}$ in \eqref{eq:varphif} is an isomorphism of $G(\Adeles_f)$-modules.

        \begin{remark}
        We emphasize that the $\OO[1/N]$-Hecke algebra constructed commutes with change of coeffients in the sense that whenever $N$ is invertible in $\OO[1/N']$, then
        \[
	\mathcal H(V_{\OO[1/N']})=\OO[1/N']\otimes_{\OO[1/N]}\mathcal H(V_{\OO[1/N]}).
        \]
        \end{remark}

        The previous remark justifies the following

	\begin{definition}[Hecke algebra for arbitrary coefficients]
	For any flat $\OO$-algebra $A$ put
	\[
	\mathcal H(V_A):=A\otimes_{\OO}\mathcal H(V_{\OO}).
	\]
        \end{definition}

        \begin{remark}
          For the purpose at hand, we do not consider torsion coefficients.
        \end{remark}

	\subsection{Integral structures on automorphic representations}\label{sec:automorphicintegralstructures}

        Assume we have an $1/N$-integral model $\pi_{f,\OO[1/N]}$ of $\pi_f$ as constructed in the previous section. We augment this model with a $1/N$-integral model of $\pi_\infty$ for suitable $N$ to obtain a global $1/N$-integral model as follows. We assume that the automorphic representation $\pi$ is cohomological of weight $V=(V_E,\rho_G,\rho_{\pi_\circ^G})$ (cf.\ Definition \ref{def:cohomologicalrepresentation}) which we may and do assume to be defined over the number field $E/\QQ$ with respect to a fixed embedding $E\to\CC$ as well. Recall that we have a the complex representation $(V_\CC,\rho_\CC)$ of the Lie group $G(\RR)$ associated to $V$ from Definition \ref{def:locallyalgebraiccomplex} as well as a sheaf $\widetilde{V}_E$ of $E$-vector spaces on $X_G(K_f)$ for any compact open subgroup $K_f\subseteq G(\Adeles_f)$.

        \subsubsection{Construction of $1/2$-integral models}
        
	For the sake of concreteness, assume that the $\QQ$-groups $G$ and $K$ are the base changes of standard $\ZZ[1/2]$-forms $G_{\ZZ[1/2]}$ and $K_{\ZZ[1/2]}$ associated to the symmetric pair $(G(\RR),K(\RR))$ as in Corollary 6.2.18 in \cite{hayashijanuszewski} and that the data necessary to realize $\pi_\infty^{(K_\infty)}$ as a cohomologically induced standard module $A_{\lieq'}(\lambda)$ as in \eqref{eq:cohomologicalarchimedeancomponent} satisfy the hypotheses of the corollary. Then $\pi_\infty^{(K_\infty)}$ admits a model over $\ZZ[1/2]$, i.\,e.\ we have a $(\lieg_{\ZZ[1/2]},K_{\ZZ[1/2]})$-module $\pi_{\infty,\ZZ[1/2]}$ together with an isomorphism
	\[
	\varphi_\infty\colon\CC\otimes_{\ZZ[1/2]}\pi_{\infty,\ZZ[1/2]}\to\pi_\infty^{(K_\infty)}
	\]
	of Harish-Chandra modules.
	
	Assume that $\pi_f$ is defined over a number field $E/\QQ$ and write $\OO\subseteq E$ for its ring of integers as before. Recall that we constructed in \eqref{eq:varphif} an $\OO$-model $\pi_{f,\OO}$ of the finite part $\pi_f$ of $\pi$. Put
	\[
	\pi_{\OO[1/2]}:=\pi_{\infty,\ZZ[1/2]}\otimes_{\ZZ}\pi_{f,\OO},
	\]
	which admits a canonical structure of $(\lieg_{\OO[1/2]},K_{\OO[1/2]}\times\pi_\circ^G)\times\mathcal H(V_{\OO[1/2]})$-module and its complexification in the sense of Definition \ref{def:locallyalgebraiccomplexmodule} comes with an isomorphism
	\begin{equation}
		\varphi:=\varphi_\infty\otimes\varphi_f\colon\CC\otimes_{\OO[1/2]}\pi_{\OO[1/2]}\to\pi^{(K_\infty)}
		\label{eq:varphiglobal}
	\end{equation}
	of complex $(\lieg_{\CC},K(\RR))\times G(\Adeles_f)$-modules.
	
	This proves
	\begin{theorem}\label{thm:globalarithmeticmodels}
		Let $(G,K)$ be the symmetric pair consisting of the standard $\ZZ\left[1/2\right]$-form of of a classical connected Lie group and its maximal compact subgroup attached to the involution $\theta$ of $G$ in \cite{hayashilinebdl} 3.3. Assume that $\pi$ is an automorphic representation of $G(\Adeles)$ which is cohomological of weight $V=(V_E,\rho_G,\rho_{\pi_\circ^G})$ such that $\pi_f$ is defined over a number field $E/\QQ$. Assume furthermore that the cohomological induction data for $\pi_\infty^{(K_\infty)}$ satisfy the hypotheses in Corollary 6.2.18 in \cite{hayashijanuszewski}. Then $\pi$ admits a model $(\pi_{\OO[1/2]},\varphi)$ over $\OO[1/2]$ in the above sense:
		
		$\pi_{\OO[1/2]}$ is a $(\lieg_{\OO[1/2]},K_{\OO[1/2]}\times\pi_\circ^G)\times \mathcal H(V_{\OO[1/2]})$-module and its complexification comes with an isomorphism $\varphi$ as in \eqref{eq:varphiglobal} identifying $\CC\otimes_{\OO[1/2]}\pi_{\OO[1/2]}$ with the irreducible $(\lieg_\CC,K_\infty)\times G(\Adeles_f)$-module $\pi^{(K_\infty)}\subseteq\pi$ of $K_\infty$-finite vectors in $\pi$.
	\end{theorem}

	\subsection{Integral structures on automorphic representations of $\GL_n$}\label{sec:integralstructuresforgln}

        In the case of $\GL_n$, the availability of sufficiently strong multiplicity one allows us to extend the result from Theorem \ref{thm:globalarithmeticmodels} to show that for $\GL_n$ over a totally real or a CM field, the field of rationality of $\pi_f$ is a global field of definition. In section \ref{sec:canonicalintegralstructures} below, we construct canonical $1/N$-integral structures on global spaces of cusp forms.

        \subsubsection{Algebraic representations of $\GL_n$}
        
	In the case of $\GL_n$ over a number field $F$ the class of so-called algebraic cuspidal automorphic representations which conjecturally correspond to irreducible motives over was characterized by Clozel in \cite{clozel1990}. The subclass of regular algebraic cuspidal representations denotes those $\pi$ with regular infinitesimal character and agrees with the class of cohomological cuspidal automorphic representations in the sense of Definition \ref{def:cohomologicalrepresentation} for $G=\res_{F/\QQ}\GL_n$. In this case $\pi_\infty$ is uniquely determined by the weight $V$ of $\pi$. We refer to section 3 of \cite{clozel1990} for a detailed discussion of $\pi_\infty$ and its realtive Lie algebra cohomology in terms of the archimedean Langlands classification for $\GL_n$. For technical reasons we restrict our discussion to number fields $F$ which are totally real or a CM field.

        We identify automorphic representation of $\GL_n(\Adeles_F)$ with automorphic representations of $G(\Adeles)$. All properties and objects such as weights of cohomological automorphic representations are understood with respect to the reductive group $G$ over $\QQ$.

        \begin{proposition}[Effect of finite order twists]\label{prop:finiteorderhecketwists}
          Assume $F$ is a number field which is totally real or CM. Consider a regular algebraic cuspidal automorphic representation $\pi$ of $\GL_n(\bfA_F)$ of weight $V=(V_E,\rho_G,\rho_{\pi_\circ^G})$ and let $\chi:F^\times\backslash\Adeles_F^\times\to\CC^\times$ denote a finite order Hecke character. Then we may consider the archimedean component $\chi_\infty$ of $\chi$ as a character of $\pi_0(K(\RR))$ and the twisted representation $\pi\otimes\chi$ will be cohomological of weight $(V_E,\rho_G,\rho_{\pi_\circ^G}\otimes\chi_\infty^{-1})$.
        \end{proposition}

        \begin{proof}
          This is an elementary calculation.
        \end{proof}
        
        \begin{proposition}[Twists with Hecke characters of type $A_0$]
          Let $\chi:F^\times\backslash\Adeles_F^\times\to\CC^\times$ denote an algebraic Hecke character in the sense of A.\ Weil. This is equivalent to saying that there is a rational character $\chi^{\rm rat}\colon\res_{F/\QQ}\GL_1\to\GL_1$ and a finite order Hecke character $\chi^{\rm fin}\colon (\RR\otimes_\QQ F)^\times\to\CC^\times$ with the property that $\chi_\infty=\chi^{\rm rat}\otimes\chi^{\rm fin}$. Then if as before $\pi$ is a regular algebraic cuspidal automorphic representation of $\GL_n(\bfA_F)$ of weight $(V_E,\rho_G,\rho_{\pi_\circ^G})$, then $\pi\otimes\chi$ is a regular algebraic cuspidal automorphic representation of weight $(V_E,\rho_G\otimes(\chi_\infty^{\rm rat})^{-1},\rho_{\pi_\circ^G}\otimes(\chi_\infty^{\rm fin})^{-1})$.
        \end{proposition}

        \begin{proof}
          This is is the same elementary calculation as above.
        \end{proof}

        \begin{remark}
          This shows that the notion of rationality induced by the natural rational structures on locally algebraic representations induces rational structures on cohomological automorphic representations compatibly with the rational structures on Hecke characters and locally algebraic coefficient systems.
        \end{remark}
        
        \begin{remark}
          If $F$ is CM, then locally algebraic representations of $G$ coincide with rational representations. If $F$ is totally real, then Proposition \ref{prop:matsumoto} together with weak approximation shows that we always find a Hecke character
          \[
          \chi\colon F^\times\backslash\bfA_F^\times\to\CC^\times
          \]
          of order at most $2$ such that a given regular algebraic cuspidal automorphic representation $\pi$ of $\GL_n(\Adeles_F)$ of weight $V$ twisted by $\chi$ becomes a regular algebraic cuspidal automorphic representation $\pi\otimes\chi$ whose weight is a rational representation. We refer to Lemme 3.14 in \cite{clozel1990} for a $(\lieg_\CC,K_\infty)$-cohomological interpretation of this observation.

          Conversely the same argument shows that if $\pi$ is of weight $V=(V_E,\rho_G,\rho_{\pi_\circ^G})$, then any locally algebraic representation of the form $V'=(V_E,\rho_G,\rho_{\pi_\circ^G}')$ with $\rho_{\pi_\circ^G}'=\rho_{\pi_\circ^G}\otimes\chi$ for an arbitrary character $\chi$ of $\pi_0(K_\infty)$ is the weight of a regular algebraic cuspidal automorphic representation $\pi'$.
        \end{remark}

        \subsubsection{Cohomological types of $\GL_n$}

        \begin{theorem}[Classification of tempered cohomological types of $\GL_n$, absolutely irreducible case]\label{thm:cohomologicaltypesofGLn}
          Let assume $F/\QQ$ is totally real and let $\pi$ denote a regular algebraic cuspidal automorphic representation of $G=\res_{F/\QQ}\GL_n$ of weight $(V,\rho_G,\rho_{\pi_\circ^G})$ over $\QQ(\pi)$. Then $V$ is essentially self-dual in the sense that there exists a $\QQ$-rational character $\eta\colon G\to\GL_1$ with the property that $V^\vee\cong V\otimes\eta$. Moreover, the cohomological type $\cohType(\pi)$ is tempered and we have
          \begin{equation}
            \cohType(\pi)=
            \begin{cases}
              \{[((\pi_\infty^{(K_\infty)},\rho_{\pi_\circ^G}\otimes\chi),V\otimes\chi)]_\cong\mid \chi\colon\pi_\circ^G\to\mu_2\},&n\;\text{even},\\
              \{[((\pi_\infty^{(K_\infty)}\otimes\rho_{\pi_\circ^G}^{-1},\rho_{\pi_\circ^G}),V)]_\cong\},&n\;\text{odd}.
            \end{cases}
          \end{equation}
          In particular, $\cohType(\pi)$ is a $\Hom(\pi_\circ^G,\Gm)$-torsor over the small \'etale site over $\Spec\CC$ for $n$ even and a singleton for $n$ odd.

          If $F/\QQ$ is CM, then with the same notation as above, $V$ is essentially conjugate self-dual in the sense that there exists a $\QQ$-rational character $\eta\colon G\to\GL_1$ with the property that $\overline{V}^\vee\cong V\otimes\eta$. Moreover, the cohomological type $\cohType(\pi)$ is tempered and is the singleton
          \begin{equation}
            \cohType(\pi)=\{[(\pi_\infty^{(K_\infty)},V)]_\cong\}.
          \end{equation}
        \end{theorem}

        \begin{remark}
          In the totally real case, $\Hom(\pi_\circ^G,\GL_1)\cong\mu_2^{[F:\QQ]}$ and $\Hom(\pi_\circ^G,\GL_1)=1$ in the CM case.
        \end{remark}
        
        \begin{remark}
          For $F$ totally real and $n$ odd or $F$ CM and $n$ arbitrary $\cohTypes_\CC^{\rm temp}(G,V)$ is a $\Hom(\pi_\circ^G,\Gm)$-torsor. 
        \end{remark}

        \begin{proof}
          The claimed statements are implicit in the proof Lemme 3.14 and Lemme 4.9 (Lemme de puret\'e) in \cite{clozel1990}. For the convenience of the reader, we remark that admitting that $\cohType(\pi)\subseteq\cohTypes_{\QQ(\pi)}^{\rm temp}(G)$, the observed patterns have the following alternative explanations. It is easy to see via cohomological induction or by computing the structure of closed $K$-orbits on the flag variety, that since $K(\RR)$ is a non-trivial semi-direct product of $K(\RR)^0$ and $\mu_2^{[F:\QQ]}$ in the first case, for $n$ even, we have $(\pi_\infty^{(K_\infty)}\otimes\chi)_\CC\cong\pi_\infty^{(K_\infty)}$ for any character $\chi\colon\pi_\circ^G\to\mu_2$. If $n$ is odd, then $K(\RR)=K(\RR)^0\times\mu_2^{[F:\QQ]}$ and the $\chi$-twists of $\pi_\infty^{(K_\infty)}$ are all pairwise non-isomorphic. In the CM case, all locally algebraic representations are rational representations, and the closed $K(\RR)$-orbit on the flag variety is unique.

          As to the explicit description, we remark that specializing Definition \ref{def:locallyalgebraicstandardmodule} to the case of a locally algebraic cohomologically induced standard module of weight $(V,\rho_G,\rho_{\pi_\circ^G}\otimes\chi)$ shows that the action of $\pi_\circ^G$ on the underlying $(\lieg,K)$-module must be given by $\rho_{\pi_\circ^G}\otimes\chi$ (which may be considered as a character by Proposition \ref{prop:absolutelyirreduciblelocallyalgebraic}). Together with the previous considerations this shows that the claimed description of the cohomological type of $\pi$ is correct in all cases.
        \end{proof}

        \begin{theorem}[Classification of tempered cohomological types of $\GL_n$ over $\QQ$]\label{thm:cohomologicaltypesofGLnoverQ}
          Let $(V,\rho_G)$ denote an irreducible essentially self-dual rational representation of $G=\res_{F/\QQ}\GL_n$ for $F$ totally real defined over $\QQ$. Then $V$ is the restriction of scalars of an absolutely irreducible essentially self-dual representation $V_0$ of $G$ inside the extension $\QQ(V_0)/\QQ$. As such $V_0$ is unique up to conjugation by automorphisms of the normal hull of $\QQ(V_0)/\QQ$. Moreover, for any character $\rho_{\pi_\circ^G}\colon\pi_\circ^G\to\mu_2$ we have
          \begin{equation}
            \cohTypes_\QQ^{\rm temp}(G,(V,\rho_G,\rho_{\pi_\circ^G}))=
            \begin{cases}
              \{\{[((A(V)_\QQ,\rho_{\pi_\circ^G}\otimes\chi),(V,\rho_G,\rho_{\pi_\circ^G}\otimes\chi))]_\cong\mid \chi\colon\pi_\circ^G\to\mu_2\}\},&n\;\text{even},\\
              \{\{[((A(V)_\QQ,\rho_{\pi_\circ^G}),(V,\rho_G,\rho_{\pi_\circ^G}))]_\cong\}\},&n\;\text{odd}.
            \end{cases}
          \end{equation}
          Here $A(V)_\QQ=\res_{\QQ(V_0)/\QQ}A(V_0)_{\QQ(V_0)}$ and $A(V_0)_{\QQ(V_0)}=\QQ(V_0)\otimes_{\OO(V_0)[1/2d_E]}$ in the notation of Definitions 6.3.7 and 6.3.8 in \cite{hayashijanuszewski}.

          Likewise, if $F/\QQ$ is CM, then with the same notation as above, and $V$ is an irreducible essentially self-dual rational representation of $G$ over $\QQ$, then $V=\res_{\QQ(V_0)/\QQ}V_0$ for an absolutely irreducible essentially conjugate self-dual rational representation $V_0$ of $G$, unique up to conjugation by automorphisms of the normal hull of $\QQ(V_0)/\QQ$. Moreover,
          \begin{equation}
            \cohTypes_\QQ^{\rm temp}(G,(V,\rho_G,{\bf1}))=
              \{\{[((A(V)_\QQ,{\bf1}),(V,\rho_G,{\bf1}))]_\cong\}\},
          \end{equation}
          with $A(V)_\QQ$ analoguously defined as in the real case by invoking Definition 6.3.11 in \cite{hayashijanuszewski}.
        \end{theorem}

        \begin{proof}
          Since $G$ is quasi-split in the case at hand, Theorem \ref{thm:irreduciblelocallyalgebraicrepresentations} combined with Corollary \ref{cor:quasisplitlocallyalgebraicrepresentations} applied $V$ shows the claims about $V$ and $V_0$.

          Theorems 6.3.6 and 6.3.10 in \cite{hayashijanuszewski} show that in all cases $A(V_0)_\QQ$ is defined over its field of rationality. Therefore \ref{thm:irreduciblelocallyalgebraicgKmodules} is applicable and shows that $A(V)_\QQ$ is irreducible and gives rise to the unique irreducible algebraic $(\lieg,K)$-module for which in light of the previously established Theorem \ref{thm:cohomologicaltypesofGLn} the claimed statements are true.
        \end{proof}

        \begin{remark}\label{rmk:cyclesforGLn}
          Theorem \ref{thm:cohomologicaltypesofGLnoverQ} shows that in the case at hand we have $m_{x,\lambda}^{A(V_0)_\CC}=m_{x,\lambda}^{V_{0,\CC}}=1$ (or $0$) for the multicities in the definition of the cycle of any of the cohomological types occuring in Definition \ref{def:geometricsupport}. These multiplicities are closely related to multiplicites of cuspidal automorphic representations in cohomology with coefficients in $V$, which we will exploit in the proof of Theorem \ref{thm:globalhalfintegralstructures} on global $1/N$-rational structures on spaces of cusp forms below. See also Proposition \ref{prop:decompositionbycohomologicaltypes}.
        \end{remark}
        
        \subsubsection{Construction of global models}
        
	Assume that $F/\QQ$ is a number field which is totally real or a CM field. Write $\OO_F$ for its ring of integers. Put $G=\res_{\OO_F[1/2]/\ZZ[1/2]}\GL_n$. Recall that $G$ is smooth over $\ZZ[1/2]$ (cf.\ Example \ref{ex:smoothgroups}). Put $K=\res_{\OO_F[1/2d_F]/\ZZ[1/2d_F]}\Oo(n)$ for $F$ is totally real. If $F$ is imaginary CM with maximal totally real subfield $F^+\subseteq F$ we put $K=\res_{\OO_{F^+}[1/2d_{F}]/\ZZ[1/2d_F]}\U(n)$. Here $\Oo(n)$ and $\U(n)$ are the standard models associated to the standard diagonal quadratic forms and $d_F$ denotes the absolute discriminant of $F/\QQ$. Then in all cases $K$ is a smooth group scheme over $\ZZ[1/2d_F]$.
	
	\begin{theorem}\label{thm:gln}
	  Assume that $\pi$ is a regular algebraic cuspidal automorphic representation of $\GL_n$ over a number field $F$ which is totally real or a CM field. Assume $\pi$ is of cohomological type ${\rm t}=\cohType(\pi)$. Fix a representative $\alpha=[(A(V),V)]_\cong\in{\rm t}$ and write $E=\QQ(V)$ for the field of rationality of the locally algebraic representation $V$ ($\QQ(V)$ is independent of the choice of representative $(A(V),V)$). If $\OO\subseteq \QQ(\pi)$ denotes the ring of integers in the field of rationality $\QQ(\pi)$ of $\pi_f$. Then $\pi$ admits an $\OO[1/N]$-model $(\pi_{\OO[1/N]},\varphi)$ for $N=2d_F$.
          
	  Moreover, if $V_{\OO[1/N]}\subseteq V_{\QQ(\pi)}$ is a lattice, the $\OO[1/N]$-model $(\pi_{\OO[1/N]},\varphi)$, as $(\lieg_{\OO[1/N]},K_{\OO[1/N]}\times\pi_\circ^G)\times\mathcal H(V_{\OO[1/N]})$-submodule of $\pi^{(K_\infty)}$ via $\varphi$ (cf.\ \eqref{eq:varphiglobal}), is canonical in the sense that it depends on the choice of lattice $V_{\OO[1/N]}$ but does not depend on the choice of the $\theta$-stable parabolic subalgebra $\lieq'$ from \eqref{eq:cohomologicalarchimedeancomponent} inside its Weyl group orbit.
	\end{theorem}
	
	\begin{proof}
          Observe that since $G$ is quasi-split, $V$ is defined over its field of rationality (cf.\ Corollary \ref{cor:quasisplitrationality}). Moreover, Clozel has shown in \cite[Proposition 3.1]{clozel1990} that $\pi_f$ is defined over its field of rationality $\QQ(\pi)$ as well. Recall that the number fields $E=\QQ(V)$ as defined as a subfield of $\CC$, then $\QQ(V)\subseteq\QQ(\pi)$ holds automatically since due to multiplicity one for $\GL_n$, $\pi_f$ determines $\pi_\infty$ uniquely by \cite{jacquetshalika1981,jacquetshalika1981.2}.

          Clozel showed in his Lemme de puret\'e \cite[Lemme 4.9]{clozel1990} that $\pi_\infty$ is essentially tempered. In particulr, if $F$ is totally real, then $V_E^\vee\otimes V_E\otimes\eta$ for a $\QQ$-rational character $\eta:G\to\GL_1$ and consequently the weight $\lambda$ of $V_E$ on the torus $\res_{F/\QQ} H$ defined by \eqref{eq:lambdaforV} satisfies condition of Definitions 6.3.7 and 6.3.8 in \cite{hayashijanuszewski}. If $F$ is a CM field, then $V$ is essentially conjugate self-dual over $\QQ$, i.\,e.\ $\overline{V}^\vee\cong V\otimes\eta$ again for a $\QQ$-rational character $\eta:G\to\QQ$. Hence Definition 6.3.11 in \cite{hayashijanuszewski} is applicable in this case.

          In both cases Hayashi and the author constructed a $(\lieg,K)$-module $A(V_E)_{\OO(V)\left[1/N\right]}$ over $\OO[1/N]$ since the ring $\OO(\vec{w}\lambda)$ in section 6.3 \cite{hayashijanuszewski} agrees with $\OO$ in the notation of the current section.

          In Theorems 6.3.6 and 6.3.10 it is shown that $A(V)_{\OO(V)\left[1/N\right]}$ is a projective $\OO[1/N]$-model of the complex algebraic standard module $A_\lieq(\lambda)_\CC$ of which $\pi_\infty^{(K_\infty)}$ is a $\rho_{\pi_\circ^G}$-twist. Hence $(\QQ(\pi)\otimes_{\OO[1/N]} A(V_E)_{\OO[1/2d_E]},\rho_{\pi_\circ^G})$ is a locally algebraic standard module in the sense of Definition \ref{def:locallyalgebraicstandardmodule} whose complex realization in the sense of Definition \ref{def:locallyalgebraiccomplexmodule} is the $(\lieg_\CC,K_\infty)$-module $\pi_\infty^{(K_\infty)}$ underlying $\pi_\infty$.
          
          To conclude the proof, observe that by construction,
          \begin{equation}
            \pi_{\OO[1/N]}:=
            A(V_{\QQ(\pi)})_{\OO[1/N]}\otimes_{\OO}\pi_{f,\OO}
            \label{eq:piOforGLn}
          \end{equation}
          is a $(\lieg,K\times\pi_\circ^G)\times G(\Adeles_f)$-module over $\OO[1/N]$ and as such is an $\OO[1/N]$-model of the complex $(\lieg_\CC,K_\infty)\times G(\Adeles_f)$-module $\pi^{(K_\infty)}$ with all desired properties.
	\end{proof}
	
	\begin{remark}
		In the case of $\GL_n$ over $F=\QQ$ global $\ZZ[1/2]$-models were first constructed by Harder, cf.\ \cite{harderraghuram}.
	\end{remark}
        
	\begin{remark}\label{rmk:fieldrestriction}
		The restriction on the ground field $F$ may be relaxed at the same cost as in \cite{januszewskirationality} that the field where we find global models may be a finite extension of the field of rationality $\QQ(\pi)$ for in the general situation $\QQ$-rational models of the maximal compact subgroup in $G(\RR)$ need not exist.
	\end{remark}

      \section{Canonical integral structures and canonical periods}\label{sec:canonicalintegralstructuresandperiods}
	
	\subsection{Canonical $1/N$-integral structures on spaces of automorphic cusp forms}\label{sec:canonicalintegralstructures}
	
	In the setting of the previous section, i.\,e.\ $G$ and $K$ both explicit smooth $\ZZ[1/2]$-models of $\GL_n(\RR\otimes_\QQ F)$ and $K(\RR)$ for $F/\QQ$ totally real or a CM field, we extend Theorem \ref{thm:gln} to global spaces of cusp forms as follows. Again we fix an absolutely irreducible locally algebraic representation $V=(V_E,\rho_G,\rho_{\pi_\circ^G})$ of $G$ defined over its field of rationality $E=\QQ(V)$ as before (cf.\ Corollary \ref{cor:quasisplitrationality}). Recall that by Proposition \ref{prop:locallyalgebraicfieldofrationality} $\QQ(V)$ agrees with the field of rationality of the underlying rational representation $(V_E,\rho_G)$. Again we write $d_E$ for the absolute discriminant of $\QQ(V)/\QQ$.
	
	\subsubsection{The $\QQ$-rational local system}

        The following construction is motivated by Theorem \ref{thm:cohomologicaltypesofGLnoverQ}, see also Remark \ref{rmk:cyclesforGLn}.

        By Theorem \ref{thm:irreduciblelocallyalgebraicrepresentations} the locally algebraic representation $V$ occurs in a unique locally algebraic representation $W=(W_\QQ,\rho^W_G,\sigma^W_K)$ over $\QQ$, i.\,e.\ $V$ occurs in $\CC\otimes_\QQ W=(\CC\otimes_\QQ W_\QQ,1\otimes\rho_G^W,1\otimes\rho_{\pi_\circ^G}^W)$, necessarily with multiplicity one, where $W$ is defined as follows:
	\[
	W:=\res_{\QQ(V)/\QQ}V:=(\res_{\QQ(V)/\QQ}V_{\QQ(V)},\res_{\QQ(V)/\QQ}\rho_G,\res_{\QQ(V)/\QQ}\rho_{\pi_\circ^G}).
	\]
	Restriction of scalars here simply forgets the $\QQ(V)$-vector space structure along the inclusion of fields $\QQ\to\QQ(V)$.
	
	Extending scalars to $\CC$ yields a decomposition
	\[
	\CC\otimes W=\bigoplus_{\sigma\colon \QQ(V)\to\CC}\CC\otimes_{\sigma,\QQ(V)}V_{\QQ(V)},
	\]
	with pairwise inequivalent irreducible complex locally algebraic representations
	\[
	V_\CC^\sigma:=\CC\otimes_{\sigma,\QQ(V)}V_{\QQ(V)},
	\]
        which by abuse of notation we also identify with their respective associated representations of $G(\RR)$ in the sense of Definition \ref{def:locallyalgebraiccomplex}, which are pairwise inequivalent as well (cf.\ Remark \ref{rmk:locallyalgebraicequivalence}).
       
	\subsubsection{Normalized lattices}
	
	As before, for a given compact open $K_f\subseteq G(\bfA_f)$ with $K_f\subseteq G(\widehat{\ZZ})$, we fix a $K_f$-stable lattice $V_{\OO_{E}[1/2d_E]}\subseteq V_{E}$, on which $G_{\OO_{E}[1/2d_E]}$ acts scheme-theoretically, \cite{hayashi2018} Proposition 2.1.6.
	
	In order to normalize $V_{\OO_{E}[1/2d_E]}$, it is sufficient to normalize $V_{\OO_{E}[1/2d_E]}$ as a sublattice of the underlying rational representation $(V_E,\rho_G)$ via the following
	\begin{lemma}\label{lem:VOnormalization}
	  Assume that $(V_E,\rho_G)$ is an essentially self-dual ($F$ totally real) resp.\ essentially conjugate-self-dual ($F$ a CM field) absolutely irreducible rational representation of $G$. As before, put $N=2d_F$ and write $A(V_E)_{\OO_E[1/N]}$ for the $1/N$-integral standard module from the proof of Theorem \ref{thm:gln}.

          For every $K_f$-stable $\OO_{E}[1/N]$-lattice $V_{\OO_{E}[1/N]}\subseteq V_{E}$, on which $G_{\OO_{E}[1/N]}$ acts scheme-theoretically, there exists a fractional $\OO_{E}[1/N]$-ideal $\mathfrak{a}\subseteq E$ with the property that the $K_f$-stable lattice
	  \begin{equation}
	    \mathfrak{a}\cdot V_{\OO_{E}[1/N]}
	    \subseteq V_{E}
	  \end{equation}
	  satisfies the condition that for the bottom degree $t_0$ for $G$ the image of
	  \[
	  H^{t_0}(\lieg,\lies, K;
	  A(V_E)_{\OO_{E}[1/N]}\otimes \mathfrak{a}\cdot V_{\OO_{E}[1/N]})
	  \]
	  in
	  \[
	  H^{t_0}(\lieg,\lies, K;
	  A(V_E)_{E}\otimes  V_{\QQ(V)})
	  \]
	  is free of rank $1$ over $\OO_{E}[1/N]$.
		
	  Moreover if $\mathfrak{a}'$ is another fractional ideal with the same property, then $\mathfrak{a}'=\alpha\cdot\mathfrak{a}$ for some $\alpha\in\QQ(V)^\times$.
	\end{lemma}
	
	\begin{proof}
		Observe first that for any embedding $E=\QQ(V)\to\CC$, since $(\lieg,\lies, K)$-cohomology commutes with arbitrary base change,
		\[
		\CC\otimes_E H^{t_0}(\lieg,\lies, K;
		A(V_E)_{E}\otimes V_{E})\cong
		H^{t_0}(\lieg_{\CC},\lies, K(\RR);
		A(V_\CC)\otimes V_{\CC})
		\]
		is a $\CC$-vector space of dimension $1$, which shows that $(\lieg,\lies, K)$-cohomology in degree $t_0$ over $E$ is of dimension $1$ as well.
		
		Any fractional ideal $\mathfrak{a}\subseteq E$ may be considered as a model of the trivial rational representation of $G$. Therefore $G$ acts on
		\[
		\mathfrak{a}\otimes V_{\OO_{E}[1/N]}\cong\mathfrak{a}\cdot V_{\OO_{E}[1/N]}
		\]
		scheme theoretically. Moreover, by the triviality of the action on $\mathfrak{a}$,
		\begin{align*}
			&H^{t_0}(\lieg,\lies, K;
			A(V_E)_{\OO_{E}[1/N]}\otimes \mathfrak{a}\cdot V_{\OO_{E}[1/N]})\\
			=\;&\mathfrak{a}\cdot H^{t_0}(\lieg,\lies, K;
			A(V_E)_{\OO_{E}[1/N]}\otimes V_{\OO_{E}[1/N]}).
		\end{align*}
		Since $\OO_{E}[1/N]$ is a Dedekind domain it therefore suffices to prove that the image of the canonical arrow
		\begin{align*}
			&H^{t_0}(\lieg,\lies, K;
			A(V_E)_{\OO_{E}[1/N]}\otimes \mathfrak{a}\cdot V_{\OO_{E}[1/N]})
			\to
			H^{t_0}(\lieg,\lies, K;
			A(V_E)_{E}\otimes  V_{E})
		\end{align*}
		is a finitely generated $\OO_{E}[1/N]$-module, because the codomain is of dimension $1$ over $E$.
		
		This in turn is a consequence of Variant 6.4.19 in \cite{hayashijanuszewski} applied to $A(V_E)_{\OO_{E}[1/N]}$.
	\end{proof}

        By Theorem \ref{thm:cohomologicaltypesofGLn} there is a unique tempered cohomological type ${\rm t}\in\cohTypes_E^{\rm temp}(G,(V_G,\rho_G))$ associated to $V$ and $V$ defines a unique representative $\alpha=[(A(V),V)]_\cong\in{\rm t}$, which in turn characterizes $V$.
        
	We henceforth assume that $V_{\OO_{E}[1/N]}$ is normalized according to Lemma \ref{lem:VOnormalization} in the sense that in bottom degree $(\lieg,\lies, K)$-cohomology over $\OO_{E}[1/N]$ maps to a free $\OO_{E}[1/N]$-module
	\[
	L_{\OO_{E}[1/N]}^{\alpha}\subseteq H^{t_0}(\lieg,\lies, K; A(V)\otimes  V)_E,
	\]
        where the right hand side denotes relative Lie algebra cohomology in the sense of Definition \ref{def:locallyalgebraicrelativeLiealgebracohomology} of the locally algebraic standard module $A(V):=(A(V_E),\rho_{\pi_\circ^G})$ of weight $V=(V_E,\rho_G,\rho_{\pi_\circ^G})$ over $E$.
        
	We fix an isomorphism
	\begin{equation}
		\vartheta_{\OO_{E}[1/N]}^{\alpha}\colon L_{\OO_{E}[1/N]}^{\alpha}\to\OO_{E}[1/N].
		\label{eq:varthetachoice}
	\end{equation}
	Since
	\[
	E\otimes L_{\OO_{E}[1/N]}^{\alpha}=H^{t_0}(\lieg,\lies, K; A(V)\otimes  V)_E
	\]
	canonically, and since $(\lieg,\lies, K)$-cohomology commutes with arbitrary base change, $\vartheta_{\OO_{E}[1/N]}^{\alpha}$ induces for every embedding $\sigma:E\to\CC$ an isomorphism
	\[
	\vartheta_\CC^{\alpha^\sigma}\colon H^{t_0}(\lieg,S(\RR)^0 K(\RR); A(V_\CC)\otimes  V_{\CC}^\sigma)\to\CC.
	\]
        Here $A(V_\CC)$ agrees with the complexification of $A(V)$ in the sense of Definition \ref{def:locallyalgebraiccomplexmodule} (see also Remark \ref{rmk:locallyalgebraicstandardmodulesarecohomologicallyinduced}).
        
	The choice of $\vartheta_{\OO_{E}[1/N]}^{\alpha}$ is unique up to units in $\OO_{E}[1/N]$. Therefore $\vartheta_\CC^{\alpha^\sigma}$ is unique up to units in $\sigma\left(\OO_{E}[1/N]^\times\right)$ and these isomorphisms are chosen coherently for all $\alpha\in{\rm t}$ and all ${\rm t}\in\cohTypes_E^{\rm temp}(G)$ at once in the sense that we impose the relation
        \begin{equation}
          \forall\sigma\colon E\to\CC,\tau\in\Aut(\CC/\QQ):
          \quad
	  \vartheta_{\OO_{E^{\sigma\tau}}[1/N]}^{\alpha^{\sigma\tau}}=
	  \left(\vartheta_{\OO_{E^\sigma}[1/N]}^{\alpha^{\sigma}}\right)^\tau.
	  \label{eq:varthetacoherence}
        \end{equation}
        Here $\alpha^{\sigma}$ denote the base change of $\alpha$ along the isomorphism $E\to E^\sigma$ induced by the embedding $\sigma\colon E\to\CC$.
	
	We restrict $V_{\OO_{E}[1/N]}$ to a $K_f$-stable $\ZZ[1/N]$-lattice
	\[
	W_{\ZZ[1/N]}=\res_{\OO_{E}/\ZZ}V_{\OO_{E}[1/N]}
	\]
	in $W_\QQ$. Once again, $W_{\ZZ[1/N]}$ is $V_{\OO_{E}[1/N]}$ considered as $\ZZ[1/N]$-module and we may and do assume that $G\times\pi_\circ^G$ acts on $W_{\ZZ[1/N]}$ scheme-theoretically.
	
	Likewise, the integral Hecke algebra $\mathcal H(V_{\OO_{E}[1/N]})$ gives rise to a $\ZZ[1/N]$-algebra
	\[
	\mathcal H(W_{\ZZ[1/N]}):=\res_{\OO_{E}/\ZZ}\mathcal H(V_{\OO_{E}[1/N]}).
	\]
	
	\subsubsection{Spaces of automorphic cusp forms}
	
	Write $\omega_{V^\sigma}$ for the dual infinitesimal character of the complexification $V_\CC$ of $V$ and denote by
	\[
	\omega_W=\bigoplus_{\sigma\colon E\to\CC}\omega_{V_{\CC}^\sigma}
	\]
	the direct sum of the infinitesimal characters occuring in the dual of $W_\CC$.
	
	We consider the space
	\begin{flalign}
		&L_0^2(\GL_n(F)Z(\RR)^0 \backslash{}\GL_n(\Adeles_F);\omega_W)^\infty
		\;=\label{eq:l2cuspforms}\\
	&\bigoplus_{\sigma:E\to\CC}
	        L_0^2(\GL_n(F)Z(\RR)^0 \backslash{}\GL_n(\Adeles_F);\omega^\sigma)^\infty,\nonumber     
	\end{flalign}
	of smooth $L^2$-automorphic cusp forms of $G(\Adeles)$ which are $Z(\lieg)$-finite and as such are sums of cusp forms with infinitesimal characters $\omega^\sigma$ for some embedding $\sigma:E\to\CC$. This space can be characterized as the space of $Z(\lieg)$-finite cusp forms annihilated by the intersection $\mathfrak{m}_W$ of the maximal ideals $\mathfrak{m}_{\omega^\sigma}\subseteq Z(\lieg)$ associated with the infinitesimal characters $\omega^\sigma$ for all embeddings $\sigma:E\to\CC$. We remark that the ideal $\mathfrak{m}_W$ is defined over $\QQ$.
	
	Moreover, \eqref{eq:l2cuspforms} admits a canonical action of the complex Hecke algebra $\mathcal H(W_{\CC})$.
	
	In Theorem A in \cite{januszewskirationality} the second named author showed that the $(\lieg_\QQ,K_\QQ)\times G(\Adeles_f)$-module of $K_\infty$-finite vectors in \eqref{eq:l2cuspforms} admits a canonical global $\QQ$-structure if $V$ is algebraic. In this section we extend the result of Theorem A to $\ZZ[1/N]$-integral structures and locally algebraic representations $V$.
	
	While the proof in loc.\ cit.\ may be translated to the situation at hand, such a construction of a global integral structure would not faithfully reflect the delicate canonical integral structure on cohomology. To take this into account, we modify the construction as follows.
	
	\subsubsection{The canonical monomorphism}
	
	We recall first the relevant Eichler--Shimura isomorphism. For any sufficiently small compact open $K_f\subseteq \GL_n(\Adeles_{F,f})$ write
	\begin{equation}
	L_0^2(\GL_n(F)Z(\RR)^0 \backslash{}\GL_n(\Adeles_F)/K_f;\omega^\sigma)^{\infty,(K_\infty)}\subseteq
	L_0^2(\GL_n(F)Z(\RR)^0 \backslash{}\GL_n(\Adeles_F);\omega^\sigma)^\infty
        \label{eq:omegasigmaautomorphicforms}
	\end{equation}
	for the subspace of $K_f$-invariant smooth $K_\infty$-finite vectors.

        \begin{remark}
          We emphasize that for any given $\omega^\sigma$ there exist up to isomorphism finitely many locally algebraic representations $V$ such that the dual of $V^\sigma$ has infinitesimal character $\omega^\sigma$: The action $\rho_{\pi_\circ^G}$ of $\pi_\circ^G$ on $V$ is varies independently of the infinitesimal character. Moreover, it is fixed under all automorphisms of $\CC$ (cf.\ Proposition \ref{prop:absolutelyirreduciblelocallyalgebraic}).
        \end{remark}

        Note that the infinitesimal character $\omega^\sigma$ is defined over $\QQ(V^\sigma)=E^\sigma$ and put for any extension $F/E^\sigma$
        \[
        {\rm locAlg}_F(\omega^\sigma):=\{V'\;\text{loc.\ alg.\ rep.\ over $F$ with infin.\ char.}\;\omega^{\sigma,\vee}\}/\cong
        \]
        for the finite set of isomorphism classes of locally algebraic representations over $F$ whose dual has infinitesimal character $\omega^\sigma$. Note that the elements of ${\rm locAlg}_{E^\sigma}(\omega^\sigma)$ are in bijection with the locally algebraic representations $(V_E^\sigma,\rho_G^\sigma,\rho_{\pi_\circ^G}')$ for varying $\pi_\circ^G$-action $\rho_{\pi_\circ^G}'$. In particular, ${\rm locAlg}_{E^\sigma}(\omega^\sigma)$ is a $\Hom(\pi_\circ^G,\Gm)$-torsor on the small \'etale site over $\Spec E^\sigma$.

        To streamline notation and emphasize coherence conditions, we introduce the abbreviation
        \[
        {\rm tcT}_E:=\cohTypes^{\rm temp}_E(G,(V_E,\rho_G))
        \]
        for the set of tempered cohomological types of weight $(V_E,\rho_G)$. Then we have for each embedding $\sigma\colon E\to\CC$
        \[
        \cohTypes^{\rm temp}_{E^\sigma}(G,(V_E^\sigma,\rho_G^\sigma))={\rm tcT}_E^\sigma=\{{\rm t}^\sigma\mid {\rm t}\in {\rm tcT}_E\},
        \]
        where $V_E^\sigma$ denotes the base change of $V_E$ along the isomorphism $E\to E^\sigma$ induced by $\sigma$. Likewise, ${\rm t}^\sigma$ is the base change of the cohomological type ${\rm t}$. We have for all automorphisms $\tau\in\Aut(\CC/\QQ)$
        \[
        \left({\rm t}^{\sigma}\right)^\tau=
        {\rm t}^{\sigma\tau},
        \]
        where $\tau$ acts element-wisely on the set of the left hand side (cf.\ Definition \ref{def:cohomologicaltypesofG}). This action is compatible with complexified cohomological types, i.\,e.\ we have
        \[
        {\rm tcT}_\CC=\cohTypes^{\rm temp}_\CC(G,(V_E,\rho_G)),
        \]
        and for every embedding $\sigma\colon E\to\CC$
        \[
        {\rm tcT}_\CC^\sigma=\cohTypes^{\rm temp}_\CC(G,(V_E^\sigma,\rho_G^\sigma)),
        \]
        compatibly with the above $\Aut(\CC/\QQ)$-action, i.\,e.\ complexification of ${\rm t}\in{\rm tcT}_E^\sigma$ is understood with respect to the embedding $E^\sigma\to\CC$ induced by $\sigma\colon E\to\CC$. We remark that complexification induces bijections of sets of tempered cohomological types.
        
        \begin{proposition}\label{prop:decompositionbycohomologicaltypes}
        With the notation as above, for each $\sigma\colon E\to\CC$ the map
        \begin{equation}
          \bigsqcup_{{\rm t}\in{\rm tcT}_E^\sigma}{\rm t}\to{\rm locAlg}_{E^\sigma}(\omega^\sigma),\quad [(A',V')]_\cong\mapsto [V']_\cong
          \label{eq:cohomologicaltypestolocallyalgebraicrepresentations}
        \end{equation}
        is an isomorphism of $\Hom(\pi_\circ^G,\Gm)$-torsors on the small \'etale site over $\Spec E^\sigma$. Moreover,
        \[
	\bigoplus_{{\rm t}\in{\rm tcT}_E^\sigma}\bigoplus\limits_{\begin{subarray}c\pi:\\\cohType(\pi)={\rm t}\end{subarray}} \pi^{(K_\infty)}=
	L_0^2(\GL_n(F)Z(\RR)^0 \backslash{}\GL_n(\Adeles_F);\omega^\sigma)^{\infty,(K_\infty)},
	\]
        where the left hand side runs over all cuspidal automorphic representations of $G(\Adeles)$ with infinitesimal character $\omega^\sigma$, partitioned into cohomological types.

        If we put for every ${\rm t}\in{\rm tcT}_E^\sigma$
        \[
        L_0^2(\GL_n(F)Z(\RR)^0 \backslash{}\GL_n(\Adeles_F);\omega^\sigma)^{\infty}_{{\rm t}}:=
        \left(\widehat{\bigoplus\limits_{\begin{subarray}c\pi:\\\cohType(\pi)={\rm t}\end{subarray}}} \pi\right)^\infty,
        \]
        then
        \begin{equation}
          L_0^2(\GL_n(F)Z(\RR)^0 \backslash{}\GL_n(\Adeles_F);\omega^\sigma)^\infty=
          \widehat{\bigoplus_{{\rm t}\in{\rm tcT}_E^\sigma}}
          L_0^2(\GL_n(F)Z(\RR)^0 \backslash{}\GL_n(\Adeles_F);\omega^\sigma)^\infty_{{\rm t}}.
          \label{eq:l02spacecohomologicaltypedecomposition}
        \end{equation}
        \end{proposition}

        \begin{proof}
          The map \eqref{eq:cohomologicaltypestolocallyalgebraicrepresentations} is tautologically surjective. The injectivity follows from the classification of tempered cohomological types over $\CC$ in Theorem \ref{thm:cohomologicaltypesofGLnoverQ} for the cases at hand.

          A cuspidal representation $\pi$ of $\GL_n(\Adeles_F)$ with infinitesimal character $\omega^\sigma$ is automatically cohomological of weight $V\in{\rm locAlg}_E(\omega^\sigma)$ whenever $E$ is sufficiently large (i.\,e.\ $\QQ(V)\subseteq E$). Moreover, refering to Theorem \ref{thm:cohomologicaltypesofGLnoverQ} once more time shows that the remaining statements are formal consequences of the bijectivity of \eqref{eq:cohomologicaltypestolocallyalgebraicrepresentations}.
        \end{proof}
        
        For any ${\rm t}\in{\rm tcT}_E^\sigma$ and any representative $[(A^{({\rm t})},V^{({\rm t})})]_\cong\in{\rm t}$ we have $V^{({\rm t})}\in {\rm locAlg}_{E^\sigma}(\omega^\sigma)$ and the canonical inclusion
	\begin{flalign*}
	&H^q\left(\lieg_\CC,S(\RR)^0 K(\RR);
	L_0^2(\GL_n(F)Z(\RR)^0 \backslash{}\GL_n(\Adeles_F)/K_f;\omega^\sigma)_{\rm t}^{\infty,(K_\infty)}
        \otimes V_{\CC}^{({\rm t})}
        \right)
	\\
	&\to
        H^q\left(\lieg_\CC,S(\RR)^0 K(\RR);
	L_0^2(\GL_n(F)Z(\RR)^0 \backslash{}\GL_n(\Adeles_F)/K_f;\omega^\sigma)^{\infty,(K_\infty)}
        \otimes V_{\CC}^{({\rm t})}
        \right)
	\end{flalign*}
        is an isomorphism and we have a canonical Eichler--Shimura isomorphism
	\begin{flalign*}
	&H^q\left(\lieg_\CC,S(\RR)^0 K(\RR);
	L_0^2(\GL_n(F)Z(\RR)^0 \backslash{}\GL_n(\Adeles_F)/K_f;\omega^\sigma)^{\infty,(K_\infty)}
        \otimes V_{\CC}^{({\rm t})}\right)\\
        &\cong
        H^q\left(\lieg_\CC,S(\RR)^0 K(\RR);
	L_0^2(\GL_n(F)Z(\RR)^0 \backslash{}\GL_n(\Adeles_F)/K_f;\omega^\sigma)_{\rm t}^{\infty,(K_\infty)}
        \otimes V_{\CC}^{({\rm t})}\right)
	\\
	&\to H^q_{\rm cusp}(\GL_n(F)\backslash{}\GL_n(\Adeles_F)/S(\RR)^0 K(\RR)K_f;\widetilde{V}_{\CC}^{({\rm t})}).
	\end{flalign*}
        Here $\widetilde{V}_{\CC}^{({\rm t})}$ denotes the sheaf associated to the complexification of $V^{({\rm t})}$ along the embedding $E^\sigma\to\CC$ induced by the embedding $\sigma\colon E\to\CC$. We refer to \cite{franke1998} and \cite{januszewskirationality}, section 8 for a discussion of the general Eichler--Shimura isomorphism and related rationality considerations.

        Note that the normalized $K_f$-stable lattice $V_{\OO_{E}[1/N]}\subseteq V_E$ (cf.\ Lemma \ref{lem:VOnormalization}) induces a $K_f$-stable lattice $V_{\OO_{E^\sigma}[1/N]}^{({\rm t})}$ in $V^{({\rm t})}$ via a fixed identification of the underlying rational representations of $V^\sigma$ and $V^{({\rm t})}$. Denote by
	\[
	H_{\rm cusp}^q\left(X_G(K_f);\widetilde{V}_{\OO_{E^\sigma}[1/N]}^{({\rm t})}\right):=
	\image\left(
	H^q\left(X_G(K_f);\widetilde{V}_{\OO_{E^\sigma}[1/N]}^{({\rm t})}\right)
	\to H_{\rm cusp}^q\left(X_G(K_f);\widetilde{V}_{\CC}^{({\rm t})}\right)\right),
	\]
	the image of cohomology with $1/N$-integral coefficients in cuspidal cohomology. Recall that cuspidal cohomology is an $E$-rational subspace of cohomology.

	As before, $t_0$ denotes the bottom degree of the cuspidal range for $G$, i.\,e.\ $t_0$ is the minimal degree where cuspidal cohomology is non-zero for some $K_f$. Consider the $\OO_E[1/N]$-model $A(V_E)_{\OO_{E}[1/N]}$ from the proof of Theorem \ref{thm:gln}. We may assume without loss of generality that $A^{({\rm t})}=A_E^\sigma$ as $(\lieg_{E^\sigma},K_{E^\sigma})$-modules. Via this identification we define $A(V^{({\rm t})})_{\OO_{E}[1/N]}$ as the lattice $A(V_E)_{\OO_{E}[1/N]}^\sigma$ inside $A^{({\rm t})}$ and consider the $(\lieg_{\OO_{E^\sigma}[1/N]},K_{\OO_{E^\sigma}[1/N]}\times\pi_\circ^G)\times\mathcal H(V_{\OO_{E^\sigma}[1/N]}^{({\rm t})})$-module
	\[
	\mathcal A(V_{\OO_{E^\sigma}[1/N]}^{({\rm t})}):=
	A(V^{({\rm t})})_{\OO_{E^\sigma}[1/N]}
        \otimes_{\OO_{E^\sigma}[1/N]}
	\varinjlim_{K_f}H_{\rm cusp}^{t_0}\left(X_G(K_f);\widetilde{V}_{\OO_{E^\sigma}[1/N]}^{({\rm t})}\right).
	\]
        Write
        \[
        W^{({\rm t})}:=\res_{E^\sigma/\QQ}V^{({\rm t})},
        \]
        for the unique irreducible locally algebraic $G$-representation which over $E$ contains $V^{({\rm t})}$. Then $W^{({\rm t})}$ is independent of the embedding $\sigma\colon E\to\CC$ (up to unique isomorphism) and contains the $K_f$-stable $\ZZ[1/N]$-lattice
        \[
        W_{\ZZ[1/N]}^{({\rm t})}:=\res_{\OO_{E^\sigma}/\ZZ}V_{\OO_{E^\sigma}[1/N]}^{({\rm t})}.
        \]
        The restriction of scalars
	\[
	\mathcal A(W_{\ZZ[1/N]}^{({\rm t})}):=
	\res_{\OO_{E^\sigma}/\ZZ}\mathcal A_{\OO_{E^\sigma}[1/N]}(V_{\OO_{E^\sigma}[1/N]}^{({\rm t})})
	\]
	admits a structure of $(\lieg_{\ZZ[1/N]},K_{\ZZ[1/N]}\times\pi_\circ^G)\times\mathcal H(W_{\ZZ[1/N]}^{({\rm t})})$-module, where the action of $\pi_\circ^G$ is inherited from its action on the locally algebraic representation $V^{({\rm t})}$.
        
        For any $\ZZ[1/N]$-algebra $A$ put
	\[
	\mathcal A(W_{A}^{({\rm t})}):=A\otimes_{\ZZ[1/N]} \mathcal A(W_{\ZZ[1/N]}^{({\rm t})}).
	\]
	In order to relate $\mathcal A(W_{\CC}^{({\rm t})})$ to automorphic cusp forms observe first that for any extension $L/E$ containing a normal hull of $E$ over $\QQ$ and any essentially (conjugate) self-dual absolutely irreducible locally algebraic representation $V$ with field of definition $E$ and $W=\res_{E/\QQ}V$ we have
	\begin{align*}
		\mathcal A(W_L)
		&=L\otimes_\QQ\mathcal A(V)\\
		&=\bigoplus_{\sigma\colon E\to L}L\otimes_{\sigma,E}\mathcal A(V)\\
		&=\bigoplus_{\sigma\colon E\to L}L\otimes_{\QQ(V^\sigma)}\mathcal A(V^\sigma),
	\end{align*}
	where likewise
	\[
	\mathcal A(V)=E\otimes_{\OO_{E}[1/N]}\mathcal A(V).
	\]
	In particular we have a canonical decomposition
	\begin{equation}
		\mathcal A(W_\CC)=\bigoplus_{\sigma\colon E\to\CC} \mathcal A(V_{\CC}^\sigma).
		\label{eq:AWcomplexification}
	\end{equation}
	
	For any $\sigma$ and for any cohomological cuspidal automorphic representation $\pi$ with infinitesimal character $\omega^\sigma$ its cohomological type $\cohType(\pi)$ agrees with the complexification of a unique ${\rm t}\in{\rm tcT}_E^\sigma$. For each representative $\alpha=[(A^{({\rm t})},V^{({\rm t})})]_\cong\in{\rm t}$ we choose an isomorphism
	\begin{equation}
	  \iota_{\pi}^{\alpha}\colon A(V_\CC^{({\rm t})})\to \pi_\infty^{(K_\infty)}.
	  \label{eq:archimedeaneichlershimura}
	\end{equation}
        As before, $A(V_\CC^{({\rm t})})$ denotes the complexification of the locally algebraic cohomologically induced standard module $A(V^{({\rm t})})$ in the sense of Definition \ref{def:locallyalgebraiccomplexmodule}.

        \begin{lemma}\label{lem:iotasigmaisomorphism}
          %Let $V$ denote an essentially (conjugate) self-dual absolutely irreducible locally algebraic representation of $G$.
        For every $\sigma\colon E\to\CC$, each ${\rm t}\in{\rm tcT}_E^\sigma$ and each representative $\alpha=[(A^{({\rm t})},V^{({\rm t})})]_\cong\in{\rm t}$
        our choice of $\iota_{\pi}^{\alpha}$ and the isomomorphism $\vartheta_{\OO_{E^\sigma}[1/N]}^{\alpha}$ induce via the identification
	\[
	\bigoplus_{\begin{subarray}c\pi:\\\cohType(\pi)={\rm t}\end{subarray}} \pi^{(K_\infty)}=
	L_0^2(\GL_n(F)Z(\RR)^0 \backslash{}\GL_n(\Adeles_F);\omega^\sigma)_{\rm t}^{\infty,(K_\infty)}
	\]
	for every ${\rm t}\in{\rm tcT}_E^\sigma$ a $(\lieg_\CC,K_\infty)\times G(\Adeles_f)$-equivariant monomorphism
	\begin{equation}
	\iota^\alpha\colon
	\mathcal A(V_{\CC}^{({\rm t})})\to
	L_0^2(\GL_n(F)Z(\RR)^0 \backslash{}\GL_n(\Adeles_F);\omega^\sigma)_{\rm t}^\infty
        \label{eq:iotasigmaisomorphism}
	\end{equation}
	whose image consists of all $K$-finite automorphic cusp forms contained in the right hand side.
 	\end{lemma}

        \begin{proof}
        Consider the composition
	\begin{align*}
		\mathcal A(V_{\CC}^{({\rm t})})&=A(V_\CC^{({\rm t})})\otimes
		\varinjlim_{K_f}H_{\rm cusp}^{t_0}\left(X_G(K_f);\widetilde{V}_\CC^{({\rm t})}\right)\\
		&\cong
		A(V_\CC^{({\rm t})})\otimes\bigoplus_{\pi}H^{t_0}(\lieg,\lies, K;\pi_\infty\otimes V_\CC^{({\rm t})})\otimes\pi_f\\
		&\cong
		\bigoplus_{\begin{subarray}c\pi:\\\cohType(\pi)={\rm t}\end{subarray}}
		A(V_\CC^{({\rm t})})\otimes
		H^{t_0}(\lieg,S(\RR)^0K(\RR);\pi_\infty\otimes V_\CC^{({\rm t})})\otimes\pi_f&\text{(distr.\ of $\otimes,\oplus$)}\\
		&\cong
		\bigoplus_{\begin{subarray}c\pi:\\\cohType(\pi)\\{\rm t}\end{subarray}}
		\pi_\infty^{(K_\infty)}\otimes
		H^{t_0}(\lieg,S(\RR)^0K(\RR);\pi_\infty\otimes V_\CC^{({\rm t})})\otimes\pi_f&\text{(via $\iota_\pi^\alpha$)}\\
		&\cong
		\bigoplus_{\begin{subarray}c\pi:\\\cohType(\pi)={\rm t}\end{subarray}}
		H^{t_0}(\lieg,S(\RR)^0K(\RR);\pi_\infty\otimes V_\CC^{({\rm t})})\otimes\pi_\infty^{(K_\infty)}\otimes\pi_f&\text{(comm.\ of $\otimes$)}\\
		&\cong
		\bigoplus_{\begin{subarray}c\pi:\\\cohType(\pi)={\rm t}\end{subarray}}
		H^{t_0}(\lieg,S(\RR)^0K(\RR);\pi_\infty\otimes V_\CC^{({\rm t})})\otimes\pi^{(K_\infty)}.
	\end{align*}
	Now for each $\pi$ of cohomological type ${\rm t}$ the inverse of $\iota_\pi^{\alpha}$ composed with $\vartheta_\CC^{\alpha}$ induces an isomorphism
	\begin{align*}
		H^{t_0}(\lieg,S(\RR)^0K(\RR);\pi_\infty\otimes V_\CC^{({\rm t})})
		&=
		H^{t_0}(\lieg,S(\RR)^0K(\RR);\pi_\infty^{(K_\infty)}\otimes V_\CC^{({\rm t})})\\
		&\cong
		H^{t_0}(\lieg,S(\RR)^0K;A(V_\CC^{({\rm t})})\otimes V_\CC^{({\rm t})})\\
		&\cong
		\CC.
	\end{align*}
        By the definition of the right hand side in \eqref{eq:iotasigmaisomorphism} the first claim follows. The second claim follows by invoking the decomposition \eqref{eq:l02spacecohomologicaltypedecomposition} in Proposition \ref{prop:decompositionbycohomologicaltypes}.
        \end{proof}

        \begin{lemma}\label{lem:iotasigmauniqueness}
          For each ${\rm t}\in{\rm tcT}_E^\sigma$ and every representative $\alpha\in{\rm t}$ the isomorphism $\iota^{\alpha}$ in \eqref{eq:iotasigmaisomorphism} is unique up to scalars in $\sigma(\OO_{E}[1/N])^\times$.
        \end{lemma}

        \begin{proof}
        For the claimed uniqueness we only need to consider the dependence on $\iota_\pi^{\alpha}$, which is unique up to complex scalars $\lambda\in\CC^\times$. If we replace $\iota_\pi^{\alpha}$ by $\lambda\cdot\iota_\pi^{\alpha}$, then the above isomorphism
	\[
	\mathcal A(V_{\CC}^{({\rm t})})\to 
	\bigoplus_{\pi}
	H^{t_0}(\lieg,S(\RR)^0K(\RR);\pi_\infty\otimes V_\CC^{({\rm t})})\otimes\pi^{(K_\infty)}
	\]
	is multiplied in the $\pi$-th summand by $\lambda$. The identification
	\[
	H^{t_0}(\lieg,S(\RR)^0K(\RR);\pi_\infty\otimes V_\CC^{({\rm t})})\cong\CC
	\]
	is constructed via the inverse $\left(\iota_\pi^{\alpha}\right)^{-1}$, which therefore is multiplied by $\lambda^{-1}$. Hence $\iota^{\alpha}$ is unique up to $\sigma$-conjugate $1/N$-integral units as claimed.
        \end{proof}

        \begin{corollary}\label{cor:iotaisomorphism}
	  Choose for every ${\rm t}\in{\rm tcT}_E$ a representative $\alpha_{\rm t}=[(A^{({\rm t})},V^{({\rm t})})]_\cong\in{\rm t}$. Then for every embedding $\sigma\colon E\to\CC$ the $\sigma$-conjugate of $\alpha_{\rm t}$ is a representative $\alpha_{\rm t}^\sigma\in{\rm t}^\sigma\in{\rm tcT}_E^\sigma$. Put
          \[
          L_0^2(\GL_n(F)Z(\RR)^0 \backslash{}\GL_n(\Adeles_F);\omega_W)_{\rm t}^\infty
          :=
          \left(
          \widehat{\bigoplus_{\sigma\colon E\to\CC}}
          \widehat{\bigoplus_{\begin{subarray}c\pi:\\\cohType(\pi)={\rm t}^\sigma\end{subarray}}}\pi
          \right)^\infty.
          \]
          Then the sum of the $\iota^{\alpha_{\rm t}^\sigma}$ over all embeddings $\sigma\colon E\to\CC$ induces a $(\lieg_\CC,K_\infty)\times\prod_{{\rm t}\in{\rm tcT}_E}\mathcal H(W_\CC^{({\rm t})}))$-module monomorphism
	  \begin{equation}
	    \sum_{\sigma\colon E\to\CC}\iota^{\alpha_{\rm t}^\sigma}\colon
	    \mathcal A(W_\CC^{({\rm t})})\to
            L_0^2(\GL_n(F)Z(\RR)^0 \backslash{}\GL_n(\Adeles_F);\omega_W)_{\rm t}^\infty
	    \label{eq:ithglobaleichlershimura}
	  \end{equation}
          It is unique up to units in $\ZZ[1/N]$ and its image consists of all smooth $K$-finite automorphic cusp forms contained in the right hand side.

          The sum of the maps \eqref{eq:ithglobaleichlershimura} over ${\rm t}\in{\rm tcT}_E$ induces a monomorphism
          \begin{equation}
	    \iota^{(\alpha_{\rm t})_{\rm t}}:=
	    \sum_{\begin{subarray}c{\rm t}\in{\rm tcT}_E\\\sigma\colon E\to\CC\end{subarray}}
            \iota^{\alpha_{\rm t}^\sigma}\colon
            \sum_{{\rm t}\in{\rm tcT}_E}
            \mathcal A(W_\CC^{({\rm t})})\to L_0^2(\GL_n(F)Z(\RR)^0 \backslash{}\GL_n(\Adeles_F);\omega_W)^\infty,
	    \label{eq:globaleichlershimura}
	  \end{equation}
          of $(\lieg_\CC,K_\infty)\times\mathcal H(W_\CC)$-modules whose image consist of all smooth $K$-finite automorphic cusp forms contained in the right hand side.% Moreover, the left hand side of \eqref{eq:globaleichlershimura} and the isomorphism $\iota$ are independent of the choice of representative.
        \end{corollary}

        \begin{remark}
          The map \eqref{eq:ithglobaleichlershimura} depends on the choice of representatives $\alpha_{\rm t}\in{\rm t}$ and only given their collection $(\alpha_{\rm t})_{{\rm t}\in{\rm tcT}_E}$ the map \eqref{eq:ithglobaleichlershimura} is unique up to $1/N$-integral units. The same applies to the uniqueness of the map \eqref{eq:globaleichlershimura}. The dependence on representatives of these maps is closely related to periods, which we will discuss below.
        \end{remark}

        \begin{remark}\label{rmk:cohmologicaltypesoverEandQ}
          The choice of representatives $(\alpha_{\rm t})_{\rm t}$, ${\rm t}$ running through $\cohTypes^{\rm temp}_E(G,(V_E,\rho_G))$ conceptually corresponds to a choice of representatives for all cohomological types
          \[
          {\rm s}\in\cohTypes^{\rm temp}_\QQ(G,(\res_{E/\QQ}V_E,\res_{E/\QQ}\rho_G)).
          \]
          The correspondence being
          \[
          [(A^{({\rm t})},V^{({\rm t})})]_\cong\mapsto
          [(\res_{E/\QQ}A^{({\rm t})},\res_{E/\QQ}V^{({\rm t})})]_\cong,
          \]
          which lies at the heart of the above construction.
        \end{remark}

        \begin{proof}
          Remark that every cuspidal automorphic representation $\pi$ contributing to the right hand side in \eqref{eq:globaleichlershimura} has a cohomological type $\cohType(\pi)$ contained in ${\rm tcT}_E^\sigma$ for a unique $\sigma\colon E\to\CC$. Lemma \ref{lem:iotasigmaisomorphism} combined with the decompositions \eqref{eq:AWcomplexification} and \eqref{eq:l2cuspforms} in Proposition \ref{prop:decompositionbycohomologicaltypes} shows the claims, except uniqueness of \eqref{eq:ithglobaleichlershimura}. The latter statement follows from the uniqueness statement in Lemma \ref{lem:iotasigmauniqueness} observing that due to the coherence conditions imposed on $A(V^{({\rm t})})$, $\vartheta_{E^\sigma}^\alpha$ and $\iota_\pi^\alpha$ (cf.\ \eqref{eq:varthetacoherence}), the resulting scalar will by construction be a norm in the extension $E/\QQ$, hence lie in $\ZZ[1/N]^\times$.
        \end{proof}
	
	\subsubsection{The canonical integral structure of infinite level}

        \begin{definition}[Integral structures on spaces of automorphic cusp forms]\label{def:globalintegralstructure}
        Given an essentially (conjugate) self-dual absolutely irreducible locally algebraic representation $V$, fix a choice $(\alpha_{\rm t})_{\rm t}$ of representatives $\alpha_{\rm t}\in{\rm t}$ for ${\rm t}\in\cohTypes_E^{\rm temp}(G,(V_E,\rho_G))$.

        For each ${\rm t}\in\cohTypes_E^{\rm temp}(G,(V_E,\rho_G))$ the $(\lieg_\CC,K_\infty)$-module $\mathcal A(W_\CC^{\rm t})$ admits by construction a canonical integral structure which is given by the image of $\mathcal A(W_{\ZZ[1/N]}^{({\rm t})})$ under the canonical map
        \[
	\mathcal A(W_{\ZZ[1/N]}^{({\rm t})})\to\mathcal A(W_\CC^{({\rm t})}).
        \]
        Define via the canonical isomorphism \eqref{eq:globaleichlershimura} from Corollary \ref{cor:iotaisomorphism} a global integral structure on \eqref{eq:l2cuspforms} as
	\begin{flalign}
	  &L_0^2(\GL_n(F)Z(\RR)^0 \backslash{}\GL_n(\Adeles_F);\omega_W)_{\ZZ[1/N]}:=\label{eq:l2cuspformsqbar}\\
                &\bigoplus_{{\rm t}\in\cohTypes_E^{\rm temp}(G,(V_E,\rho_G))}
                \iota^{(\alpha_{\rm t})_{\rm t}}\left(
                \image\left(
                           {\mathcal A}(W_{\ZZ[1/N]}^{({\rm t})})\to
                           {\mathcal A}(W_{\CC}^{({\rm t})})
                \right)\right).
                \nonumber
	\end{flalign}
        \end{definition}

        \begin{remark}
          By construction, \eqref{eq:l2cuspformsqbar} is equipped with a canonical $(\lieg_{\ZZ[1/N]},K_{\ZZ[1/N]}\times\pi^G_{0,\ZZ[1/N]})$-action. Its ${\rm t}$-th summand with respect to the decomposition \eqref{eq:l02spacecohomologicaltypedecomposition} in Proposition \ref{prop:decompositionbycohomologicaltypes} admits a canonical action of a $\ZZ[1/N]$-integral Hecke algebra $\mathcal H(W_{\ZZ[1/N]}^{(\rm t)})=\res_{\OO_{E}/\ZZ}\mathcal H(V_{\OO_{E}[1/N]}^{({\rm t})})$ (cf.\ section \ref{sec:integralcohomology}).
        \end{remark}

        \begin{remark}
          By Corollary \ref{cor:iotaisomorphism} the integral structure \eqref{eq:l2cuspformsqbar} is uniquely determined by the choice of representatives $(\alpha_{\rm t})_{{\rm t}\in\cohTypes_E(G,(V_E,\rho_G))}$. Replacing $V$ by $V^\sigma$ for $\sigma\colon E\to\CC$ and $(\alpha_{\rm t})_{\rm t}$ by $(\alpha_{\rm t}^\sigma)_{{\rm t}\in\cohTypes_E(G,(V_E,\rho_G))}$ does not change the resulting integral structure.
        \end{remark}
	
	The canonical integral structure \eqref{eq:l2cuspformsqbar} satisfies the following properties.
	\begin{theorem}\label{thm:globalhalfintegralstructures}
          Let $F/\QQ$ be a totally real or CM number field. Assume that $(V_E,\rho_G)$ is an absolutely irreducible rational representation of $G=\res_{F/\QQ}\GL_n$ with field of rationality $E=\QQ(V_E)$ which is essentially selfdual for $F$ real and essentially conjugate selfdual if $F$ is CM. Fix moreover a lattice $V_{\OO_E[1/N]}\subseteq V_E$ which is normalized as in Lemma \ref{lem:VOnormalization}.

          Fix for every tempered cohomological type ${\rm t}\in\cohType_E^{\rm temp}(G,(V_E,\rho_G))$ a representative $\alpha_{\rm t}\in{\rm t}$ and put $\alpha_{{\rm t}^\sigma}:=\alpha_{\rm t}^\sigma$ for every embedding $\sigma\colon E\to\CC$.
          
          Then as a representation of $\GL_n(\Adeles_F)$, the space \eqref{eq:l2cuspforms} admits a canonical model \eqref{eq:l2cuspformsqbar} over $\ZZ[1/N]$ depending only on the choice of $(\alpha_{\rm t})_{\rm t}$ and the $1/N$-integral model $W_{\ZZ[1/N]}$ of $W_\QQ=\res_{E/\QQ} V_E$ induced by $V_{\OO_E[1/N]}$ with the following properties:
		\begin{itemize}
		        \item[(a)] The $1/N$-integral model \eqref{eq:l2cuspformsqbar} admits a canonical structure of $(\lieg_{\ZZ[1/N]},K_{\ZZ[1/N]},\pi_\circ^G)\times\prod_{{\rm t}\in{\rm tcT}_E}\mathcal H(W_{\ZZ[1/N]}^{({\rm t})})$-module.
			\item[(b)] The complexification
			\[
			\CC\otimes_{\ZZ[1/N]} L_0^2(\GL_n(F)Z(\RR)^0 \backslash{}\GL_n(\Adeles_F);\omega_W)_{\ZZ[1/N]}
			\]
			of \eqref{eq:l2cuspformsqbar} is naturally identified with the subspace of smooth $K$-finite vectors in \eqref{eq:l2cuspforms} as $(\lieg_{\CC},K(\RR))\times\prod_{{\rm t}\in{\rm tcT}_E}\mathcal H(W_{\CC}^{({\rm t})})$-module.
			\item[(c)] For each irreducible cuspidal $\pi$ occuring in \eqref{eq:l2cuspforms} the integral structure $\pi_{\OO_{\QQ(\pi)}[1/N]}$ induced on $\pi$ by the intersection of $\pi$ as a subspace of \eqref{eq:l2cuspforms} with the integral structure
			\[
			%L_0^2(\GL_n(F)Z(\RR)^0 \backslash{}\GL_n(\Adeles_F);\omega_W)_{\OO_{\QQ(\pi)}[1/N]}=
			\OO_{\QQ(\pi)}\otimes_\ZZ
			L_0^2(\GL_n(F)Z(\RR)^0 \backslash{}\GL_n(\Adeles_F);\omega_W)_{\ZZ[1/N]}
			\]
			has the following properties. For ${\rm t}=\cohType(\pi)$, if $\alpha_{\rm t}=[(A^{({\rm t})},V^{({\rm t})})]_\cong$, then the canonical morphism
			\begin{equation}
				H^{t_0}(\lieg,\lies, K;\pi_{\OO_{\QQ(\pi)}[1/N]}\otimes V_{\OO_{\QQ(\pi)}[1/N]}^{({\rm t})})
				%H^\bullet(\lieg,\lies, K;\pi^{(K_\infty)}\otimes V_\CC)
				\to
				\varinjlim_{K_f}H_{\rm cusp}^{t_0}(X_G(K_f);\widetilde{V}_\CC^{({\rm t})})
				\label{eq:canonicalmap}
			\end{equation}
			is injective and factors over $\varinjlim_{K_f}H^{t_0}(X_G(K_f);\widetilde{V}_{\OO_{\QQ(\pi)}[1/N]}^{({\rm t})})$. The monomorphism
			\[
			H^{t_0}(\lieg,\lies, K;\pi_{\OO_{\QQ(\pi)}[1/N]}\otimes V_{\OO_{\QQ(\pi)}[1/N]}^{({\rm t})})\to
			\]
			\[
			\image\left(
			\varinjlim_{K_f}H_{\rm cusp}^{t_0}(X_G(K_f);\widetilde{V}_{\OO_{\QQ(\pi)}[1/N]}^{({\rm t})})\to \varinjlim_{K_f}H_{\rm cusp}^{t_0}(X_G(K_f);\widetilde{V}_\CC^{({\rm t})})\right)[\pi_f]
			\]
			is an isomorphism.
			\item[(d)] For each $\tau\in\Aut(\CC/\QQ)$ and $\pi$ as in (c),
			\[
			(\pi_{\OO_{\QQ(\pi)}[1/N]})^\tau=
			\OO_{\QQ(\pi^\tau)}[1/N]
			\otimes_{\tau^{-1},\OO_{\QQ(\pi)}[1/N]}
			\pi_{\OO_{\QQ(\pi)}[1/N]}
			\]
			is the $\OO_{\QQ(\pi^\tau)}[1/N]$-integral structure on $\pi^\tau$ from (c) for the $\tau$-conjugated representative $\alpha_{\rm t}^\tau\in{\rm t}^\tau$.
		      \item[(e)] Taking $(\lieg,\lies,K)$-cohomology with coefficients in
                        \[
                        W_{\ZZ[1/N]}^{({\rm t})}=\res_{\QQ(V^{({\rm t})})/\QQ} V^{({\rm t})}_{\OO_{\QQ(V^{({\rm t})})}[1/N]}
                        \]
                        for $\alpha_{\rm t}=[(A^{({\rm t})},V^{({\rm t})})]_\cong\in{\rm t}\in\cohTypes_E(G,(V_E,\rho_G))$, then the canonical morphism
                        sends the $\ZZ[1/N]$-structure \eqref{eq:l2cuspformsqbar} onto the natural $\ZZ[1/N]$-structure
		        \begin{flalign}
			&H_{\rm cusp}^{t_0}(\GL_n(F)\backslash{}\GL_n(\Adeles_F)/S(\RR)^0 K(\RR);\widetilde{W}_{\ZZ[1/N]}^{({\rm t})}):=
                        \nonumber\\
			&{\rm{image}}(
			  H^{t_0}(\GL_n(F)\backslash{}\GL_n(\Adeles_F)/S(\RR)^0 K(\RR);\widetilde{W}_{\ZZ[1/N]}^{({\rm t})})\to
			\label{eq:hcusp}\\
			&H_{\rm cusp}^{t_0}(\GL_n(F)\backslash{}\GL_n(\Adeles_F)/S(\RR)^0 K(\RR);\widetilde{W}_\CC^{({\rm t})})).\nonumber
			\end{flalign}
			on cuspidal cohomology in bottom degree $t_0$.
		\end{itemize}
		%Minimality of the half-integral model \eqref{eq:l2cuspformsqbar} is understood with respect to \lq{}$\subseteq$\rq{} subject to condition (f).
	\end{theorem}

        \begin{remark}\label{rmk:globalgeneralreductiveG}
          Theorem \ref{thm:globalhalfintegralstructures} extends Theorem 8.29 in \cite{januszewskirationality} to integral structures and also takes all possible parity conditions at $\infty$ into account. Similarly to Theorem 8.17 of loc.\ cit.\ Theorem \ref{thm:globalhalfintegralstructures} has an extension to general connected reductive groups $G/\QQ$ with a slightly weaker statement. Since it is straightforward to translate our proof below to general connected reductive $G/\QQ$, we only emphasize the main differences here:

          If the extremal weights of $V$ are not regular, then it is unkown if cuspidal cohomology admits a model over $\overline{\QQ}$. Therefore in such cases we cannot construct a $1/N$-integral structure on spaces of global cusp forms. Moreover, the field of definition of a maximal compact group $K_\infty\subseteq G(\RR)$ in general is a non-trivial number field $\QQ_K/\QQ$. This has to be taken into account: All rationality statements and integrality statements will be above $\QQ_K$, not above $\QQ$. Moreover, of $G$ is not quasi-split, then the field of rationality of $V$ may differ from a field of definition of $V$, enlarging the ground field, hence the base ring, further.
        \end{remark}
	
	\begin{remark}
		In the statement of the theorem \lq{}bottom degree\rq{} may be replaced by \lq{}top degree\rq{} to obtain another canonical integral structure. For the conjectural relation between both integral structures we refer to Remark \ref{rmk:venkatesh} below.
	\end{remark}
	
	\begin{remark}
		In (e) we used the convention
		\begin{flalign*}
		&H_{\rm cusp}^{\bullet}(\GL_n(F)\backslash{}\GL_n(\Adeles_F)/S(\RR)^0 K(\RR);\widetilde{W}_{\ZZ[1/N]}^{({\rm t})}):=
		\\
		&\varinjlim_{K_f}H_{\rm cusp}^{\bullet}(\GL_n(F)\backslash{}\GL_n(\Adeles_F)/S(\RR)^0 K(\RR)K_f;\widetilde{W}_{\ZZ[1/N]}^{({\rm t})}).
		\end{flalign*}
	\end{remark}
	
	\begin{remark}
		The canonical map
		\[
		H^{t_0}(\lieg,\lies, K;\mathcal A(W_{\ZZ[1/N]}^{({\rm t})})\otimes W_{\ZZ[1/N]}^{({\rm t})})\to
		H_{\rm cusp}^{t_0}(\GL_n(F)\backslash{}\GL_n(\Adeles_F)/S(\RR)^0 K(\RR);\widetilde{W}_{\ZZ[1/N]}^{({\rm t})})
		\]
		underlying (e) need not be injective.
	\end{remark}

        \begin{remark}
          As explained in remark \ref{rmk:cohmologicaltypesoverEandQ}, conceptually, the choice of $\alpha_{\rm t}\in{\rm t}$ for
          \[{\rm t}\in\cohTypes_E^{\rm temp}(G,(V_E,\rho_G))\]
          corresponds bijectively to a choice of $\beta_{\rm s}\in{\rm s}$ for
          \[{\rm s}\in\cohTypes_E^{\rm temp}(G,(\res_{E/\QQ}V_E,\res_{E/\QQ}\rho_G)).\]
        \end{remark}
	
	\begin{remark}
	  The potential existence of congruences gives rise to the following phenomenon. Recall that since $G$ is quasi-split, there is a $\QQ$-Borel subgroup $B\subseteq G$ and the $B$-highest weight $\lambda$ of $(V_E,\rho_E)$ is defined over $E$. Fix a maximal $\QQ$-torus $T\subseteq B$. If $\lambda$ satisfies a congruence
          \[
          \lambda\equiv \lambda^\tau\pmod{\mathfrak{p}}
          \]
          and $\lambda^\tau\neq\lambda$ for some $\tau\in\Aut(F/\QQ)$ and $\mathfrak{p}\subseteq\OO_{E}[1/N]$ a non-zero prime ideal. If additonally for a minimal $K$-stable parabolic $Q$ the $K$-orbits of $Q^\sigma$ and $Q^w$ agree, then by base change, for each tempered cohomological type ${\rm t}\in\cohType_E(G,(V_E,\rho_G))$ the two half-integral models
          \begin{equation}
          A(V^{({\rm t})})_{\OO_{\QQ(V)}[1/N]}\quad\text{and}\quad
	  A(V^{({\rm t}^\tau)})_{\OO_{\QQ(V)}[1/N]}
          \label{eq:congruentmodulexampe}
	  \end{equation}
	  are isomorphic modulo $\mathfrak{p}$ as locally algebraic $(\lieg,K)$-modules. In general, the two half-integral models
          \begin{equation}
          V_{\OO_{\QQ(V)}[1/N]}^{({\rm t})}\quad\text{and}\quad V_{\OO_{\QQ(V)}[1/N]}^{({\rm t}^\tau)}
          \label{eq:congruentcoefficientexample}
	  \end{equation}
          need not be isomorphic modulo $\mathfrak{p}$ as locally algebraic representations. However, they both occur in $W_{\ZZ[1/N]}^{({\rm t})}$ after base change to $\OO_{\QQ(V)}[1/N]$. Therefore, modulo $\mathfrak{p}$, the $(\lieg,\lies,K)$-cohomologies of the two modules \eqref{eq:congruentmodulexampe} with coefficients chosen from \eqref{eq:congruentcoefficientexample} are non-trivial in all $4$ possible combinations. Consequently, modulo $\mathfrak{p}$ the multiplicities in
		\[
		H^{t_0}(\lieg_{\OO_{\QQ(V)}/\mathfrak{p}},\lies, K_{\OO_{\QQ(V)}/\mathfrak{p}}; L_0^2(\GL_n(F)Z(\RR)^0 \backslash{}\GL_n(\Adeles_F);\omega_W)_{\OO_{\QQ(V)}/\mathfrak{p}}\otimes W_{\OO_{\QQ(V)}/\mathfrak{p}}^{({\rm t})})
		\]
		may be strictly larger than in \eqref{eq:hcusp}, the latter considered modulo $\mathfrak{p}$ as well.

                In particular, in the setting of (e), the $(\lieg,\lies,K)$-cohomology of the global integral structure with coeffients in $W_{\ZZ[1/N]}^{({\rm t})}$ may contain $\mathfrak{p}$-torsion reflecting such congruences. This torsion contribution is finite at finite level and therefore absorbed in the kernel of the localization map when $\mathfrak{p}$ is inverted.
	\end{remark}
	
	\begin{proof}
          Statement (a) is true by construction.
	  Statement (b) is a consequence of Corollary \ref{cor:iotaisomorphism} by the construction of the $1/N$-integral structure \eqref{eq:l2cuspformsqbar}.
		
		We next prove (e). Let ${\rm t}\in\cohTypes_E(G,(V_E,\rho_G))$ and $\alpha_{\rm t}=[(A^{({\rm t})},V^{({\rm t})})]_\cong$. We compute $(\lieg_\CC,S(\RR)^0 K(\RR))$-cohomology of the relevant spaces:
		\begin{align*}
			&
			H^{t_0}\left(\lieg_\CC,S(\RR)^0 K(\RR); L_0^2(\GL_n(F)Z(\RR)^0 \backslash{}\GL_n(\Adeles_F);\omega_W)^\infty\otimes W_\CC^{({\rm t})}\right)\\
			=&
			\bigoplus_{\sigma,\tau}
			H^{t_0}\left(\lieg_\CC,S(\RR)^0 K(\RR); L_0^2(\GL_n(F)Z(\RR)^0 \backslash{}\GL_n(\Adeles_F);\omega^\sigma)^\infty\otimes V_\CC^{({\rm t}^\tau)}\right)\\
			=&
			\bigoplus_{\sigma}
			H^{t_0}\left(\lieg_\CC,S(\RR)^0 K(\RR); L_0^2(\GL_n(F)Z(\RR)^0 \backslash{}\GL_n(\Adeles_F);\omega^\sigma)^\infty\otimes V_\CC^{({\rm t}^\sigma)}\right)\\
			=&
			\bigoplus_{\sigma}
			\varinjlim_{K_f}H_{\rm cusp}^{t_0}\left(X_G(K_f);\widetilde{V}^{({\rm t}^\sigma)_{\CC}}\right)\\
			=&
			\varinjlim_{K_f}H_{\rm cusp}^{t_0}\left(X_G(K_f);\widetilde{W}^{({\rm t})}_{\CC}\right),
		\end{align*}
		since $(\lieg_\CC,S(\RR)^0 K(\RR))$-cohomology vanishes unless the infinitesimal characters $\omega^\sigma$ are dual to the infinitesimal character of $V^\tau$. Moreover those are pairwise distinct since $G$ is quasi-split: the field of definition of $V$ agrees with the field of rationality (cf.\ Corollary \ref{cor:quasisplitrationality}).
		
		Similarly we obtain
		\begin{align*}
			&
			H^{t_0}\left(\lieg_\CC,S(\RR)^0 K(\RR); \mathcal A(W^{({\rm t})}_\CC)\otimes W^{({\rm t})}_\CC\right)\\
			=&
			\bigoplus_{\sigma,\tau}
			H^{t_0}\left(\lieg_\CC,S(\RR)^0K(\RR); \mathcal A(V^{({\rm t}^\sigma)}_\CC)\otimes V^{({\rm t}^\tau)}_\CC\right)\\
			=&
			\bigoplus_{\sigma,\tau}
			H^{t_0}\left(\lieg_\CC,S(\RR)^0K(\RR); A(V^{({\rm t}^\sigma)})_\CC\otimes V^{({\rm t}^\tau)}_\CC\right)\otimes
			\varinjlim_{K_f}H_{\rm cusp}^{t_0}\left(X_G(K_f);\widetilde{V}^{({\rm t}^\sigma)}_\CC\right)\\
			=&
			\bigoplus_{\sigma}
			\underbrace{H^{t_0}\left(\lieg_\CC,S(\RR)^0K(\RR); A(V^{({\rm t}^\sigma)}_\CC)\otimes V^{({\rm t}^\sigma)}_\CC\right)}_{\cong\CC\;\text{via}\;\vartheta_\CC^{\alpha_{{\rm t}^\sigma}}}
			\otimes
			\varinjlim_{K_f}H_{\rm cusp}^{t_0}\left(X_G(K_f);\widetilde{V}^{({\rm t}^\sigma)}_\CC\right)\\
			\cong&
			\bigoplus_{\sigma}
			\varinjlim_{K_f}H_{\rm cusp}^{t_0}\left(X_G(K_f);\widetilde{V}^{({\rm t}^\sigma)}_\CC\right)\\
			=&
			\varinjlim_{K_f}H_{\rm cusp}^{t_0}\left(X_G(K_f);\widetilde{W}^{({\rm t})}_\CC\right).
		\end{align*}
		We remark that this composition depends the choice of isomorphisms $\vartheta_\CC^{\alpha_{\rm t}^\sigma}$, which in turn are chosen coherently depending on the isomorphism $\vartheta_{\OO_{E}[1/N]}^{\alpha_{\rm t}^\sigma}$ in \eqref{eq:varthetachoice}. Therefore, this above isomorphism is the $\CC$-base change of the map
		\begin{align*}
			&H^{t_0}\left(\lieg,\lies,K; \mathcal A(W_{\ZZ[1/N]}^{({\rm t})})\otimes W_{\ZZ[1/N]}^{({\rm t})}\right)\\
			=\;&
			H^{t_0}\left(\lieg,\lies,K; \res_{\OO_{E}/\ZZ}\mathcal A(V_{\OO_{E}[1/N]}^{({\rm t})})\otimes\res_{\OO_{E}/\ZZ} V_{\OO_{E}[1/N]}^{({\rm t})}\right)\\
			\to\;&
			H^{t_0}\left(\lieg,\lies,K;  \bigoplus_{\sigma}\mathcal A(V_{\OO_{E}[1/N]}^{({\rm t}^\sigma)})\otimes V_{\OO_{E}[1/N]}^{({\rm t}^\sigma)}\right)\\
			=\;&
			\bigoplus_{\sigma}
			\underbrace{H^{t_0}\left(\lieg,\lies, K; \mathcal A(V_{\OO_{E}[1/N]}^{({\rm t}^\sigma)})\otimes V_{\OO_{E}[1/N]}^{({\rm t}^\sigma)}\right)}_{\cong\OO_{E}[1/N]^\sigma\;\text{via}\;\vartheta_{\OO_{E}[1/N]}^{\alpha_{\rm t}^\sigma}}\\
			\to\;&
			\bigoplus_{\sigma}
			\varinjlim_{K_f}H^{t_0}\left(X_G(K_f);\widetilde{V}_{\OO_{E}[1/N]}^{({\rm t}^\sigma)}\right).
		\end{align*}
		The last direct sum contains
		\[
		\varinjlim_{K_f}H^{t_0}\left(X_G(K_f);\widetilde{W}_{\ZZ[1/N]}^{({\rm t})}\right)=
		\res_{\OO_{E}/\ZZ} \varinjlim_{K_f}H^{t_0}\left(X_G(K_f);\widetilde{V}_{\OO_{E}[1/N]}^{({\rm t})}\right)
		\]
		as a via $\sigma:E\to\CC$ diagonally embedded copy of $\varinjlim_{K_f}H^{t_0}\left(X_G(K_f);\widetilde{V}_{\OO_{E}[1/N]}^{({\rm t})}\right)$, which is identified with the image of $H^{t_0}\left(\lieg,\lies, K; \mathcal A(W_{\ZZ[1/N]}^{({\rm t})})\otimes W_{\ZZ[1/N]}^{({\rm t})}\right)$.
		
		This proves statement (e) of the theorem conditionally on the commutativity of the square
		\[
		\begin{CD}
			H^{t_0}\left(\lieg,\lies, K; \mathcal A(W_\CC^{({\rm t})})\otimes W_\CC^{({\rm t})}\right)
			@>\iota^{\alpha_{\rm t}}>>
			H^{t_0}\left(\lieg_\CC,S(\RR)^0 K(\RR); L_0^2(\GL_n(F)Z(\RR)^0 \backslash{}\GL_n(\Adeles_F);\omega_W)^{\infty}\otimes W_\CC^{({\rm t})}\right)\\
			@V{\sum_{\sigma}\vartheta_\CC^{\alpha_{\rm t}^\sigma}}VV @VVV\\
			\varinjlim\limits_{K_f}H^{t_0}\left(X_G(K_f);\widetilde{W}_\CC^{({\rm t})}\right)
			@=
			\varinjlim\limits_{K_f}H^{t_0}\left(X_G(K_f);\widetilde{W}_\CC^{({\rm t})}\right),
		\end{CD}
		\]
		whose commutativity is a consequence of the commutativity of
		\[
		\begin{CD}
			H^{t_0}\left(\lieg,\lies, K; \mathcal A(V_\CC^{({\rm t}^\sigma)})\otimes V_\CC^{({\rm t}^\sigma)}\right)
			@>\iota^{\alpha_{\rm t}^\sigma}>>
			H^{t_0}\left(\lieg_\CC,S(\RR)^0 K(\RR); L_0^2(\GL_n(F)Z(\RR)^0 \backslash{}\GL_n(\Adeles_F);\omega^\sigma)^{\infty}\otimes V^{({\rm t}^\sigma)}\right)\\
			@V{\vartheta_\CC^{\alpha_{\rm t}^\sigma}}VV @VVV\\
			\varinjlim_{K_f}H^{t_0}\left(X_G(K_f);\widetilde{V}_\CC^{({\rm t}^\sigma)}\right)
			@=
			\varinjlim_{K_f}H^{t_0}\left(X_G(K_f);\widetilde{V}_\CC^{({\rm t}^\sigma)}\right)
		\end{CD}
		\]
		for every embedding $\sigma:E\to\CC$, which in turn is a consequence of Schur's Lemma and the definitions: For every cuspidal automorphic representation $\pi$ contributing to
                \[\varinjlim_{K_f}H^{t_0}\left(X_G(K_f);\widetilde{V}_\CC^{({\rm t}^\sigma)}\right),\]
                i.\,e.\ $\cohType(\pi)={\rm t}^\sigma$, consider the subspace
		\begin{equation}
			H^{t_0}\left(\lieg_\CC,S(\RR)^0K(\RR);\pi_\infty^{(K_\infty)}\otimes V^{({\rm t}^\sigma)}_\CC\right)\otimes\pi_f\subseteq
			\varinjlim_{K_f}H^{t_0}\left(X_G(K_f);\widetilde{V}^{({\rm t}^\sigma)}_\CC\right).
			\label{eq:canonicalEichlerShimurainclusionforpi}
		\end{equation}
		Then on the corresponding subspace
		\begin{flalign*}
		  &H^{t_0}\left(\lieg_\CC,S(\RR)^0K(\RR); A(V^{({\rm t}^\sigma)}_\CC)\otimes H^{t_0}\left(\lieg_\CC,S(\RR)^0K(\RR);\pi_\infty^{(K_\infty)}\otimes V^{({\rm t}^\sigma)}_\CC\right)\otimes\pi_f\otimes V^{({\rm t}^\sigma)}_\CC\right)\\
                  &\subseteq H^{t_0}\left(\lieg_\CC,S(\RR)^0K(\RR); \mathcal A(V^{({\rm t}^\sigma)}_\CC)\otimes V^{({\rm t}^\sigma)}_\CC\right)
		\end{flalign*}
		the left vertical map $\vartheta_\CC^{\alpha_{\rm t}^\sigma}$ is the composition of the isomorphism
		\begin{align*}
		  &
		  H^{t_0}\left(\lieg_\CC,S(\RR)^0K(\RR); A(V^{({\rm t}^\sigma)}_\CC)\otimes
                  H^{t_0}\left(\lieg_\CC,S(\RR)^0K(\RR);\pi_\infty^{(K_\infty)}\otimes V^{({\rm t}^\sigma)}_\CC\right)\otimes\pi_f\otimes V^{({\rm t}^\sigma)}_\CC\right)\\
		  \cong\;&
		  \underbrace{H^{t_0}\left(\lieg_\CC,S(\RR)^0K(\RR); A(V^{({\rm t}^\sigma)}_\CC)\otimes V^{({\rm t}^\sigma)}_\CC\right)}_{\cong\CC\;\text{via}\;\vartheta_\CC^{\alpha_{\rm t}^\sigma}}
		  \otimes H^{t_0}\left(\lieg_\CC,S(\RR)^0K(\RR);\pi_\infty^{(K_\infty)}\otimes V^{({\rm t}^\sigma)}_\CC\right)\otimes\pi_f\\
		  \cong\;&
		  H^{t_0}\left(\lieg_\CC,S(\RR)^0K(\RR);\pi_\infty^{(K_\infty)}\otimes V^{({\rm t}^\sigma)}_\CC\right)\otimes\pi_f
		\end{align*}
		induced by $\vartheta_\CC^{\alpha_{\rm t}^\sigma}$ with the canonical inclusion \eqref{eq:canonicalEichlerShimurainclusionforpi} into
                \[\varinjlim_{K_f}H^{t_0}\left(X_G(K_f);\widetilde{V}^{({\rm t}^\sigma)}_\CC\right).\]
		
		The top horizontal map on the same subspace is the map
		\begin{align*}
		  &
		  H^{t_0}\left(\lieg_\CC,S(\RR)^0K(\RR); A(V^{({\rm t}^\sigma)}_\CC)\otimes
                  H^{t_0}\left(\lieg_\CC,S(\RR)^0K(\RR);\pi_\infty^{(K_\infty)}\otimes V^{({\rm t}^\sigma)}_\CC\right)\otimes\pi_f\otimes V^{({\rm t}^\sigma)}_\CC\right)\\
		  \cong\;&
		  H^{t_0}(\lieg_\CC,S(\RR)^0K(\RR); \underbrace{A(V^{({\rm t}^\sigma)}_\CC)}_{\cong\pi_\infty^{(K_\infty)}\;\text{via}\;\iota_\pi^{\alpha_{\rm t}}}\otimes\pi_f
		  \otimes V^{({\rm t}^\sigma)}_\CC)
		  \otimes H^{t_0}\left(\lieg_\CC,S(\RR)^0K(\RR);\pi_\infty^{(K_\infty)}
		  \otimes V^{({\rm t}^\sigma)}_\CC\right)\\
		  \cong\;&
		  H^{t_0}\left(\lieg_\CC,S(\RR)^0K(\RR); \pi_\infty^{(K_\infty)}\otimes\pi_f
		  \otimes V^{({\rm t}^\sigma)}_\CC\right)
		  \otimes H^{t_0}(\lieg_\CC,S(\RR)^0K(\RR);\underbrace{\pi_\infty^{(K_\infty)}}_{\cong A(V^{({\rm t}^\sigma)}_\CC)\;\text{via}\;\left(\iota_\pi^{\alpha_{\rm t}}\right)^{-1}}\otimes V^{({\rm t}^\sigma)}_\CC)\\
		  \cong\;&
		  H^{t_0}\left(\lieg_\CC,S(\RR)^0K(\RR); \pi_\infty^{(K_\infty)}\otimes\pi_f
		  \otimes V^{({\rm t}^\sigma)}_\CC\right)
		  \otimes \underbrace{H^{t_0}\left(\lieg_\CC,S(\RR)^0K(\RR);A(V^{({\rm t}^\sigma)}_\CC)\otimes V^{({\rm t}^\sigma)}_\CC\right)}_{\cong\CC\;\text{via}\;\vartheta_\CC^{\alpha_{\rm t}^\sigma}}\\
		  \cong\;&
		  H^{t_0}\left(\lieg_\CC,S(\RR)^0K(\RR); \pi_\infty^{(K_\infty)}\otimes\pi_f
		  \otimes V^{({\rm t}^\sigma)}_\CC\right)
		\end{align*}
		composed with the respective inclusion of $\pi_\infty^{(K_\infty)}\otimes\pi_f$ into the space of smooth cusp forms. On the latter space the right vertical map is the canonical inclusion \eqref{eq:canonicalEichlerShimurainclusionforpi}.
		
		By Schur's Lemma applied to the irreducible Hecke module $\pi_f$ it is sufficient to exhibit a non-zero element
		\[
		x\in H^{t_0}\left(\lieg_\CC,S(\RR)^0K(\RR); A(V^{({\rm t}^\sigma)}_\CC)\otimes
                H^{t_0}\left(\lieg_\CC,S(\RR)^0K(\RR);\pi_\infty^{(K_\infty)}\otimes V^{({\rm t}^\sigma)}_\CC\right)\otimes\pi_f\otimes V^{({\rm t}^\sigma)}_\CC\right)
		\]
		which maps to the same element under both maps under consideration. Pick a representative
		\[
		\sum_{i}\omega_i\otimes a_i\otimes v_i\in\bigwedge^{t_0}\left(\lieg_\CC/\lies_\CC+\liek_\CC\right)^\ast\otimes A(V^{({\rm t}^\sigma)}_\CC)\otimes V^{({\rm t}^\sigma)}_\CC
		\]
		of a non-zero cohomology class in $H^{t_0}\left(\lieg_\CC,S(\RR)^0K(\RR); A(V^{({\rm t}^\sigma)}_\CC)\otimes V^{({\rm t}^\sigma)}_\CC\right)$ and any non-zero element $\varphi_f\in\pi_f$. Then the element
		\[
		\sum_{i}\omega_i\otimes a_i\left(\sum_{j}\omega_j\otimes\iota_\pi(a_j)\otimes v_j\right)\otimes\varphi_f\otimes v_i
		\]
		represents a non-zero element $x$ with the desired property. This concludes the proof of the claimed commutativity of the above square and the proof of statement (e) is complete.
		
		\smallskip
		Statement (c) is proved as follows. Consider a normal hull $E'/\QQ$ of $E$ with integer ring $\OO_{E'}$. Then a rational prime $p$ ramifies in $E'/\QQ$ if and only if it ramifies in $E/\QQ$. Therefore inverting $d_{E'}$ is equivalent to inverting $d_E$ and the discussion of the isotypic component of $\pi$ in the proof of (e) shows the compatibility of the inclusion
		\begin{equation}
			W^{({\rm t})}_{\OO_{E'}[1/N]}=
			\OO_{E'}\otimes_\ZZ V^{(\rm t)}_{\OO_{E}[1/N]}
			\to \prod_{\sigma:E\to E'}V^{({\rm t}^\sigma)}_{\OO_{E'}[1/N]}
			\label{eq:WtodiagVsigma}
		\end{equation}
		with the isomorphism
		\[
		W^{({\rm t})}_\CC\to\prod_{E\to\CC} V^{({\rm t}^\sigma)}_\CC
		\]
		once we fix an embedding $E'\to\CC$ extending a given embedding $E\to\CC$.
		
		Likewise, we have a canonical monomorphism
		\begin{equation}
			\mathcal A(W^{({\rm t})}_{\OO_{E'}[1/N]})=
			\OO_{E'}\otimes_\ZZ \mathcal A(V^{({\rm t})}_{\OO_{E}[1/N]})
			\to \prod_{\sigma:E\to E'}\mathcal A(V^{({\rm t}^\sigma)}_{\OO_{E'}[1/N]}),
			\label{eq:AWtodiagAVsigma}
		\end{equation}
		again compatible with the canonical isomorphism \eqref{eq:AWcomplexification}:
		\[
		\mathcal A(W^{({\rm t})}_\CC)=\bigoplus_{\sigma:E\to\CC}\mathcal A(V^{({\rm t}^\sigma)}_\CC).
		\]
		Observe that \eqref{eq:WtodiagVsigma} and \eqref{eq:AWtodiagAVsigma} are isomorphisms by Lemma 6.3.2 in \cite{hayashijanuszewski}.

                For any cuspidal automorphic representation $\pi$ of cohomological type $\cohType(\pi)={\rm t}^\sigma$, we have $\QQ(V^{({\rm t}^\sigma)})=E^\sigma\subseteq\QQ(\pi)$. Since the inclusion
		\[
		j\colon\pi^{(K_\infty)}\to \mathcal A(W^{({\rm t})}_\CC)
		\]
		by construction factors over $\mathcal A(V^{({\rm t}^\sigma)}_\CC)$, the integral structure $\pi_{\OO_{\QQ(\pi)}[1/N]}$, which is the intersection of the image of $j$ with $\mathcal A(W^{({\rm t})}_{\OO_{\QQ(\pi)}[1/N]})$, factors over $\mathcal A(V^{({\rm t}^\sigma)}_{\OO_{\QQ(\pi)}[1/N]})$ in accordance with \eqref{eq:AWtodiagAVsigma}. Therefore, the $(\lieg,\lies, K)$-cohomology of the integral model $\pi_{\OO_{\QQ(\pi)}[1/N]}$ with coefficients in $V^{({\rm t}^\sigma)}_{\OO_{\QQ(\pi)}[1/N]}$ canonically maps into the $(\lieg,\lies, K)$-cohomology of $\mathcal A(V^{({\rm t}^\sigma)}_{\OO_{\QQ(\pi)}[1/N]})$ with the same coefficients, which in turn identifies with
                \[
                \varinjlim\limits_{K_f}H^{t_0}(X_G(K_f);\widetilde{V}^{{\rm t}^\sigma}_{\OO_{\QQ(\pi)}[1/N]}).
                \]
                This proves the first statement that \eqref{eq:canonicalmap} factors as claimed. The second statement follows since \eqref{eq:AWtodiagAVsigma} is an isomorphism.
		
		Statement (d) is an immedate consequences of the definitions. This concludes the proof.
	\end{proof}

        \subsubsection{The canonical integral structure of finite level}\label{sec:canonicalintegralstructuresKf}

	%\subsection{Canonical $1/N$-integral structures on spaces of automorphic cusp forms of finite level}
        
        Instead of considering the colimit over all finite levels $K_f\subseteq G(\Adeles_F)$, we obtain similar and slightly stronger results at finite level for the following spaces.
        
        \begin{variant}
          Replace \eqref{eq:l2cuspforms} by
          \begin{flalign}
		&L_0^2(\GL_n(F)Z(\RR)^0 \backslash{}\GL_n(\Adeles_F)/K_f;\omega_W)^\infty
		\;=\label{eq:l2cuspformsKf}\\
	        &\bigoplus_{\sigma:E\to\CC}
	        L_0^2(\GL_n(F)Z(\RR)^0 \backslash{}\GL_n(\Adeles_F)/K_f;\omega^\sigma)^\infty,\nonumber     
	  \end{flalign}
          the spaces $\mathcal A(V_{\OO_{E^\sigma}[1/N]}^{({\rm t})})$ and $\mathcal A(W_{\ZZ[1/N]}^{({\rm t})})$ by
          \[
	  \mathcal A(K_f; V_{\OO_{E^\sigma}[1/N]}^{({\rm t})}):=
	  A(V^{({\rm t})})_{\OO_{E^\sigma}[1/N]}
          \otimes_{\OO_{E^\sigma}[1/N]}
	  H_{\rm cusp}^{t_0}\left(X_G(K_f);\widetilde{V}_{\OO_{E^\sigma}[1/N]}^{({\rm t})}\right),
	  \]
          and
	  \[
	  \mathcal A(K_f; W_{\ZZ[1/N]}^{({\rm t})}):=
	  \res_{\OO_{E^\sigma}/\ZZ}\mathcal A_{\OO_{E^\sigma}[1/N]}(V_{\OO_{E^\sigma}[1/N]}^{({\rm t})},K_f),
	  \]
          respectively. We define their corresponding scalar extensions to $\ZZ[1/N]$-algebras $A$ mutatis mutandis as in the level-free case. The decomposition \eqref{eq:AWcomplexification} for finite level $K_f$ then reads
          \begin{equation}
	    \mathcal A(K_f; W_\CC)=\bigoplus_{\sigma\colon E\to\CC} \mathcal A(K_f; V_{\CC}^\sigma),
	    \label{eq:AWcomplexificationKf}
	  \end{equation}
          and it is compatible with \eqref{eq:AWcomplexification} in the sense that passing to $K_f$-invariants in the latter canonically identifies with the former.

          Likewise we put for every ${\rm t}\in{\rm tcT}_E^\sigma$
          \[
          L_0^2(\GL_n(F)Z(\RR)^0 \backslash{}\GL_n(\Adeles_F)/K_f;\omega^\sigma)^{\infty}_{{\rm t}}:=
          \left(\widehat{\bigoplus\limits_{\begin{subarray}c\pi:\\\cohType(\pi)={\rm t}\\\pi_f^{K_f}\neq 0\end{subarray}}} \pi^{K_f}\right)^\infty.
          \]
          Then the analogue of \eqref{eq:l02spacecohomologicaltypedecomposition} is
          \begin{equation}
            L_0^2(\GL_n(F)Z(\RR)^0 \backslash{}\GL_n(\Adeles_F)/K_f;\omega^\sigma)^\infty=
            \widehat{\bigoplus_{{\rm t}\in{\rm tcT}_E^\sigma}}
            L_0^2(\GL_n(F)Z(\RR)^0 \backslash{}\GL_n(\Adeles_F)/K_f;\omega^\sigma)^\infty_{{\rm t}},
            \label{eq:l02spacecohomologicaltypedecompositionKf}
          \end{equation}
          cf.\ Proposition \ref{prop:decompositionbycohomologicaltypes}.
        \end{variant}

        With this notation Lemma \ref{lem:iotasigmaisomorphism} and its proof translate to finite level and give
        \begin{lemma}\label{lem:iotasigmaisomorphismKf}
          For every $\sigma\colon E\to\CC$, each ${\rm t}\in{\rm tcT}_E^\sigma$ and each representative $\alpha=[(A^{({\rm t})},V^{({\rm t})})]_\cong\in{\rm t}$
          our choice of $\iota_{\pi}^{\alpha}$ and the isomomorphism $\vartheta_{\OO_{E^\sigma}[1/N]}^{\alpha}$ induce via the identification
	  \[
	  \bigoplus_{\begin{subarray}c\pi:\\\cohType(\pi)={\rm t}\\\pi_f^{K_f}\neq 0\end{subarray}} \pi^{(K_\infty),K_f}=
	  L_0^2(\GL_n(F)Z(\RR)^0 \backslash{}\GL_n(\Adeles_F)/K_f;\omega^\sigma)_{\rm t}^{\infty,(K_\infty)}
	  \]
	  for every ${\rm t}\in{\rm tcT}_E^\sigma$ a $(\lieg_\CC,K_\infty)\times G(\Adeles_f)$-equivariant monomorphism
	  \begin{equation}
	    \iota^{\alpha}_{K_f}\colon
	    \mathcal A(K_f; V_{\CC}^{({\rm t})})\to
	    L_0^2(\GL_n(F)Z(\RR)^0 \backslash{}\GL_n(\Adeles_F)/K_f;\omega^\sigma)_{\rm t}^\infty
            \label{eq:iotasigmaisomorphismKf}
	  \end{equation}
	  whose image consists of all $K$-finite automorphic cusp forms of level $K_f$ contained in the right hand side. Moreover, this decomposition is compatible with the two canonical identifications
          \[
	  \mathcal A(K_f; V_{\CC}^{({\rm t})})=
	  \mathcal A(V_{\CC}^{({\rm t})})^{K_f},
          \]
          and
          \begin{flalign*}
	    &L_0^2(\GL_n(F)Z(\RR)^0 \backslash{}\GL_n(\Adeles_F)/K_f;\omega^\sigma)_{\rm t}^\infty=\\
	    &L_0^2(\GL_n(F)Z(\RR)^0 \backslash{}\GL_n(\Adeles_F);\omega^\sigma)_{\rm t}^{\infty,K_f},
          \end{flalign*}
          in the sense that
          \[
          \varinjlim\limits_{K_f}
	  \iota^{\alpha}_{K_f}=
	  \iota^{\alpha}.
          \]
 	\end{lemma}

        \begin{proof}
          On the one hand the proof of Lemma \ref{lem:iotasigmaisomorphism} applies mutatis mutandis to finite level and the claimed compatibility follows by an elementary comparison. On the other hand, we may apply Lemma \ref{lem:iotasigmaisomorphism} and pass to $K_f$-invariants, which is an exact operation commuting with all involved colimits and definitions as above and also proves the statement including the compatibility relations.
        \end{proof}

        \begin{lemma}\label{lem:iotasigmauniquenessKf}
          For each ${\rm t}\in{\rm tcT}_E^\sigma$ and every representative $\alpha\in{\rm t}$ the isomorphism $\iota^{\alpha}_{K_f}$ in \eqref{eq:iotasigmaisomorphismKf} is unique up to scalars in $\sigma(\OO_{E}[1/N])^\times$.
        \end{lemma}

        \begin{proof}
          The argument for inductive limits extends to finite level as well.
        \end{proof}

        We obtain as before
        \begin{corollary}\label{cor:iotaisomorphismKf}
	  Choose for every ${\rm t}\in{\rm tcT}_E$ a representative $\alpha_{\rm t}=[(A^{({\rm t})},V^{({\rm t})})]_\cong\in{\rm t}$. Then for every embedding $\sigma\colon E\to\CC$ the $\sigma$-conjugate of $\alpha_{\rm t}$ is a representative $\alpha_{\rm t}^\sigma\in{\rm t}^\sigma\in{\rm tcT}_E^\sigma$. Put
          \[
          L_0^2(\GL_n(F)Z(\RR)^0 \backslash{}\GL_n(\Adeles_F)/K_f;\omega_W)_{\rm t}^\infty
          :=
          \left(
          \widehat{\bigoplus_{\sigma\colon E\to\CC}}
          \widehat{\bigoplus_{\begin{subarray}c\pi:\\\cohType(\pi)={\rm t}^\sigma\\\pi^{K_f}\neq 0\end{subarray}}}\pi^{K_f}
          \right)^\infty.
          \]
          Then the sum of the $\iota^{\alpha_{\rm t}^\sigma}_{K_f}$ over all embeddings $\sigma\colon E\to\CC$ induces a $(\lieg_\CC,K_\infty)\times\prod_{{\rm t}\in{\rm tcT}_E}\mathcal H(W_\CC^{({\rm t})},K_f)$-module monomorphism
	  \begin{equation}
	    \sum_{\sigma\colon E\to\CC}\iota^{\alpha_{\rm t}^\sigma}_{K_f}\colon
	    \mathcal A(K_f; W_\CC^{({\rm t})})\to
            L_0^2(\GL_n(F)Z(\RR)^0 \backslash{}\GL_n(\Adeles_F)/K_f;\omega_W)_{\rm t}^\infty
	    \label{eq:ithglobaleichlershimuraKf}
	  \end{equation}
          It is unique up to units in $\ZZ[1/N]$ and its image consists of all smooth $K$-finite automorphic cusp forms contained in the right hand side. Moreover we have the compatibility relation
          \[
          \varinjlim\limits_{K_f}
          \iota^{\alpha_{\rm t}^\sigma}_{K_f}=
          \iota^{\alpha_{\rm t}^\sigma}.
          \]

          The sum of the maps \eqref{eq:ithglobaleichlershimuraKf} over ${\rm t}\in{\rm tcT}_E$ induces a monomorphism
          \begin{equation}
	    \iota^{(\alpha_{\rm t})_{\rm t}}_{K_f}:=
	    \sum_{\begin{subarray}c{\rm t}\in{\rm tcT}_E\\\sigma\colon E\to\CC\end{subarray}}
            \iota^{\alpha_{\rm t}^\sigma}_{K_f}\colon
            \sum_{{\rm t}\in{\rm tcT}_E}
            \mathcal A(K_f; W_\CC^{({\rm t})})\to L_0^2(\GL_n(F)Z(\RR)^0 \backslash{}\GL_n(\Adeles_F)/K_f;\omega_W)^\infty,
	    \label{eq:globaleichlershimuraKf}
	  \end{equation}
          of $(\lieg_\CC,K_\infty)\times\mathcal H(W_\CC,K_f)$-modules whose image consist of all smooth $K$-finite automorphic cusp forms contained in the right hand side. Moreover we have the compatibility relation
          \[
          \varinjlim\limits_{K_f}
	  \iota^{(\alpha_{\rm t})_{\rm t}}_{K_f}=
	  \iota^{(\alpha_{\rm t})_{\rm t}}.
          \]
        \end{corollary}

	\subsubsection{The canonical integral structure of finite level}

        \begin{definition}[Integral structures on spaces of automorphic cusp forms]\label{def:globalintegralstructureKf}
        Given an essentially (conjugate) self-dual absolutely irreducible locally algebraic representation $V$, fix a choice $(\alpha_{\rm t})_{\rm t}$ of representatives $\alpha_{\rm t}\in{\rm t}$ for ${\rm t}\in\cohTypes_E^{\rm temp}(G,(V_E,\rho_G))$.

        For each ${\rm t}\in\cohTypes_E^{\rm temp}(G,(V_E,\rho_G))$ and each sufficiently small compact open $K_f\subseteq G(\Adeles_F)$ the $(\lieg_\CC,K_\infty)$-module $\mathcal A(W_\CC^{\rm t},K_f)$ admits by construction a canonical integral structure which is given by the image of $\mathcal A(W_{\ZZ[1/N]}^{({\rm t})},K_f)$ under the canonical map
        \[
	\mathcal A(W_{\ZZ[1/N]}^{({\rm t})},K_f)\to\mathcal A(W_\CC^{({\rm t})},K_f).
        \]
        Define via the canonical isomorphism \eqref{eq:globaleichlershimuraKf} from Corollary \ref{cor:iotaisomorphismKf} a global integral structure on \eqref{eq:l2cuspformsKf} as
	\begin{flalign}
	  &L_0^2(\GL_n(F)Z(\RR)^0 \backslash{}\GL_n(\Adeles_F)/K_f;\omega_W)_{\ZZ[1/N]}:=\label{eq:l2cuspformsqbarKf}\\
                &\bigoplus_{{\rm t}\in\cohTypes_E^{\rm temp}(G,(V_E,\rho_G))}
                \iota^{(\alpha_{\rm t})_{\rm t}}_{K_f}\left(
                \image\left(
                           {\mathcal A}(K_f; W_{\ZZ[1/N]}^{({\rm t})})\to
                           {\mathcal A}(k_f; W_{\CC}^{({\rm t})})
                \right)\right).
                \nonumber
	\end{flalign}
        \end{definition}

        \begin{remark}
          By construction, as in the case of \eqref{eq:l2cuspformsqbar}, \eqref{eq:l2cuspformsqbarKf} is equipped with a canonical $(\lieg_{\ZZ[1/N]},K_{\ZZ[1/N]}\times\pi^G_{0,\ZZ[1/N]})$-action. Its ${\rm t}$-th summand with respect to the decomposition \eqref{eq:l02spacecohomologicaltypedecompositionKf} admits a canonical action of a $\ZZ[1/N]$-integral Hecke algebra
          \[\mathcal H(K_f;W_{\ZZ[1/N]}^{(\rm t)})=\res_{\OO_{E}/\ZZ}\mathcal H(K_f;V_{\OO_{E}[1/N]}^{({\rm t})})\]
          defined in section \ref{sec:integralcohomology}.
        \end{remark}

        \begin{remark}
          As in the colimit case, the integral structure \eqref{eq:l2cuspformsqbarKf} is uniquely determined by the choice of representatives $(\alpha_{\rm t})_{{\rm t}\in\cohTypes_E(G,(V_E,\rho_G))}$ (cf.\ Corollary \ref{cor:iotaisomorphismKf}). Replacing $V$ by $V^\sigma$ for $\sigma\colon E\to\CC$ and $(\alpha_{\rm t})_{\rm t}$ by $(\alpha_{\rm t}^\sigma)_{{\rm t}\in\cohTypes_E(G,(V_E,\rho_G))}$ does not change the resulting integral structure.
        \end{remark}
        
        \begin{theorem}\label{thm:globalhalfintegralstructuresKf}
          Under the same assumptions as in Theorem \ref{thm:globalhalfintegralstructures} and using the same notation, fix a compact open subgroup $K_f\subseteq G(\Adeles_f)$ which is sufficiently small in the sense that all arithmetic groups \eqref{eq:arithmeticgroupforKf} associated to conjugates of $K_f$ are torsion free.
          
          Then as a representation of
          \[
          G(\RR)\times\prod\limits_{{\rm t}\in{\rm tcT}_E}\mathcal H_\CC(K_f;W_{\CC}^{({\rm t})}),
          \]
          the space \eqref{eq:l2cuspformsKf} admits a canonical model \eqref{eq:l2cuspformsqbarKf} over $\ZZ[1/N]$ depending only on the choice of $(\alpha_{\rm t})_{\rm t}$ and the $1/N$-integral model $W_{\ZZ[1/N]}$ of $W_\QQ=\res_{E/\QQ} V_E$ induced by $V_{\OO_E[1/N]}$ with the following properties:
		\begin{itemize}
		        \item[(a)] The $1/N$-integral model \eqref{eq:l2cuspformsqbarKf} is projective as $\ZZ[1/N]$-module and admits a canonical structure of $(\lieg_{\ZZ[1/N]},K_{\ZZ[1/N]},\pi_\circ^G)\times\prod_{{\rm t}\in{\rm tcT}_E}\mathcal H(W_{\ZZ[1/N]}^{({\rm t})};K_f)$-module.
			\item[(b)] The complexification
			\[
			\CC\otimes_{\ZZ[1/N]} L_0^2(\GL_n(F)Z(\RR)^0 \backslash{}\GL_n(\Adeles_F)/K_f;\omega_W)_{\ZZ[1/N]}
			\]
			of \eqref{eq:l2cuspformsqbarKf} is naturally identified with the subspace of smooth $K$-finite vectors in \eqref{eq:l2cuspformsKf} as $(\lieg_{\CC},K(\RR))\times\prod_{{\rm t}\in{\rm tcT}_E}\mathcal H(W_{\CC})$-module.
			\item[(c)] For each cuspidal automorphic $\pi$ occuring in \eqref{eq:l2cuspformsKf} the integral structure $\pi_{\OO_{\QQ(\pi)}[1/N]}^{K_f}$ induced on $\pi^{K_f}$ by the intersection of $\pi^{K_f}$ as a subspace of \eqref{eq:l2cuspformsKf} with the integral structure
			\[
			\OO_{\QQ(\pi)}\otimes_\ZZ
			L_0^2(\GL_n(F)Z(\RR)^0 \backslash{}\GL_n(\Adeles_F)/K_f;\omega_W)_{\ZZ[1/N]}
			\]
			has the following properties. It is projective as $\OO_{\QQ(\pi)}[1/N]$-module and for
                        \[{\rm t}=\cohType(\pi),\]
                        if $\alpha_{\rm t}=[(A^{({\rm t})},V^{({\rm t})})]_\cong$, then the canonical morphism
			\begin{equation}
				H^{t_0}(\lieg,\lies, K;\pi_{\OO_{\QQ(\pi)}[1/N]}^{K_f}\otimes V_{\OO_{\QQ(\pi)}[1/N]}^{({\rm t})})
				\to
				H_{\rm cusp}^{t_0}(X_G(K_f);\widetilde{V}_\CC^{({\rm t})})
				\label{eq:canonicalmapKf}
			\end{equation}
			is injective and factors over $H^{t_0}(X_G(K_f);\widetilde{V}_{\OO_{\QQ(\pi)}[1/N]}^{({\rm t})})$. The monomorphism
			\[
			H^{t_0}(\lieg,\lies, K;\pi_{\OO_{\QQ(\pi)}[1/N]}^{K_f}\otimes V_{\OO_{\QQ(\pi)}[1/N]}^{({\rm t})})\to
			\]
			\[
			\image\left(
			H_{\rm cusp}^{t_0}(X_G(K_f);\widetilde{V}_{\OO_{\QQ(\pi)}[1/N]}^{({\rm t})})\to H_{\rm cusp}^{t_0}(X_G(K_f);\widetilde{V}_\CC^{({\rm t})})\right)[\pi_f]
			\]
			is an isomorphism.
			\item[(d)] For each $\tau\in\Aut(\CC/\QQ)$ and $\pi$ as in (c),
			\[
			(\pi_{\OO_{\QQ(\pi)}[1/N]}^{K_f})^\tau=
			\OO_{\QQ(\pi^\tau)}[1/N]
			\otimes_{\tau^{-1},\OO_{\QQ(\pi)}[1/N]}
			\pi_{\OO_{\QQ(\pi)}[1/N]}^{K_f}
			\]
			is the $\OO_{\QQ(\pi^\tau)}[1/N]$-integral structure on $(\pi^\tau)^{K_f}$ from (c) for the $\tau$-conjugated representative $\alpha_{\rm t}^\tau\in{\rm t}^\tau$.
		      \item[(e)] Taking $(\lieg,\lies,K)$-cohomology with coefficients in
                        \[
                        W_{\ZZ[1/N]}^{({\rm t})}=\res_{\QQ(V^{({\rm t})})/\QQ} V^{({\rm t})}_{\OO_{\QQ(V^{({\rm t})})}[1/N]}
                        \]
                        for $\alpha_{\rm t}=[(A^{({\rm t})},V^{({\rm t})})]_\cong\in{\rm t}\in\cohTypes_E(G,(V_E,\rho_G))$, then the canonical morphism
                        sends the $\ZZ[1/N]$-structure \eqref{eq:l2cuspformsqbarKf} onto the natural $\ZZ[1/N]$-structure
		        \begin{flalign}
			&H_{\rm cusp}^{t_0}(X_G(K_f); \widetilde{W}_{\ZZ[1/N]}^{({\rm t})}):=
                        \nonumber\\
			&{\rm{image}}(
			  H^{t_0}(X_G(K_f); \widetilde{W}_{\ZZ[1/N]}^{({\rm t})})\to
			\label{eq:hcuspKf}\\
			&H_{\rm cusp}^{t_0}(X_G(K_f); \widetilde{W}_\CC^{({\rm t})})).\nonumber
			\end{flalign}
			on cuspidal cohomology in bottom degree $t_0$.
                      \item[(f)] We have canonically
                        \begin{flalign*}
                        &\varinjlim\limits_{K_f}L_0^2(\GL_n(F)Z(\RR)^0 \backslash{}\GL_n(\Adeles_F)/K_f;\omega_W)_{\ZZ[1/N]}=\\
                        &L_0^2(\GL_n(F)Z(\RR)^0 \backslash{}\GL_n(\Adeles_F);\omega_W)_{\ZZ[1/N]}
                        \end{flalign*}
                        and for every cuspidal automorphic $\pi$ of cohomological type
                        \[
                          \cohType(\pi)\in\cohTypes_E(G,(V_E^\sigma,\rho_G))
                        \]
                        for some embedding $\sigma\colon E\to\CC$ canonically
                        \[
                        \varinjlim\limits_{K_f}\pi_{\OO_{\QQ(\pi)}[1/N]}^{K_f}=\pi_{\OO_{\QQ(\pi)}[1/N]}
                        \]
                        $\Aut(\CC/\QQ)$-equivariantly compatibly with the canonical isomorphisms after complexification.
		\end{itemize}
	\end{theorem}

        \begin{proof}
          The proof of (a) --- (e) proceeds as in the case of Theorem \ref{thm:globalhalfintegralstructuresKf}. Note that projectivity of the integral models at finite level over the base ring is a consequence of the projectivity of our model at infinity combined with the torsion-freeness of the model of finite part.

          Statement (f) follows by construction and the natural commutation relations for the involved functors.
        \end{proof}

        \begin{remark}
          For sufficiently small $K_f$ we have a canonical arrow
          \begin{flalign*}
            &L_0^2(\GL_n(F)Z(\RR)^0 \backslash{}\GL_n(\Adeles_F)/K_f;\omega_W)_{\ZZ[1/N]}=\\
            &L_0^2(\GL_n(F)Z(\RR)^0 \backslash{}\GL_n(\Adeles_F);\omega_W)_{\ZZ[1/N]}^{K_f},
          \end{flalign*}
          which is always monomorphic but does not need to be epimorphic due to potential congruences at higher levels.
        \end{remark}

	\subsection{Definition of canonical periods}\label{sec:canonicalperiods}
	
        \subsubsection{Setup}
        
	We keep the notation from the previous section: $(G,K)$ denotes a standard $\ZZ\left[1/2\right]$-form of the symmetric pair associated to $\GL_n$ over a number field $F$ which is assumed totally real or CM for mild technical reasons, cf.\ remark \ref{rmk:fieldrestriction}. Fix a cohomological cuspidal automorphic representation $\pi$ of $\GL_n$ over $F$ of cohomological type ${\rm t}:=\cohType(\pi)$. For simplicity of notation, fix a representative of the form
        \[
        \alpha=[(A(V),V)]\in{\rm t},
        \]
        with $A(V)$ and $V$ defined over its field of rationality $E$ and considered as a locally algebraic $(\lieg,K)$-module and a locally algebraic representation over $E$ respectively. We write $\OO_E\subseteq E$ for the ring of algebraic integers in $E$ and $\OO\subseteq\QQ(\pi)$ for ring of integers in the field of rationality $\QQ(\pi)$ of $\pi$. Then $E\subseteq\QQ(\pi)$ and we fix once again an $\OO_{E}[1/N]$-model $V_{\OO_{E}[1/N]}$ of $V$ on which $G_{\OO_{E}[1/N]}$ acts scheme-theoretically. We do assume that $V_{\OO_{E}[1/N]}$ is normalized in accordance with Lemma \ref{lem:VOnormalization}.
        %and recall the notation
        %\[
        %W_{\ZZ[1/N]}=\res_{\OO_E/\ZZ}V_{\OO_{E}[1/N]}.
        %\]

        We extend this choice to a compatible choice of representatives $[(A(V_{\alpha}),V_\alpha)]_\cong\in\alpha$  for every $\alpha\in\cohType(\pi)$. Then $E=\QQ(V')$ is indepentent of $\alpha$, as is the chosen lattice $V_{\alpha,\OO_{E}[1/N]}\subseteq V_\alpha$.

        \subsubsection{Periods at finite level}

        Fix a sufficiently small compact open subgroup $K_f\subseteq G(\Adeles_f)$ with the property that $\pi^{K_f}\neq 0$ and that all arithmetic subgroups associated to conjugates of $K_f$ (cf.\ \eqref{eq:arithmeticgroupforKf}) are torsion free and that $V_{\OO_{E}[1/N]}$ is $K_f$-stable. According to Theorem \ref{thm:globalhalfintegralstructuresKf} (c) we have for each $\alpha\in\cohType(\pi)$ a canonical $1/N$-integral structure
	\begin{equation}
	\pi^{(K_\infty),K_f}_{\alpha,\OO[1/N]}
	\subseteq
	\pi^{(K_\infty),K_f},
        \label{eq:integralstructureforalpha}
	\end{equation}
        which is compatible with $(\lieg,\lies,K)$-cohomology in bottom degree with coefficients in $V_{\alpha,\OO[1/N]}$ (cf.\ Theorem \ref{thm:globalhalfintegralstructuresKf}, (c)).

	For each $\alpha\in\cohType(\pi)$ we have the canonical identification
	\begin{equation}
		H^q(\lieg_{\CC},S(\RR)^0 K(\RR);\pi\otimes V_{\alpha,\CC})=
		H^q(\GL_n(F)\backslash{}\GL_n(\Adeles_F)/Z(\RR)^0 K(\RR);\widetilde{V}_{\alpha,\CC})[\pi_f]
		\label{eq:complexliealgbracohomology}
	\end{equation}
	of the $\pi_f$-isotypic component of cohomology in degree $q$ with Lie algebra cohomology.

        The right hand side of \eqref{eq:complexliealgbracohomology} carries another canonical $\OO[1/N]$-integral structure induced by the topological $\OO[1/N]$-integral structure on sheaf cohomology in degree $q$.

	On the left hand side of \eqref{eq:complexliealgbracohomology} the integral structure \eqref{eq:integralstructureforalpha} induces via compatibiliy of $(\lieg,\lies,K)$-cohomology in degree $q$ with base change an $\OO[1/N]$-integral structure:	The identification
	\[
	\CC\otimes_{\OO[1/N]}H^q(\lieg,\lies, K;\pi_{\OO[1/N]}\otimes V_{\alpha,\OO[1/N]})\cong H^q(\lieg_{\CC},S(\RR)^0 K(\RR);\pi\otimes V_{\alpha,\CC})
	\]
	induces a second $\OO[1/N]$-integral structure on \eqref{eq:complexliealgbracohomology}, originating from bottom degree (or alternatively top degree).
	
	Fix a $1/N$-integral structure $\pi_{f,\alpha,\OO[1/N]}$ on $\pi_f$ with the property that
	\[
	\pi^{(K_\infty)}_{\alpha,\OO[1/N]}=\pi_{\infty,\alpha,\OO[1/N]}\otimes\pi_{f,\alpha,\OO[1/N]}
	\]
	for a $1/N$-integral structure $\pi_{\infty,\alpha,\OO[1/N]}$ at infinity. Then the factorization
	\begin{flalign*}
	&H^q(\lieg,\lies, K;\pi_{\alpha,\OO[1/N]}\otimes V_{\alpha,\OO[1/N]})=\\
	&H^q(\lieg,\lies, K;\pi_{\infty,\alpha,\OO[1/N]}\otimes V_{\alpha,\OO[1/N]})\otimes\pi_{f,\alpha,\OO[1/N]}
	\end{flalign*}
	shows that the two integral structures on relative Lie algebra cohomology of $\pi$ in degree $q$ induce two integral structures on the relative Lie algebra cohomology
	\begin{equation}
		H^q(\lieg_{\CC},S(\RR)^0 K(\RR);\pi_\infty\otimes V_{\alpha,\CC}).
		\label{eq:relativeliealgebracohomologyatinfinity}
	\end{equation}
	Localizing these $1/N$-integral structures at a finite place $v$ of $\QQ(\pi)$ not dividing $N$, we obtain two $\OO_{(v)}$-integral structures on the left hand side of \eqref{eq:relativeliealgebracohomologyatinfinity}. Since $\OO_{(v)}$ is a discrete valuation ring, these integral structures are free of finite rank $r$, which is the dimension of \eqref{eq:relativeliealgebracohomologyatinfinity}. Choosing bases therefore produces a period matrix $\Omega_{\alpha,q}\in\CC^{r\times r}$. The double coset $\GL_r(\OO_{(v)})\Omega_{\alpha,q}\GL_r(\OO_{(v)})$ then only depends on $\pi$ and the choice of $v$-integral model $V_{\OO_{(v)}}$ of $V$. For any $t_0\leq q\leq q_0$ between bottom and top degree $t_0$ and $q_0$ respectively, we obtain a non-trivial period matrix. Our normalization in bottom degree amounts to $\Omega_{\alpha,b}\in\OO_{(v)}^\times$ for all $v\nmid N$. One may equally normalize in top degree, which would amount to decreeing that $\Omega_{\alpha,b}\in\OO_{(v)}^\times$ unter the same hypothesis.
	
	\begin{remark}
		A posteriori the normalization via Lemma \ref{lem:VOnormalization} becomes superfluous after localization at $v$, because $\OO_{(v)}$ is a principal ideal domain.
	\end{remark}
	
	We proved the following integral analogue of Theorem B in \cite{januszewskirationality}, taking all possible parity conditions at infinity into account.
	\begin{theorem}[Canonical periods of finite level]\label{thm:periods}
		For each cohomological cuspidal automorphic representation $\pi$ of $\GL_n$ over a number field $F$ which is totally real or CM, for each sufficiently small compact open subgroup $K_f\subseteq\GL_n(\Adeles_{F,f})$ as above, for each $\alpha\in\cohType(\pi)$ and each cohomological degree $t_0\leq q\leq q_0$ in the cuspidal range, each finite place $v$ of $\QQ(\pi)$ not dividing $N$ with valuation ring $\OO_{(v)}\subseteq\QQ(\pi)$ and each $v$-integral $K_f$-stable model $(V_{\OO_{(v)}},\iota)$ of the locally algebraic representation $V_\alpha$ attached to $\alpha$, there is a period matrix $\Omega_{q}(\pi^{K_f},\alpha,\iota,v)\in\GL_{r}(\CC)$ with the following properties:
		\begin{itemize}
			\item[(i)] $\Omega_q(\pi^{K_f},\alpha,\iota,v)$ is the transformation matrix transforming $H^q(\lieg,\lies, K; \iota)$ into the natural $\OO_{(v)}$-structure on cuspidal cohomology.
			\item[(ii)] The double coset $\GL_{r}(\OO_{(v)})\Omega_q(\pi^{K_f},\alpha,\iota,v)\GL_{r}(\OO_{(v)})$ depends only on the quadruple $(\pi,K_f,\alpha,\iota)$ and the degree $q$.
			\item[(iii)] For each $c\in\CC^\times$ we have the relation
			\[
			\Omega_q(\pi^{K_f},\alpha,c\cdot\iota,v)\;=\;c\cdot\Omega_q(\pi^{K_f},\alpha,\iota,v).
			\]
			\item[(iv)] The ratio
			\[
			\frac{1}
			{\Omega_{t_0}(\pi^{K_f},\alpha,\iota,v)}\cdot
			\Omega_q(\pi^{K_f},\alpha,\iota,v)\;\in\;\GL_{r}(\CC)
			\]
			is independent of $\iota$.
		\end{itemize}
	\end{theorem}
	
	\begin{remark}\label{rmk:venkatesh}
	  At this stage it is unclear how the period matrices $\Omega_q(\pi^{K_f},\alpha,\iota,v)$ are related to Venkatesh's conjectural action of motivic cohomology $\bigwedge^\bullet H_{\rm{mot}}^1(M_{\pi,\rm{coad}};\QQ(1))$ of the (conjectural) coadjoint motive attached to $\pi$ and the derived Hecke algebra on the $\pi$-isotypic component of cuspidal cohomology as studied in \cite{venkatesh2017,venkatesh2018ICM,galatiusvenkatesh2018,venkatesh2019,prasannavenkatesh2021}. That being said, there is room for speculation.

          If we interpret $H_{\rm{mot}}^1(M_{\pi,\rm{coad}},\QQ(1))$ as the Betti realization of a (conjectural) motive $\mathcal M=H^1(M_{\pi,\rm{coad}};\QQ(1))$ we may consider the Betti and the de Rham realizations of $\mathcal M$. By Venkatesh's conjecture, the \lq{}topological\rq{} rational structure on the isotypic component $H^\bullet(\mathscr X_G(K_f);\widetilde{V})[\pi_f]$ should be (canonically) a free module over $\bigwedge^\bullet \mathcal M_{\mathrm{Betti}}$. One may hope that the second rational stucture on the same isotypic component we construct in Theorem \ref{thm:globalhalfintegralstructures} is related to the exterior algebra $\bigwedge^\bullet \mathcal M_{\mathrm{deRham}}$ over the de Rham realization of $\mathcal M$. In this vein, the period matrices $\Omega_q(\pi^{K_f},\alpha,\iota,v)$ considered over $\QQ(\pi)$ would --- hypothetically --- measure the difference between the two exterior algebras $\bigwedge^\bullet \mathcal M_{\mathrm{deRham}}$ and  $\bigwedge^\bullet \mathcal M_{\mathrm{Betti}}$. However, at this point, we do not have any evidence for such a statement.
	\end{remark}

        \begin{remark}
          Integrally, Theorem \ref{thm:periods} has no immediate analogue at infinite level due to potential congruences at deeper levels than $K_f$. The latter may in principle lead to invertibility of elements in $\OO[1/N]$ with positive $v$-valuation in the integral model.
        \end{remark}

        \subsection{Special values of automorphic $L$-functions}\label{sec:specialvalues}

        In this section we discuss the relation of rational and $1/N$-integral structures to special values of $L$-functions.

        \subsubsection{Rationality properties of restrictions}

        We let $G$ denote a connected reductive group over $\QQ$ and $K\subseteq G$ is assumed to be a $\QQ$-subgroup with $K(\RR)\subseteq G(\RR)$ maximal compact.

        We consider a representative $[(A,V)]_\cong\in{\rm t}$ for ${\rm t}\in\cohTypes_{\QQ(A)}(G)$, assumed to be defined over the field of definition $\QQ(A)/\QQ(V)$ of $A$. We let explicitly $V=(V_{\QQ(A)},\rho_G,\rho_{\pi_\circ^G})$ and likewise $A=(A_{\QQ(A)},\rho_{\pi_\circ^G})$.

        Fix an admissible weight $V'=(V_{\QQ(A)},\rho_G,\rho_{\pi_\circ^G}')$ of $A$ in the sense of Definition \ref{def:admissibleweights}, i.\,e.\ $\rho_{\pi_\circ^G}'=\rho_{\pi_\circ^G}\otimes\chi^{\rm const}$ with a $\QQ$-rational character $\chi\colon\pi_\circ(K)\to\Gm$ of the not necessarily constant \'etale group scheme $\pi_\circ(K)$. Then we have a corresponding admissible twist $A'=(A_{\QQ(A)},\rho_{\pi_\circ^G}')$ of $A$. In order to keep notational load at a minimum, we realize $V'$ and $A'$ on the same vector spaces as $A$ and $V$.

        We assume that $A_\CC\cong A_\CC'$, i.\,e.\ $(A,V)$ and $(A',V')$ are representatives of the same cohomological type ${\rm t}$. Since $A'$ is an admissible twist of $A$ we have an isomorphism
        \begin{equation}
          \psi_{\QQ(A)}\colon A_{\QQ(A)}\otimes\chi\to A_{\QQ(A)}
          \label{eq:psiQQV}
        \end{equation}
        of $(\lieg_{\QQ(A)},K_{\QQ(A)})$-modules.
        
        Denote by $\widetilde{K}:=\ker\chi$ the kernel of $\chi$ on $K$ and assume that $\chi$ is non-trivial, i.\,e.\
        \[
        K/\widetilde{K}\cong\mu_2
        \]
        as constant group schemes.

        By composition with the canonical map $a\mapsto a\otimes 1$ we consider the isomorphism $\psi_{\QQ(A)}$ in \eqref{eq:psiQQV} as a $\chi$-twisted $(\lieg_{\QQ(A)},K_{\QQ(A)})$-linear isomorphism
        \begin{equation}
          \widetilde{\psi}_{\QQ(A)}\colon A_{\QQ(A)}\to A_{\QQ(A)}.
          \label{eq:psiQQVtilde}
        \end{equation}

        From now on we assume $A$ to be absolutely irreducible, i.\,e.\ $A_\CC$ is irreducible as $(\lieg_\CC,K(\RR))$-module.
          
        \begin{proposition}\label{prop:psidecompositions}
          The minimal polynomial of $\widetilde{\psi}_{\QQ(A)}$ exists and is of the form
          \begin{equation}
            X^2-\alpha\in\QQ(A)[X],\quad\alpha\in\QQ(A)^\times.
            \label{eq:psimp}
          \end{equation}
          Denote $E:=\QQ(A)[\sqrt{\alpha}]$ the field obtained by adjoining a square root of $\alpha$. Then
          \begin{itemize}
          \item[(a)] The class of $\alpha$ modulo squares in $A_{\QQ(A)}^\times$ does not depend on the choice of $\widetilde{\psi}_{\QQ(A)}$ in \eqref{eq:psiQQVtilde} or $\psi_{\QQ(A)}$ in \eqref{eq:psiQQV}. In particular, the extension $E/\QQ(A)$ is independent of the choice of $\widetilde{\psi}_{\QQ(A)}$.
          \item[(b)] The eigenspaces
            \begin{equation}
              \EE_{\pm\sqrt{\alpha}}\left(\widetilde{\psi}_{E}\right)\subseteq A_{E}
              \label{eq:Epmalpha}
            \end{equation}
            are absolutely irreducible $(\lieg_{E},\widetilde{K}_{E})$-submodules and are independent of the choice of $\widetilde{\psi}_{\QQ(A)}$ in \eqref{eq:psiQQVtilde} or $\psi_{\QQ(A)}$ in \eqref{eq:psiQQV}.
          \item[(c)] We have the decomposition
            \begin{equation}
              A_{E}=\EE_{\sqrt{\alpha}}\left(\widetilde{\psi}_{E}\right)\oplus\EE_{-\sqrt{\alpha}}\left(\widetilde{\psi}_{E}\right).
              \label{eq:AEdecomposition}
            \end{equation}
          \item[(d)] Let $E'/\QQ(A)$ be an arbitrary field extension. Then $(\lieg_{E},\widetilde{K}_{E})$-modules \eqref{eq:Epmalpha} are defined over $E'$ if $\alpha$ is a square in $E'$. If the $(\lieg_{E},\widetilde{K}_{E})$-modules \eqref{eq:Epmalpha} lie in different isomorphism classes as $(\lieg_{E},\widetilde{K}_{E})$-modules, then the converse is true as well.
          \item[(e)] If the $(\lieg_{E},\widetilde{K}_{E})$-modules \eqref{eq:Epmalpha} lie in different isomorphism classes as $(\lieg_{E},\widetilde{K}_{E})$-bmodules, and if $\alpha$ is no square in $E'/\QQ(A)$, then $A_{E'}$ is an irreducible $(\lieg_{E'},\widetilde{K}_{E'})$-module.
          \end{itemize}
        \end{proposition}

        \begin{remark}\label{rmk:restrictiontobottomlayer}
          In applications, $\alpha$ in Proposition \ref{prop:psidecompositions} may be computed by restricting attention to the bottom layer (i.\,e.\ the minimal $K_\infty$-type) in $A_\CC$. If $B_{\QQ(A)}\subseteq A_{\QQ(A)}$ denotes the bottom layers, we have a commutative diagram,
          \begin{equation}
            \begin{CD}
              A_{\QQ(A)}@>{\widetilde{\psi}_{\QQ(A)}}>> A_{\QQ(A)}\\
              @AAA @AAA\\
              B_{\QQ(A)}@>>{\widetilde{\psi}_{\QQ(A)}^{\rm K}}> B_{\QQ(A)}
            \end{CD}
            \label{eq:bottomlayerdiagram}
            \end{equation}
            where $\widetilde{\psi}_{\QQ(A)}^{\rm K}$ is a non-zero $\chi$-twisted isomorphism of locally algebraic $(\liek,K)$-modules, which in this case, are locally algebraic representations of $K$. This reduces the determination of $\alpha$ to combinatorial considerations in the context of the structure theory of linear reductive groups.
        \end{remark}

        \begin{proof}
          Consider the self-composition
          \[
          %\widetilde{\psi}_{\QQ(A)}^2=
          \widetilde{\psi}_{\QQ(A)}\circ\widetilde{\psi}_{\QQ(A)}\colon \quad A_{\QQ(A)}\to A_{\QQ(A)},
          \]
          which as a composition of two $\chi$-twisted linear isomorphisms is an untwisted $(\lieg_{\QQ(A)},K_{\QQ(A)})$-linear isomorphism. By Schur's Lemma and the absolute irreducibility of $A_{\QQ(A)}$ we therefore have
          \[
          \widetilde{\psi}_{\QQ(A)}^2=\alpha\cdot{\bf1}_{A_{\QQ(A)}}
          \]
          for a unique $\alpha\in\QQ(A)^\times$. Again by Schur's Lemma and the absolute irreducibility of $A_{\QQ(A)}$ we see that $\widetilde{\psi}_{\QQ(A)}$ is unique up to a non-zero scalar in $\QQ(A)^\times$, hence $\alpha$ is unique up to multiplication by non-zero squares.

          We conclude that the minimal polynomial of $\widetilde{\psi}_{\QQ(A)}$ is a divisor of the quadratic polynomial $X^2-\alpha$. If its minimal polynomial would be linear, $\widetilde{\psi}_{\QQ(A)}$ were a multiple of the identity, hence untwisted linear, which is a contradiction. Therefore the minimal polynomial is indeed \eqref{eq:psimp} as claimed. This proves the first claim and statement (a).

          The eigenspaces in (b) are non-zero by linear algebra and are $(\lieg_{E},\widetilde{K}_{E})$-submodules by the $(\lieg_{E},\widetilde{K}_{E})$-linearity of $\widetilde{\psi}_{\QQ(A)}$. Now since $\widetilde{K}_E$ is of index $2$ in $K$, the restriction of $A_E$ to $(\lieg_{E},\widetilde{K}_{E})$ is at most of length $2$. Therefore both eigen sppaces constitute be the two composition factors of a composition series and both must be absolutely irreducible. This proves (b) and (c).

          If $E'/\QQ(A)$ is an extension and if $\alpha$ is a square in $E'$, then the eigenspaces are $E'$-rational subspaces, which in particular implies that \eqref{eq:Epmalpha} are defined over $E'$. Assume conversely that the eigenspaces lie in different isomorphism classes and are defined over $E'$ in the sense that for some sufficiently large extension $E''/E$ we have an embedding $E'\to E''$ such that we find. Then we find abstract irreducible $E'$-rational $(\lieg_{E'},\widetilde{K}_{E'})$-modules $A_{\pm,E'}$ which over $E''$ become abstractly isomorphic to the eigenspaces.

          Since the two eigenspaces lie in different isomorphism classes, the decomposition \eqref{eq:AEdecomposition} is unique, which implies that it descends from $E''$ to $E'$ since the individual summands descend. Now the $\chi$-twisted isomorphism $\widetilde{\psi}$ descends to $E'$ as well and must preserve this decomposition, which shows that its eigen values lie in $E'$, i.\,e.\ $\alpha$ is a square in $E'$. This concludes the proof of (d).

          The statement in (e) is again an immediate consequence of the uniqueness of the decomposition \eqref{eq:AEdecomposition} in this case: If $A_{E'}$ is reducible as $(\lieg_{E'},\widetilde{K}_{E'})$-module, then its isotypic decomposition agrees over a common extension of $E'$ and $E$ with the eigenspace decomposition. Hence the eigen values are $E'$-rational once again (cf.\ proof of (d)).
        \end{proof}

        \begin{corollary}\label{cor:structureofK}
          If $A_\CC\cong A_\CC'$ for two $(A,V)$ and $(A',V')$ non-isomorphic representatives of the same cohomological type ${\rm t}\in\cohTypes_E(G)$ and if $A'$ is an admissible twist of $A$, then $\RR$, $K$ is a non-trivial semi-direct product
          \[
          K(\RR)=\widetilde{K}(\RR)\rtimes \mu_2(\RR).
          \]
        \end{corollary}

        \begin{proof}
          Since a section $\mu_2(\RR)\to K(\RR)$ exists $K(\RR)$ is a semi-direct product. To show its non-triviality, we observe that by statement (b),
          \[
          A_E=\ind_{\lieg,\widetilde{K}}^{\lieg,K}\EE_{\pm\sqrt{\alpha}}\left(\widetilde{\psi}_{E}\right).
          \]
          is (absolutely) irreducible. The same remains true over arbitrary field extensions. Therefore, $K_{E'}$ is not a direct product of $\widetilde{K}_{E'}$ and $\mu_2$ over any field extension $E'/E$.
        \end{proof}

        \begin{theorem}\label{thm:AEdecomposition}
          Let $[(A,V)]_\cong\in{\rm t}\in\cohTypes_{\QQ(A)}(G)$ be a representative of a cohomological type, assumed to be defined over the field of definition $\QQ(A)$ of $V$. Assume that $A_\CC$ decomposes as $(\lieg_\CC,K_\infty^0)$-module multiplicity-free. Then for each admissible weight $V'\not\cong V$ of $A$ consider
          \[
          \widetilde{K}=\ker\chi\subseteq K
          \]
          where $\chi$ is the character of $K$ twisting $V$ into $V'$. Then for each field extension $E'/\QQ(A)$ the following are equivalent:
          \begin{itemize}
            \item[(a)] As a $(\lieg_{E'},\widetilde{K}_{E'})$-module $A_{E'}$ is reducible.
            \item[(b)] As a $(\lieg_{E'},\widetilde{K}_{E'})$-module $A_{E'}$ decomposes into two non-isomorphic absolutely irreducible modules.
            \item[(c)] For any $\chi$-twisted linear isomorphism $\widetilde{\psi}_{E'}$ as in \eqref{eq:psiQQVtilde} we have
              \[
              \widetilde{\psi}_{E'}^2=\widetilde{\alpha}^2\cdot{\bf1}_{A_{E'}},
              \]
              for some $\widetilde{\alpha}\in E'$.
            \item[(d)] Any $\chi$-twisted linear isomorphism $\widetilde{\psi}_{E'}$ as in \eqref{eq:psiQQVtilde} acts via scalars $\pm\widetilde{\alpha}\in E'$ on any irreducible $(\lieg_{E'},\widetilde{K}_{E'})$-submodule of $A_{E'}$.
          \end{itemize}
          Moreover, the above statements are equivalent to the following statements about the bottom layer $B_{\QQ(A)}\subseteq A_{\QQ(A)}$:
          \begin{itemize}
            \item[(a')] As a $(\liek_{E'},\widetilde{K}_{E'})$-module $B_{E'}$ is reducible.
            \item[(b')] As a $(\liek_{E'},\widetilde{K}_{E'})$-module $B_{E'}$ decomposes into two non-isomorphic absolutely irreducible modules.
            \item[(c')] For any $\chi$-twisted linear isomorphism $\widetilde{\psi}_{E'}^K\colon B_{E'}\to B_{E'}$ we have
              \[
              (\widetilde{\psi}_{E'}^K)^2=\widetilde{\alpha}^2\cdot{\bf1}_{A_{E'}},
              \]
              for some $\widetilde{\alpha}\in E'$.
            \item[(d')] Any $\chi$-twisted linear isomorphism $\widetilde{\psi}_{E'}^K\colon B_{E'}\to B_{E'}$ acts via scalars $\pm\widetilde{\alpha}\in E'$ on any irreducible $(\liek_{E'},\widetilde{K}_{E'})$-submodule of $B_{E'}$.
          \end{itemize}
        \end{theorem}

        \begin{proof}
          The statement of the Theorem is a consequence of Proposition \ref{prop:psidecompositions}. Remark \ref{rmk:restrictiontobottomlayer} justifies the equivalence with the same statements about the bottom layer.
        \end{proof}

        \begin{remark}
          The above reduction of the descent problem from $G$ to $K$ via bottom layers is a variant of the descent criterion formulated as Proposition 5.8 in \cite{januszewskirationality}.
        \end{remark}

        \begin{example}
          Let $F/\QQ$ be a totally real extension of number fields and put $G=\res_{F/\QQ}G$ and $K=\res_{F/\QQ}\Oo(n)$.
          
          For odd $n$ we have
          \[
          K(\RR)=K(\RR)^0\times\mu_2(\RR)^{[F:\QQ]},
          \]
          which in light of Corollary \ref{cor:structureofK} is a manifestation of the fact that hypotheses of Proposition \ref{prop:psidecompositions} and Theorem \ref{thm:AEdecomposition} are never satisfied, which in turn is a consequence of the classification of tempered cohomological types \ref{thm:cohomologicaltypesofGLnoverQ} for $G$. In this case $A_{\QQ(A)}$ remains irreducible as $(\lieg,\widetilde{K})$-module for any $K^\circ\subseteq \widetilde{K}\subseteq K$, hence this restriction is always defined over $\QQ(A)$.
          
          For even $n$ the situation is more interesting. In this case $K(\RR)$ is a non-split semi-direct product.

          Write $B_{\QQ(A)}\subseteq A_{\QQ(A)}$ for the bottom layer. Consider $E=\QQ(A)[\sqrt{-1}]$. Then $K^\circ$ is split over $E$. This implies that $B$ decomposes over $E$ into absolutely irreducible $\widetilde{K}$-modules. By Theorem \ref{thm:AEdecomposition}, the case $E=\QQ(A)$ is therefore trivial. Assume that $E\neq\QQ(A)$, i.\,e.\ $E/\QQ(A)$ is a quadratic extension. Let $B_E^\circ=A_{\lieq\cap\liek}(\lambda|_{\lieq\cap\liek,L\cap K^\circ})_E$ denote the bottom layer of the cohomologically induced standard module $A_{\lieq}^\circ(\lambda^\circ)_E$ for the pair $(\lieg,K^\circ)$. Since $n$ is even, we know that
          \[
          A_E=A_\lieq(\lambda)_E=\ind_{\lieg,K^\circ}^{\lieg,K}A_{\lieq}^\circ(\lambda^\circ)_E,
          \]
          and likewise for the bottom layer:
          \[
          B_E=A_{\lieq\cap\liek}(\lambda|_{\lieq\cap\liek})_E=\ind_{K^\circ}^{K}B_E^\circ.
          \]
          Now $K^\circ\subseteq\widetilde{K}\subsetneq K$, which by transitivity of induction implies
          \[
          B_E=\ind_{\widetilde{K}}^{K}\widetilde{B}_E,\quad\text{for}\;\widetilde{B}_E=\ind_{K^\circ}^{\widetilde{K}}B_E^\circ.
          \]
          The $\widetilde{K}$-module $\widetilde{B}_E$ is absolutely irreducible and a summand in $B_E$. Hence by Theorem \ref{thm:AEdecomposition} (a'), we need to decide when $\widetilde{B}_E$ descends to $\QQ(A)$.

          Assume first that $\lambda=0$, i.\,e.\ $V=({\bf1},\rho_{\pi_\circ^G})$ is a self-dual character. In this case $\QQ(A)=\QQ$ and $B_E$ is self-dual and infinitesimally unitary. Therefore, complex conjugation as an automorphism of the extension $E/\QQ(A)=\QQ[\sqrt{-1}]/\QQ$ sends $\widetilde{B}_E$ to its dual. Hence, by Galois descent, $\widetilde{B}_E$ descends to $\QQ$ if and only if it is self-dual.

          Write
          \begin{align*}
            K_\RR&=\prod_{v\mid\infty}\Oo(n)_\RR,\quad B_\CC=\bigotimes_{v\mid\infty}B_{v,\CC},\\
            K_\RR^\circ&=\prod_{v\mid\infty}\SO(n)_\RR,\quad B_\CC^\circ=\bigotimes_{v\mid\infty}B_{v,\CC}^\circ,
          \end{align*}
          where $B_{v,\CC}$ (resp.\ $B_{v,\CC}^\circ$) is the bottom layer for the classical orthogonal group $\Oo(n)_\RR$ (resp.\ $\SO(n)_\RR$).

          Choose for every archimedean place $v\mid\infty$ of $F$ a representative $\varepsilon_v\in\Oo(n)(\RR)\setminus\SO(n)(\RR)$. Example \ref{ex:admissiblecharacters} shows that the admissible character $\chi$ decomposes as a character of $\pi_0(K(\RR))$ into
          \[
          \chi=\bigotimes_{v\mid\infty}\chi_v,\quad \forall v\mid\infty\colon\;\chi_v\neq{\bf1}.
          \]
          Therefore its kernel is given by the union of the cosets $(\varepsilon_v^{\delta_v})_v\cdot K(\RR)^0$ with exponent vectors ranging over all $(\delta_v)_v\in\{0,1\}^{\Hom(F,\CC)}$ with
          \[
          \sum_{v\mid\infty}\delta_v\in2\ZZ.
          \]
          This implies that
          \[
          \widetilde{B}_\CC=\sum_{(\delta_v)_v}\bigotimes_{v\mid\infty}\varepsilon_v^{\delta_v}B_{v,\CC}^\circ,
          \]
          where the sum ranges over the same \lq{}even\rq{} exponent vectors. Hence the dual of $\widetilde{B}_\CC$ is give by
          \[
          \widetilde{B}_\CC^\vee=\sum_{(\delta_v)_v}\bigotimes_{v\mid\infty}\varepsilon_v^{\delta_v}B_{v,\CC}^{\circ,\vee},
          \]
          By Proposition 2.1 in \cite{januszewskiperiods1}, $B_{v,\CC}^\circ$ is self-dual if and only if $n$ is even. In thise, $\widetilde{B}_\CC$ is self-dual as well. If $n$ is odd,
          \[
          B_{v,\CC}^{\circ,\vee}\cong \varepsilon_v^{\delta_v}B_{v,\CC}^\circ,
          \]
          which shows that
          \[
          \widetilde{B}_\CC^\vee=\sum_{(\delta_v)_v}\bigotimes_{v\mid\infty}\varepsilon_v^{1-\delta_v}B_{v,\CC}^\circ,
          \]
          as submodules of $B_\CC$. Now the \lq{}diagonal\rq{} element $\varepsilon:=(\varepsilon_v)_v$ has an odd number of entries, which implies that
          \[
          \widetilde{B}_\CC+\widetilde{B}_\CC^\vee=B_\CC,
          \]
          hence $\widetilde{B}_\CC$ is not self-dual.
          
          Summing up, this shows that $\widetilde{B}_E$ descends to $\QQ$ if and only if $2\mid n$.

          The general case reduces to $\lambda=0$ via translation functors along the same lines as in the proof of Theorem 2.4 in \cite{januszewskiperiods1}.
          
          We remark that equation (16) in \cite{januszewskiperiods1} is only valied \'etale locally if $F\neq\QQ$. Nonetheless the argument given in the proof of Lemma 2.3 and Theorem 2.4 in \cite{januszewskiperiods1} is correct and shows that $B^\circ_E$ (resp.\ $A_\lieq^\circ(0)$) is defined over $\QQ$ if and only if $4\mid n$.
        \end{example}

        The argument given in the example generalizes mutatis mutandis to arbitrary products $G=\res_{F/\QQ}\left(\GL(n_1)\times\cdots\times\GL(n_r)\right)$ along the same lines as Theorem 2.4 in \cite{januszewskiperiods1}. This shows

        \begin{theorem}\label{thm:glnrationalreducibility}
          Let $F/\QQ$ denote a finite extension which is totally real or a CM field. Put
          \[
          G=\res_{F/\QQ}\left(\GL(n_1)\times\cdots\times\GL(n_r)\right)
          \]
          with $n_1,\dots,n_r\geq 1$. Let $[(A,V)]_\cong\in{\rm t}\in\cohTypes_{\QQ(A)}(G)$ be a representative of a cohomological type. Then $A$ and $V$ are defined over the field of definition $\QQ(A)$ of $V$. Then an admissible weight $V'\not\cong V$ of $A$ exists if and only if $2\mid n_i$ for some $1\leq i\leq r$ and $F$ is totally real. Independently of the parity of $n$, put
          \[
          \widetilde{K}=\ker\chi\subseteq K
          \]
          for a non-trivial admissible character $\chi$ of $G$. Write $I\subseteq\{1,\dots,r\}$ for the subset of indices for which $n_i$ is even and the restriction $\chi|_{\res_{F/\QQ}\GL(n_i)}$ to the $i$-th factor of $G$ is non-trivial.

          Then the following are equivalent:
          \begin{itemize}
          \item[(a)] As $(\lieg,\widetilde{K})$-module, $A_{\QQ(A)}$ is reducible.
          \item[(b)] $I\neq\emptyset$ and: $\sqrt{-1}\in\QQ(A)$ or for all $i\in I$: $4\mid n_i$.
          \end{itemize}
        \end{theorem}

        \begin{remark}
          If $F/\QQ$ is a general number field, then $K$ admits a rational model over a number field $\QQ_K$. Theorem \ref{thm:glnrationalreducibility} may be generalized to $G=\res_{F/\QQ}(\GL(n_1)\times\cdots\times\GL(n_r))$ by replacing the ground field $\QQ$ by $\QQ_K$.
        \end{remark}

        \subsubsection{General archimedean peroid relations}

        In this subsection we put ourselves in the following general situation. We let $G$ denote a connected reductive group over $\QQ$ and $H\subseteq G$ a connected reductive $\QQ$-subgroup. As before, $K\subseteq G$ is assumed to be a $\QQ$-subgroup with $K(\RR)\subseteq G(\RR)$ maximal compact. We assume that $L:=K\cap H$ has the following property: $L(\RR)\subseteq H(\RR)$ is of finite index in a maximal compact subgroup in $H(\RR)$.

        Recall that $\widetilde{K}=\ker\chi$ is the kernel of $\chi$ on $K$ and likewise put $\widetilde{L}:=\ker\chi|_{L}$. Assuming the non-triviality of $\chi|_{\pi_0(L)}$, $\widetilde{K}$ and $\widetilde{L}$ are of index $2$ in $K$ and $L$ and
        \[
        K/\widetilde{K}\cong L/\widetilde{L}\cong\mu_2
        \]
        are constant group schemes.

        Write
        \[
        \Delta_\RR\colon K_\infty\to K(\RR)\times\pi_\circ^G(\RR)=K(\RR)\times\pi_0(K(\RR))
        \]
        for the \lq{}diagonal\rq{} morphism. We write $K_\infty:=\Delta(K(\RR))$ and abuse the same notation for $H$.

        \begin{theorem}\label{thm:rationalperiodrelations}
          Let $[(A,V)]_\cong\in{\rm t}\in\cohTypes_{\QQ(A)}(G)$ be a representative of a cohomological type, assumed to be defined over the field of definition $\QQ(A)$ of $A$. Assume that $A_\CC$ decomposes as $(\lieg_\CC,K_\infty^0)$-module multiplicity-free. Fix an arbitrary admissible weight $V'\not\cong V$ of $A$.

          Given two non-zero $\QQ(A)$-rational $(\lieh,L\times\pi_0^{\rm const}(L))$-linear functionals
          \[
          \lambda \colon A \to E ,\quad
          \lambda'\colon A'\to E',
          \]
          where $E=(\QQ(A),\rho_H,\rho_L)$ and $E'=(\QQ(A),\rho_H,\rho_L')$ with
          \[
          \rho_{\pi_\circ^G}'=\rho_{\pi_\circ^G}\otimes\chi^{\rm const}
          \quad\text{and}\quad
          \rho_L'=\rho_L\otimes\chi^{\rm const}|_{\pi_0^{\rm const}(L)}.
          \]
          Assume that $\chi|_L$ is non-trivial.
          
          Let $\pi_\infty$ denote a $(\lieg_\CC,K_\infty)$-module isomorphic to $A_\CC$ and $A_\CC'$. Then the rational models $A$ and $A'$ of $A_\CC$ and $A_\CC'$ together with fixed isomorphisms $A_\CC\cong\pi_\infty$ and $A_\CC'\cong\pi_\infty$ of $(\lieg_\CC,K_\infty)$-modules induce fixed $E$-linear embeddings
          \[
          \iota \colon A \to\pi_\infty,\quad
          \iota'\colon A'\to\pi_\infty.
          \]
          Both maps are understood as $(\lieg_{\QQ(A)},K_{\QQ(A)}\times\pi_\circ^G)$-linear maps on the underlying vector spaces, which after base change to $\CC$ become $(\lieg_\CC,K_\infty)$-linear. Note that $\pi_\infty$ carries two different structures of locally algebraic $(\lieg,K)$-module.
          
          Assume that we are given two complex functionals
          \[
          \Lambda \colon\pi_\infty\to\CC,\quad
          \Lambda'\colon\pi_\infty\to\CC,
          \]
          with the property that for some irreducible $(\lieg_\CC,K_\infty^0)$-submodule $\pi_0\subseteq\pi_\infty$
          \begin{equation}
            \Lambda|_{\pi_0}=\Lambda'|_{\pi_0}.
            \label{eq:Lambdacompatibility}
          \end{equation}
          and that we have commutative diagrams
          \begin{equation}
          \begin{CD}
            A @>{\lambda}>> E @.\hspace*{4em} @. A' @>{\lambda'}>> E'\\
            @V{\iota}VV @VV{\nu}V @. @V{\iota'}VV @VV{\nu'}V \\
            \pi_\infty @>>{\Lambda}> \CC @. @. \pi_\infty @>>{\Lambda'}> \CC
          \end{CD}
          \label{eq:Lambdacommutativity}
          \end{equation}
          with the property that the images of the right vertical maps agree:
          \[
          \nu(E)=\nu'(E').
          \]
          Remark that the right vertical arrows are both $(\lieh_{\QQ(A)},L_{\QQ(A)}\times\pi_\circ^H)$-linear if $\CC$ is given the unique structure of locally algebraic representation with the property that its complex realization as $(\lieh_\CC,H_\infty)$-module make $\Lambda$ and $\Lambda'$ both $(\lieh_\CC,L_\infty)$-linear. Then with $\alpha\in\QQ(A)^\times$ from \eqref{eq:psimp} we have the period relation
          \begin{equation}
            \sqrt{\alpha}\cdot\iota(A)=\iota'(A').
          \end{equation}
        \end{theorem}

        \begin{proof}
          We claim the existence of a constant $c\in\CC^\times$, unique modulo $\QQ(A)^\times$, with the property that
          \[
          c\cdot\iota(A)=\iota'(A').
          \]
          In order to prove its existence and compute $c$, we consider the common underlying $\QQ(A)$-vector space $A_{\QQ(A)}$ of $A$ and $A'$ with its canonical $(\lieg_{\QQ(A)},K_{\QQ(A)})$-module structure. Denote the $\iota$ underlying map by
          \[
          \iota_{\QQ(A)}\colon A_{\QQ(A)}\to\pi_\infty
          \]
          and likewise write
          \[
          \iota_{\QQ(A)}'\colon A_{\QQ(A)}\to\pi_\infty
          \]
          for the map underlying $\iota'$. Then both $\iota_{\QQ(A)}$ and $\iota_{\QQ(A)}'$ are $(\lieg_{\QQ(A)},K_{\QQ(A)}\times\pi_\circ^G)$-linear for two locally algebraic structures on $\pi_\infty$, differing by an admissible twist by a character $\chi$ of $\pi_\circ(K)$. Extending scalars to $\CC$, we obtain a $\chi$-twisted $(\lieg,K)$-module isomorphism
          \[
          \psi\colon A_\CC\to A_\CC
          \]
          of $(\lieg,K)$-modules, which by Schur's Lemma is unique up to a scalar in $\CC^\times$. To verify the $\chi$-twisted linearity, note that $\psi$ is $(\lieg_{\CC},K_\infty)$-linear and also $(\lieg_\CC,K(\RR)^0)$-linear. Since the two $\pi_\circ^G$-module structures on $A_{\QQ(A)}$ differ by a $\chi^{\rm const}$-twist, we get for any $a\in A_\CC$ and any $\varepsilon\in L(\RR)$, taking the various actions into account:
          \begin{align*}
            \chi(\varepsilon)\cdot\psi(\varepsilon a)
            &=
            \chi^{\rm const}(\varepsilon)\cdot\psi(\varepsilon a)\\
            &=
            \psi\left(\Delta(\varepsilon)a\right)\\
            &=
            \Delta(\varepsilon)\psi(a)\\
            &=\varepsilon\cdot\psi(a).
          \end{align*}
          This shows the claimed $\chi$-twisted linearity.

          The space of $\chi$-twisted $(\lieg,K)$-linear maps $A_{\QQ(A)}\to A_{\QQ(A)}$ is isomorphic to $\QQ(A)$, again by Schur's Lemma, and after tensoring with $\CC$, gives us the space of $\chi$-twisted $(\lieg,K)$-linear maps $A_\CC\to A_\CC$. Therefore we find a constant $c\in\CC^\times$ with
          \[
          c\cdot\iota_{\QQ(A)}(A_{\QQ(A)})=\iota_{\QQ(A)}'(A_{\QQ(A)}).
          \]
          This proves the claimed existence of $c$ above. Uniqueness modulo $\QQ(A)^\times$ is clear as well.

          If we consider $A_{\QQ(A)}$ with its canonical $(\lieg_{\QQ(A)},K_{\QQ(A)})$-module structure which is inherited from the locally algebraic module structure by forgetting the action of $\pi_\circ^G$ on $A$ and $A'$, then both yield by Definition \ref{def:locallyalgebraicstandardmodule} the same $(\lieg_{\QQ(A)},K_{\QQ(A)})$- module structure.
          
          By our choice of $c$, the composition
          \begin{equation}
            \widetilde{\psi}_{\QQ(A)}
            :=\iota_{\QQ(A)}^{-1}\circ \frac{1}{c}\cdot{\bf1}_{\pi_\infty}\circ\iota_{\QQ(A)}'\colon A_{\QQ(A)}\to A_{\QQ(A)}
            \label{eq:widetildepsiformula}
          \end{equation}
          is a $\chi$-twisted isomorphism of $(\lieg_{\QQ(A)},K_{\QQ(A)})$-modules as in \eqref{eq:psiQQVtilde}.

          Therefore Proposition \ref{prop:psidecompositions} is applicable and shows that over $\QQ(A)[\sqrt{\alpha}]$ we have the decomposition \eqref{eq:AEdecomposition} of $A_{\QQ(A)[\sqrt{\alpha}]}$ into eigenspaces of $\widetilde{\psi}_{\QQ(A)[\sqrt{\alpha}]}$.

          In order to compute $c$, we need more information about the relation between $\Lambda$ and $\Lambda'$. Note that by the respective invariance properties of $\Lambda$ and $\Lambda'$ the identity \eqref{eq:Lambdacompatibility} extends to the $(\lieg_\CC,\widetilde{K}(\RR))$-submodule
          \[
          \widetilde{\pi}:=\sum_{\varepsilon'\in\pi_0(\widetilde{K}_\infty)}\varepsilon'\pi_0,
          \]
          i.\,e.\ we have
          \begin{equation}
          \Lambda|_{\widetilde{\pi}}=
          \Lambda'|_{\widetilde{\pi}}.
          \label{eq:LambdaLambdaprime}
          \end{equation}
          Moreover, any $\varepsilon\in L(\RR)$ with $\chi(\varepsilon)=-1$ normalizes the kernel $\widetilde{K}(\RR)$ and we get
          \begin{equation}
            \pi_\infty=\widetilde{\pi}\oplus\varepsilon\widetilde{\pi},
            \label{eq:piinfinitydecomposition}
          \end{equation}
          as $(\lieg_\CC,\widetilde{K}(\RR))$-modules. Decomposing a given $\phi\in\pi_\infty$ accordingly into
          \[
          \phi=\phi_++\varepsilon\phi_-,
          \]
          with $\phi_\pm\in\widetilde{\pi}$, then by \eqref{eq:LambdaLambdaprime} we have the relations
          \begin{align}
            \Lambda (\phi)&=\Lambda(\phi_+)+\rho_{\pi_\circ^G}(\varepsilon)\Lambda(\phi_-),\label{eq:Lambda}\\
            \Lambda'(\phi)&=\Lambda(\phi_+)-\rho_{\pi_\circ^G}(\varepsilon)\Lambda(\phi_-).\label{eq:Lambdaprime}
          \end{align}

          Observe that by Theorem \ref{thm:AEdecomposition} the decomposition \eqref{eq:piinfinitydecomposition} is the eigenspace decomposition of $\widetilde{\psi}_\CC$, i.\,e.\ we may assume without loss of generality that
          \begin{align*}
          \widetilde{\pi}&=\iota_\CC\left(\EE_{\sqrt{\alpha}}\left(\widetilde{\psi}_{\CC}\right)\right)\\
          \varepsilon\widetilde{\pi}&=\iota_\CC\left(\EE_{-\sqrt{\alpha}}\left(\widetilde{\psi}_{\CC}\right)\right),
          \end{align*}
          and the image under $\iota_\CC'$ of the respective eigenspaces is the same as above.

          For any $t\in A_{\QQ(A)}$ we decompose $t=t_++t_-$ with
          \[
          t_\pm\in \EE_{\pm\sqrt{\alpha}}\left(\widetilde{\psi}_{{\QQ(A)[\sqrt{\alpha}]}}\right)\subseteq A_{{\QQ(A)[\sqrt{\alpha}]}}.
          \]
          We conclude
          \begin{equation}
            \widetilde{\psi}_{\QQ(A)[\sqrt{\alpha}]}(t)=\sqrt{\alpha}t_+-\sqrt{\alpha}(t_-).
            \label{eq:psioft}
          \end{equation}
          In particular, since this number is $\QQ(A)$-rational by the $\QQ(A)$-rationality of $t$ and $\widetilde{\psi}$, we get
          \begin{equation}
            t_+-t_-\in\sqrt{\alpha}\QQ(A).
            \label{eq:tplusminusdifference}
          \end{equation}
          
          We remark that the action of $\varepsilon\in L(\RR)$ with $\chi(\varepsilon)=-1$ and therefore also the action of $\Delta(\varepsilon)$ intertwines the two eigenspaces $\EE_{\pm\sqrt{\alpha}}\left(\widetilde{\psi}_{{\CC}}\right)$ as $(\lieg_{\CC},\widetilde{K}(\RR))$-modules.

          We compute $\Lambda\left(\iota_{\QQ(A)}'(t)\right)$ in two different ways:
          \begin{align*}
            \Lambda\left(\iota_{\QQ(A)}'(t)\right) &=
            \Lambda\left(c\cdot\iota_{\QQ(A)}\left(\widetilde{\psi}_{\QQ(A)}(t)\right)\right)
            &\text{(cf.\ \eqref{eq:widetildepsiformula})}\\
            &=
            c\cdot\Lambda\left(\iota_{\QQ(A)}(\sqrt{\alpha}(t_+-t_-)\right)
            &\text{(cf.\ \eqref{eq:psioft})}\\
            &=
            c\cdot\nu_{\QQ(A)}\circ\lambda_{\QQ(A)}(\sqrt{\alpha}(t_+-t_-))
            &\text{(cf.\ \eqref{eq:Lambdacommutativity})}\\
            &\in
            c\cdot\QQ(A)\subseteq\CC
            &\text{(cf.\ \eqref{eq:tplusminusdifference})}
          \end{align*}
          Exploiting the relation between $\Lambda$ and $\Lambda'$ we obtain
          \begin{align*}
            \Lambda\left(\iota_{\QQ(A)}'(t)\right) &=
            \Lambda\left(\iota_{\QQ(A)[\sqrt{\alpha}]}'(t_++t_-)\right)\\
            &=
            \Lambda\left(\iota_{\QQ(A)[\sqrt{\alpha}]}'(t_+)\right)+\rho_{\pi_\circ^G}(\varepsilon)\Lambda\left(\varepsilon\iota_{\QQ(A)[\sqrt{\alpha}]}'(t_-)\right)
            &\text{(cf.\ \eqref{eq:Lambda})}\\
            &=
            \Lambda\left(\iota_{\QQ(A)[\sqrt{\alpha}]}'(t_+)\right)-\rho_{\pi_\circ^G}(\varepsilon)\Lambda\left(\varepsilon\iota_{\QQ(A)[\sqrt{\alpha}]}'(-t_-)\right)\\
            &=
            \Lambda'\left(\iota_{\QQ(A)[\sqrt{\alpha}]}'(t_+-t_-)\right)
            &\text{(cf.\ \eqref{eq:Lambdaprime})}\\
            &=
            \nu_{\QQ(A)[\sqrt{\alpha}]}'\circ\lambda_{\QQ(A)[\sqrt{\alpha}]}(t_+-t_-)
            &\text{(cf.\ \eqref{eq:Lambdacommutativity})}\\
            &\in
            \sqrt{\alpha}\cdot\QQ(A)\subseteq\CC.
            &\text{(cf.\ \eqref{eq:tplusminusdifference})}
          \end{align*}
          Since this relation holds for arbitrary $t\in A_{\QQ(A)}$ and $\Lambda$ is non-zero, we conclude that $c=\alpha$ as claimed.
        \end{proof}

        \begin{remark}
          The condition \eqref{eq:Lambdacompatibility} that $\Lambda$ and $\Lambda'$ agree on an irreducible $(\lieg_\CC,K_\infty^0)$-submodule is automatically satisfied if $\Lambda$ and $\Lambda'$ are defined via archimedean $\zeta$-integrals which are given by integration over (a quotient of) $H(\RR)$ and whose integrands differ by a twist by a finite order character $\chi$ considered as a character of $H(\RR)$.
        \end{remark}

        \subsubsection{$1/N$-integral structures}

        With the notation of the previous section, assume that $A$ admits an $\OO[1/N]$-integral model $A_{\OO[1/N]}$ for the ring of algebraic integers $\OO\subseteq \QQ(A)$. By appropriate normalization we obtain from $\psi_{\QQ(A)}$ an isomorphism
        \[
        \psi_{\OO[1/N]}\colon A_{\OO[1/N]}\otimes\chi\to A_{\OO[1/N]}
        \]
        of $(\lieg_{\OO},K_{\OO})$-modules. In particular we observe that the characteristic polynomial \eqref{eq:psimp} of $\widetilde{\psi}_{\OO}$ lies in $\OO[X]$, since $\alpha\in\OO^\times$ by the normalization of $\psi_{\QQ(A)}$.

        The integral analogue of Theorem \ref{thm:rationalperiodrelations} is

        \begin{theorem}\label{thm:integralperiodrelations}
          Let $[(A,V)]_\cong\in{\rm t}\in\cohTypes_{\QQ(A)}(G)$ be a representative of a cohomological type, assumed to be defined over the field of definition $\QQ(A)$ of $A$. Assume that $A_\CC$ decomposes as $(\lieg_\CC,K_\infty^0)$-module multiplicity-free. Fix an arbitrary admissible weight $V'\not\cong V$ of $A$.

          Assume that $A$ admits an $\OO[1/N]$-integral model $A_{\OO[1/N]}$ for the ring of integers $\OO\subseteq \QQ(A)$ and $2\mid N$.
          
          Given two non-zero $\QQ(A)$-rational $(\lieh,L\times\pi_0^{\rm const}(L))$-linear functionals
          \[
          \lambda \colon A \to E ,\quad
          \lambda'\colon A'\to E',
          \]
          where $E=(\QQ(A),\rho_H,\rho_L)$ and $E'=(\QQ(A),\rho_H,\rho_L')$ with
          \[
          \rho_{\pi_\circ^G}'=\rho_{\pi_\circ^G}\otimes\chi^{\rm const}
          \quad\text{and}\quad
          \rho_L'=\rho_L\otimes\chi^{\rm const}|_{\pi_0^{\rm const}(L)}.
          \]
          Assume that $\chi|_{L}\neq{\bf1}$.
          
          Let $\pi_\infty$ denote a $(\lieg_\CC,K_\infty)$-module isomorphic to $A_\CC$ and $A_\CC'$ via $\CC$-base change of fixed $\QQ(A)$-linear embeddings
          \[
          \iota \colon A \to\pi_\infty,\quad
          \iota'\colon A'\to\pi_\infty.
          \]
          Both maps are understood as $(\lieg_{\QQ(A)},K_{\QQ(A)}\times\pi_\circ^G)$-linear maps on the underlying vector spaces for suitable locally algebraic module structures on $\pi_\infty$.
          
          Assume that we are given two complex functionals
          \[
          \Lambda \colon\pi_\infty\to\CC,\quad
          \Lambda'\colon\pi_\infty\to\CC,
          \]
          with the property that for some irreducible $(\lieg_\CC,K_\infty^0)$-submodule $\pi_0\subseteq\pi_\infty$
          \begin{equation}
            \Lambda|_{\pi_0}=\Lambda'|_{\pi_0}.
          \end{equation}
          and that we have commutative diagrams
          \begin{equation}
          \begin{CD}
            A @>{\lambda}>> E @.\hspace*{4em} @. A' @>{\lambda'}>> E'\\
            @V{\iota}VV @VV{\nu}V @. @V{\iota'}VV @VV{\nu'}V \\
            \pi_\infty @>>{\Lambda}> \CC @. @. \pi_\infty @>>{\Lambda'}> \CC
          \end{CD}
          \end{equation}
          with the property that the images of the right vertical maps agree:
          \[
          \lambda_{\QQ(A)}\circ\nu_{\QQ(A)}(\OO[1/N])=
          \lambda_{\QQ(A)}'\circ\nu_{\QQ(A)}'(\OO[1/N]).
          \]
          Remark that the right vertical arrows are both $(\lieh_{\QQ(A)},L_{\QQ(A)}\times\pi_\circ^H)$-linear with the corresponding locally algebraic structures on $\CC$ for the corresponding $(\lieh_\CC,L_\infty)$-module structures to make $\Lambda$ and $\Lambda'$ both $(\lieh_\CC,H_\infty)$-linear. Then with $\alpha\in\OO[1/N]^\times$ from \eqref{eq:psimp} we have the period relation
          \begin{equation}
            \sqrt{\alpha}\cdot\iota(A_{\OO[1/N]})=\iota'(A_{\OO[1/N]}').
          \end{equation}
        \end{theorem}

        \begin{proof}
          Follows mutatis mutandis the proof of Theorem \ref{thm:integralperiodrelations}.
        \end{proof}

        \begin{remark}
          The parity condition on $N$ stems from the fact that the decomposition \eqref{eq:AEdecomposition} may introduce denominators at $2$ due to congruences between the two direct summands. This is also reflected in the inseparability of the minimal polynomial \eqref{eq:psimp}.
        \end{remark}

        \subsubsection{Applications to $L$-functions of Hilbert modular forms}\label{sec:hilbertmodularforms}

        We fix a totally real number field $F/\QQ$ and put $G=\res_{F/\QQ}\GL(2)$ and $K=\res_{F/\QQ}\Oo(2)$, as well as $B=\res_{F/\QQ}B_2$ where $B_2\subseteq\GL(2)$ denotes the standard upper triangular Borel subgroup. We write $T=\res_{F/\QQ}T_2$ for the diagonal torus in $B$. Inside the restriction of scalars $Z=\res_{F/\QQ}Z_2\subseteq T$ of the center we find the maximal central $\QQ$-split torus $S\cong\GL(1)$. The determinant, composed with the norm $N_{F/\QQ}\colon F^\times\to\QQ^\times$, considered as a morphism $H:=\res_{F/\QQ}\GL(1)\to\GL(1)$, induces an isogeny $S\to\GL(1)$. We have
        \[
        H\otimes_\QQ\overline{\CC}=\prod_{\sigma\colon F\to\CC}\GL(1)
        \]
        and if $x=(x_\sigma)_{\sigma}\in H(\CC)$, then
        \[
        N_{F/\QQ}(x)=\prod_{\sigma\colon F\to\CC}x_\sigma.
        \]
        Likewise, $T$ splits over $\CC$ into $[F:\QQ]$ copies of $T_2$, indexed by the embeddings $\sigma\colon F\to\CC$. A $B$-dominant weight $\lambda$ of $T$ over $\CC$ then corresponds to a collection $\lambda=(\lambda_\sigma)_\sigma$ of $B_2$-dominant weights $\lambda_\sigma$ indexed by embeddings $\sigma\colon F\to\CC$. If
        \[
        \lambda_\sigma\;=\;\left(\lambda_{\sigma,1},\,\lambda_{\sigma,2}\right)
        \]
        with $\lambda_{\sigma,i}=\pi_i^{\lambda_{\sigma,i}}$ and $\pi_i$ denoting the $i$-th projection on the diagonal realization $T_2=\GL(1)\times\GL(1)$, then $\lambda$ is dominant if and only if for all $\sigma\colon F\to\CC$:
        \[
        \lambda_{\sigma,1}\geq\lambda_{\sigma,2}.
        \]
        Recall the subgroup ${}^0G\subseteq G$ from section \ref{sec:cohomologicaltypes}. The essential unitarizability condition imposed on Casselman--Wallach representations of $G(\RR)$ has the following explicit interpretation in terms of $\lambda$.

        Consider the Cartan-involution $\theta$ associated to $K(\RR)$. Over $\CC$ (or over $\RR$), $\theta$ may be realized as $[F:\QQ]$ copies of the Cartan involution $\theta_2\colon g\mapsto g^{-\rm{t}}$ on $\GL_2(\RR)$, again indexed by the embeddings $\sigma\colon F\to\CC$.

        We already classified the tempered cohomological types of $G$ in Theorem \ref{thm:cohomologicaltypesofGLn}, in particular we only need to consider the case of $(A_{\lieq}(\widetilde{\lambda}),V)$ for a $\theta$-stable Borel subalgebra $\lieq\subseteq\lieg_\CC$.

        As $\theta$-stable Borel subalgebra $\lieq$ inside the complex Lie algebra $\lieg_\CC$ of $G$ we may therefore choose
        \[
        \lieq=\prod_{\sigma\colon F\to\CC}\lieq_2,
        \]
        where $\lieq_2$ is a $\theta_2$-stable Borel subalgebra in $\CC^{2\times 2}$ containing $\lieh_2:=\liez_2\oplus\lieso_2$ as a Levi factor. The complex Lie algebra $\liet_2$ of the diagonal torus $T_2\subseteq\GL(2)$ and $\lieh_2$ are conjugate. We may choose this conjugation in a way that it induces the isomorphism
        \[
        \lieh_2\to\liet_2,\quad\begin{pmatrix}c&-d\\d&c\end{pmatrix}\,\mapsto\,\begin{pmatrix}c+id&0\\0&c-id\end{pmatrix}.
        \]
        Then a dominant weight $(\lambda_1,\lambda_2)$ for $T$ induces the weight
        \[
        \liet_2\to\CC,\quad (c,d)\;\mapsto\;(\lambda_1 + \lambda_2)\cdot c \,+\, i(\lambda_1 - \lambda_2)\cdot d,
        \]
        which we may assume dominant for $\lieq_2$, i.\,e.\ we may choose $\lieq_2$ accordingly. Now this weight is invariant under the Cartan convolution $\theta_2$ if and only if
        \begin{equation}
          \lambda_1+\lambda_2\;=\;0.
          \label{eq:individualcartaninvariance}
        \end{equation}
        Applying \eqref{eq:individualcartaninvariance} to our given weight $\lambda=(\lambda_\sigma)_\sigma$ restricted to $\liet_\CC\cap{}^0\lieg_\CC$, shows that the unitarizability condition \ref{eq:lambdacartan} from Remark \ref{rmk:unitarizability} amounts in our setting to the condition
        \begin{equation}
          \exists w\in\ZZ\colon\forall\sigma\colon F\to\CC\colon\quad
          \lambda_{\sigma,1}+\lambda_{\sigma,2}\;=\;w.
          \label{eq:gl2cartaninvariance}
        \end{equation}
        In representation theoretic terms this means that the corresponding rational representation $(V,\rho_G)$ of $G$ highest weight $\lambda$ is essentially self-dual over $\QQ$ in the sense that it is self-dual up to a twist by a power $N_{F/\QQ}^{\otimes w}$ of the norm: $V\cong V^\vee\otimes N_{F/\QQ}^{\otimes w}$. Here
        \[
        N_{F/\QQ}\colon G\to\GL(1)\quad a\mapsto N_{F/\QQ}(\det(a))
        \]
        denotes the $\QQ$-rational character of $G$ induced by the norm.

        Condition \eqref{eq:centralcondition} in Remark \ref{rmk:nontriviality} translates to
        \begin{equation}
          \forall\sigma\colon F\to\CC\colon\quad
          \lambda_{\sigma,1}(-1)=1,
          \label{eq:lambdaparity}
        \end{equation}
        if $-{\bf1}_2\in K_f$.

        We may translate the above observations into the language of Hilbert modular forms as usual. Let $V$ denote an absolutely irreducible locally algebraic representation of highest weight $\lambda$, defined over a finite extension $E/\QQ$. Assume that $\lambda$ satiesfies conditions \eqref{eq:gl2cartaninvariance} and \eqref{eq:lambdaparity}. For any embedding $\sigma\colon F\to\CC$ put
        \[
        k_\sigma:=\lambda_{\sigma,1}-\lambda_{\sigma,2}+2.
        \]
        Note that condition \eqref{eq:gl2cartaninvariance} implies that for all $\sigma,\tau\colon F\to\CC$:
        \[
        k_\sigma=2\lambda_{\sigma,1}-w+2\equiv 2\lambda_{\tau,1}-w+2=k_\tau\pmod{2},
        \]
        i.\,e.
        \begin{equation}
          \forall\sigma,\tau\in\Hom(F,\CC)\colon\quad
          k_\sigma\equiv k_\tau\pmod{2}.
          \label{eq:weightregularitycondition}
        \end{equation}
        Then a cuspidal automorphic representation $\pi$ of $G$ with $\cohType(\pi)\in\cohTypes_E(G,V)$ corresponds to a primitive Hilbert modular cusp form ${\bf f}$ of weight $k=(k_\sigma)_\sigma$ satisfying \eqref{eq:weightregularitycondition}. As explained in \cite{raghuramtanabe2011}, any primitive Hilbert modular cusp form ${\bf f}$ whose weight satisfies \eqref{eq:weightregularitycondition} and additionally $k_\sigma\geq 2$ for all $\sigma\colon F\to\CC$ corresponds to a regular algebraic cuspidal automorphic representation $\pi$ of $G$ as above.
        
        Condition \eqref{eq:gl2cartaninvariance} implies that $\pi\otimes|\cdot|_{\Adeles}^{w/2}$ is unitary. This characterizes $w\in\ZZ$. If $-{\bf1}_2\in K_f$ and $\pi^{K_f}\neq 0$, then $k$ is even in the sense that $2\mid k_\sigma$ for all $\sigma\colon F\to\CC$ by \eqref{eq:lambdaparity}. 

        Write $L(s,\pi)$ for the standard $L$-function associated to $\pi$. Then a half-integer $s=\frac{1}{2}+j$ is {\em critical} for $L(s,\pi)$ in the sense of \cite{deligne1979} if and only if
        \begin{equation}
          \forall\sigma\in\Hom(F,\CC)\colon\quad
          \lambda_{\sigma,2}\leq j\leq \lambda_{\sigma,1}.
          \label{eq:criticalityforgl2}
        \end{equation}
        We refer to \cite[Proposition 3.18]{raghuramtanabe2011} for a direct proof in our setting. For given $j\in\ZZ$, condition \eqref{eq:criticalityforgl2} is equivalent to
        \begin{equation}
          \Hom_{G_1}(V,N_{F/\QQ}^{\otimes j})\neq 0,
          \label{eq:algebraiccriticalityforgl2}
        \end{equation}
        in the category of rational representations of $G$, where
        \[
        G_1=\res_{F/\QQ}\GL_1=\left\{\begin{pmatrix}\ast & \\ & 1\end{pmatrix}\right\}\subseteq G.
        \]
        Remark that the $\Hom$-spaces in \eqref{eq:criticalityforgl2} and \eqref{eq:algebraiccriticalityforgl2} are each one-dimensional if non-zero.

        Write $\pi_\infty^{(K_\infty)}$ for the underlying $(\lieg_\CC,K_\infty)$-module of the archimedean component $\pi_\infty$. Then any representative of the cohomological type $\cohType_\CC(\pi)$ over $\CC$ is of the form $(A_\CC,V)$ with $V=(V,\rho_G,\rho_{\pi_0^G})$ where $(V,\rho_G)$ is the rational representation of $G$ of heighest weight $\lambda$ over $\CC$ and the complex realization $A_\CC$ of $A$ in the sense of Definition \ref{def:locallyalgebraiccomplexmodule} is isomorphic to the $(\lieg_\CC,K_\infty)$-module $\pi_\infty^{(K_\infty)}$.        

        Condition \eqref{eq:algebraiccriticalityforgl2} may be rephrased in the category of locally algebraic representations as follows. We may augment $N_{F/\QQ}^{\otimes j}$ to a locally algebraic representation $N_{F/\QQ}^{\otimes j,\varepsilon}=(\CC,N_{F/\QQ}^{\otimes j},\varepsilon)$ where $\varepsilon$ is a one-dimensional representation of $\pi_\circ^G$. Then choosing $\varepsilon=\rho_{\pi_0^G}^\vee|_{\pi_0^{G_1}}=\rho_{\pi_0^G}|_{\pi_0^{G_1}}$, we get the equivalent condition
        \begin{equation}
          \Hom_{G_1\times\pi_\circ^{G_1}}(V,\,N_{F/\QQ}^{\otimes j,\varepsilon})\neq 0,
          \label{eq:locallyalgebraiccriticalityforgl2}
        \end{equation}
        in the category of locally algebraic representations of $G_1$. Note that $G_1(\RR)=(\RR\otimes_\QQ F)^\times\cong\prod\limits_{\sigma}\RR^\times$ and hence $\pi_\circ^{G_1}\cong\pi_\circ^G$ via the given inclusion $G_1\to G$.

        We may think of the isomorphism of $A_\CC$ to $\pi_\infty^{(K_\infty)}$ to depend on $\varepsilon$, i.\,e.\ we write
        \[
        \iota_\varepsilon\colon A_\CC\to\pi_\infty^{(K_\infty)}
        \]
        for a fixed choice of isomorphism.

        Passing to complex realizations, we obtain on the level of $(\lieg_\CC,K_\infty)$-modules for $d=[F:\QQ]$ a canonical map
        \begin{align*}
          &\left(\bigwedge^d\lieg_{1,\CC}\right)\otimes
          \left(N_{F/\QQ,\CC}^{\otimes j,\varepsilon}\right)^\vee\otimes
          \Hom_{G_1(\RR)}(V_\CC,\,N_{F/\QQ,\CC}^{\otimes j,\varepsilon})\otimes
          H^d(\lieg_\CC,K_\infty\cdot S(\RR)^0;\,A_\CC\otimes V_\CC)\\
          &\to
          H^0(\lieg_{1,\CC},\{\pm1\}^{\Hom(F,\CC)}; \CC)=\CC,\\
          &\omega\otimes\mu\otimes\eta\otimes(\varpi\otimes\varphi\otimes v)\;\mapsto\;
          \varpi(\omega)\cdot\mu\left(\!\!\int_{G_1(\RR)}\iota_\varepsilon(\varphi)(h)\cdot (h\eta(v))dh\right),
        \end{align*}
        only depending on the choice of a Haar measure $dh$ on $G_1(\RR)$, unique up to scalars in $\RR_{>0}$, and the choice of $\iota_\varepsilon$, unique up to scalars in $\CC^\times$. Here $h\in G_1(\RR)$ acts on the vector $\eta(v)$ inside the complex realization of $N_{F/\QQ,\CC}^{\otimes j,\varepsilon}$ via the restriction of the corresponding action of $G(\RR)$ to $G_1(\RR)$, i.\,e.\ via the algebraic norm $N_{F/\QQ}^{\otimes j}$ twisted by $\varepsilon$.

        We remark that in general, a cohomology class in the domain may be represented by a finite sum of tensors of the form $\varpi\otimes\varphi\otimes v$. We refer to equation (3.9) in \cite{raghuramtanabe2011} for an explicit expression in the case of $\varepsilon={\bf1}$.

        The identity
        \[
        \int_{G_1(\RR)}\iota_\varepsilon(\varphi)(h)\cdot (h\eta(v))dh\;=\;
        \int_{G_1(\RR)}\iota_\varepsilon(\varphi)(h)\cdot N_{K/\QQ,\CC}^{\otimes j,\varepsilon}dh\cdot\eta(v)
        \]
        shows that the values of the above map are given by the evaluation of archimedean local $\zeta$-integrals for $\GL(2)/F$. The quasi-character $N_{K/\QQ,\CC}^{\otimes j,\varepsilon}$ admits the following explicit description:
        \begin{equation}
          N_{K/\QQ,\CC}^{\otimes j,\varepsilon}=|\cdot|_\infty^j\otimes \sgn_\infty^j\otimes\varepsilon,
        \end{equation}
        where $|\,\cdot\,|_\infty=|N_{F/\QQ}(-)|=\prod_{v\mid\infty}|\,\cdot\,|_v$ is the product of all extensions of the archimedean absolute value on $\QQ$ to $F$, canonically extended to $\RR\otimes_\QQ F=\prod\limits_{v\mid\infty}F_v$ and likewise $\sgn_\infty=\otimes_{v\mid\infty}\sgn_v$ where
        \[
        \sgn_v\colon\; F_v^\times\to\{\pm1\},\;x_v\mapsto \frac{x_v}{|x_v|_v},
        \]
        cf.\ Example \ref{ex:GLnoverQ}.

        Globally, the above map has the following sheaf theoretic interpretation:
        \[
        H_d^{\rm BM}({\mathscr X}_{G_1}(L);\CC)\otimes
        \left(N_{F/\QQ}^{\otimes j,\varepsilon}\right)^\vee\otimes
        \Hom_{G_1(\RR)\times\pi_\circ^G}(V,\,N_{F/\QQ}^{\otimes j,\varepsilon})\otimes
        H^d_{\rm c}({\mathscr X}_G(K);\,\widetilde{V})\to
        H^0({\mathscr X}_{G_1}(L);\CC),
        \]
        \[
        a\otimes\mu\otimes\eta\otimes c\;\mapsto\;\mu\circ\eta(c|_{{\mathscr X}_{G_1}(L)})\cap a
        \]
        composed with
        \[
        H^0({\mathscr X}_{G_1}(L);\CC)=\CC^{\pi_0({\mathscr X}_{G_1}(L))}\to\CC,\;(x_C)_{C\in \pi_0({\mathscr X}_{G_1}(L))}\mapsto \sum_{C}x_C
        \]
        realizes a canonical global map whose archimedean component, when interpreted in terms of automorphic forms, agrees with the above map on $(\lieg_\CC,K_\infty)$-cohomology.

        Globally, this sheaf theoretic map has an analogous representation via the global $\zeta$-integral
        \begin{equation}
          \int_{G_1(\QQ)\backslash G_1(\Adeles)}\Phi_\varepsilon(h)|h|_{\Adeles}^jdh.
          \label{eq:globalzetaintegralforgl2}
        \end{equation}
        Here $\Phi_\varepsilon\in\pi^{(K_\infty)}$ is (essentially) the automorphic cusp form representing a cohomology class
        \begin{equation}
          c\in H^d_{\rm c}({\mathscr X}_G(K);\,\widetilde{V})[\pi_f^K].
          \label{eq:hilbertmodularcohomology}
        \end{equation}
        We emphasize that the right hand side in \eqref{eq:hilbertmodularcohomology} is a free Hecke module of rank $1$. By invariance properties of integration, for given $j$, the above map on sheaf cohomology is non-trivial on the $\pi_f^K$-isotypic component for at most one choice of $\varepsilon$. Therefore we obtain only one global $\zeta$-integral \eqref{eq:globalzetaintegralforgl2}.

        An elementary computation shows that if we write $\varepsilon=\otimes_{\sigma}\varepsilon_\sigma$, then the necessary parity condition is
        \begin{equation}
          \varepsilon_\sigma(-1)\;=\;(-1)^{w+j}\omega_{\pi,v}(-1),
          \label{eq:epsilonparitycondition}
        \end{equation}
        where $\omega_{\pi}\colon G_1(\QQ)\backslash G_1(\Adeles)\to\CC^\times$ denotes the central character of $\pi$ and $v\mid\infty$ is the archimedean place corresponding to $\sigma\colon F\to\CC$.

        That being said, we obtain for $\pi$ for each element $\alpha\in[(A,V)]_\cong$ in the cohomological type $\cohType(\pi)$ a well defined Whittaker period $\Omega_{\alpha}\in\CC^\times/\QQ(\pi)^\times$ which arises from the comparison of the canonical rational structure
        \[
        H^d_{\rm c}({\mathscr X}_G(K);\,\widetilde{V}_{\QQ(\pi)})[\pi_f^K]\subseteq H^d_{\rm c}({\mathscr X}_G(K);\,\widetilde{V})[\pi_f^K],
        \]
        with the canonical rational structure on the (non-archimedean) Whittaker model $\mathcal W(\pi_f,\psi_f)$ combined with a suitable analytic normalization of a $\QQ(\pi)$-rational structure on the archimedean Whittaker model $\mathcal W(\pi_\infty,\psi_\infty)$, i.\,e.\ by enforcing the existence of a $\QQ(\pi)$-rational good test vector at $\infty$.

        Morally, the latter global rational structure on $\pi^{(K_\infty)}$ guarantees the existence of a $\QQ(\pi)$-rational vector $\Phi^{\rm analytic}\in\pi^{(K_\infty)}$ with the property that for all $s\in\CC$:
        \[
        \int_{G_1(\QQ)\backslash G_1(\Adeles)}\Phi^{\rm analytic}(h)|h|_{\Adeles}^{\frac{1}{2}+s}dh\;=\;L(s,\pi),
        \]
        where $L(s,\pi)$ denotes the {\em completed} $L$-function, including $\Gamma$-factors.
        
        The topological rational normalization ensures the rationality of the topological integration procedure, i.\,e.\ it enforces rationality of special $\zeta$-values normalized by the corresponding periods. Therefore, the difference $\Omega_\alpha\in\CC^\times/\QQ(\pi)^\times$ between these rational structures captures the transcendental part of the corresponding special $L$-values. Now for given $\pi$, at a given critical value $s=\frac{1}{2}+j$, there is a unique choice of $\varepsilon$ with \eqref{eq:epsilonparitycondition} and hence a unique corresponding $\alpha\in\cohType(\pi)$ with
        \[
        \frac{L(s,\pi)}{\Omega_\alpha}\in\QQ(\pi).
        \]
        Moving from one critical value $s=\frac{1}{2}+j$ to $\frac{1}{2}+j+1$ switches the parity in \eqref{eq:epsilonparitycondition} for all $\sigma$ simultaneously. Therefore, given $\pi$, we have precisely two periods $\Omega_\alpha$ and $\Omega_{\beta}$ for suitable $\alpha,\beta\in\cohType(\pi)$, explaining the rationality of all special values of the standard $L$-function of $\pi$.

        Twisting $\pi$ with a finite order Hecke character $\chi\colon G_1(\QQ)\backslash G_1(\Adeles)\to\CC^\times$ has the fowing effect: The infinity type $\chi_\infty$ of $\chi$ affects the corresponding the parity condition \eqref{eq:epsilonparitycondition} for $\pi\otimes\chi$ in the sense that the new condition reads for any embedding $\sigma\colon F\to\RR$
        \begin{equation}
          \varepsilon_\sigma(-1)\;=\;(-1)^{w+j}\omega_{\pi,v}(-1)\chi_v(-1).
          \label{eq:twistedepsilonparitycondition}
        \end{equation}
        Therefore, we obtain another pair $\alpha_\chi,\beta_\chi\in\cohType(\pi)=\cohType(\pi_\infty)$ which is relevant for the algebraicity of the special values of the twisted $L$-function $L(s,\pi\otimes\chi)$. Now the effect on the rational structures on the non-archimedean Whittaker model $\mathcal W(\pi_f,\psi_f)$ combined with the corresponding local $\zeta$-integrals, together with the analogous rationality consideration on the archimedean Whittaker model, which corresponds to our Theorem \ref{thm:rationalperiodrelations} in the rational case (resp.\ Theorem \ref{thm:integralperiodrelations} in the integral case), allows us to compute, how the periods $\Omega_\alpha$ for $\pi$ change when passing to $\pi\otimes\chi$. This then allows us to show that the collection of the periods $\Omega_\gamma$ for any $\gamma\in\cohType(\pi)$ explain the rationality of {\em all} twisted $L$-values $L(s,\chi)$ for {\em all} finite order Hecke characters $\chi$.
        
        \subsubsection{Applications to automorphic $L$-functions}\label{sec:applicationstoautomorphicL}

        Generalizing our discussion in the Hilbert modular case, the following explicit cases may be studied via Theorem \ref{thm:rationalperiodrelations} and Theorem \ref{thm:integralperiodrelations}:
          \begin{itemize}
          \item[(i)] The Rankin--Selberg $\zeta$-integrals for the pair $(G,H)=(\res_{F/\QQ}\GL(n+1)\times\GL(n),\GL(n))$ of Jacquet, Shalika and Piatetski-Shapiro as studied in \cite{jacquetshalika1981,jacquetshalika1981.2,jpss1983,schmidt1993,kazhdanmazurschmidt2000,kastenschmidt2008,raghuramshahidi2008,raghuram2010,jacquet2009,januszewski2011,januszewski2014,januszewski2015,raghuram2016,sun2017,januszewskiperiods1,januszewskiajm,haranamikawa}.
          \item[(ii)] The Shalika $\zeta$-integrals for the pair $(G,H)=(\res_{F/\QQ}\GL(2n),\GL(n)\times\GL(n))$ of Friedberg--Jacquet as studied in \cite{friedbergjacquet1993,grobnerraghuram2014,chensun2017,sun2019,jiangetal2019,djr2020,januszewskiperiods2}.
          \end{itemize}
          
          We remark that in the first case the rationality of the archimedean $\zeta$-integral is only known for $n=1$ and $n=2$ \cite{januszewskiperiods1}. Theorem \ref{thm:rationalperiodrelations} provides a corrected proof of the main result in \cite{januszewskiperiods1} taking into account that $\pi_\circ(K)$ in general is a non-constant \'etale group scheme.

          In the second case the rationality of the archimedean $\zeta$-integral is known outside a non-explicit countable set of complex arguments $s\in\CC$ \cite{januszewskiperiods2}. The cohomologically induced functional constructed by Sun in Section 2 of \cite{sun2019} is shown to be algebraic in \cite{januszewskiperiods2}. Theorem \ref{thm:rationalperiodrelations} and Theorem \ref{thm:integralperiodrelations} are always applicable to such functionals. This applies in particular to case (ii).

          \begin{example}[{Hilbert modular forms and automorphic forms on $\GL(2)$}]
            Both (i) and (ii) contain the case of the standard $L$-function for $G=\res_F\GL(2)$ as a special case. From that perspective, as explained in the previous section, Theorem \ref{thm:rationalperiodrelations} recasts the archimedean rationality considerations in \cite{manin1972,manin1976,shimura1976,shimura1977,shimura1978,hida1994,raghuramtanabe2011,namikawa2016} representation theoretic terms.
          \end{example}

          \begin{remark}
            Regardless whether rational or integral structures are preserved by certain $\zeta$-integrals, in both the Rankin--Selberg and the Shalika situation the abstract rationality pattern deduced from Theorems \ref{thm:rationalperiodrelations} and \ref{thm:integralperiodrelations} in combination with Theorem \ref{thm:glnrationalreducibility} is compatible with the one predicted by Deligne in \cite{deligne1979}.
          \end{remark}

	\bibliographystyle{plain}

\end{document}